\documentclass[11pt,reqno]{amsart}
\usepackage{graphicx,amsmath,amsthm,verbatim,amssymb,lineno,titletoc}
\usepackage{mathrsfs}
\usepackage[export]{adjustbox}
\usepackage{ytableau}
\usepackage{bbm}
\usepackage{pdflscape}
\usepackage{hyperref}
\hypersetup{colorlinks}
\usepackage[numbers]{natbib}
\usepackage{array}
\usepackage{tikz}
\usetikzlibrary{shapes.multipart,patterns,arrows}
\usepackage[top=1.1in, left=1.25in, right=1.25in, bottom=1.1in]{geometry}
\hypersetup{colorlinks,linkcolor={red},citecolor={olive},urlcolor={red}}

\usepackage[font=small]{caption}
\usepackage{subcaption}
\usepackage{amssymb}
\usepackage[T1]{fontenc} 
\usepackage{fourier} 
\usepackage[english]{babel} 
\usepackage{amsmath,amsfonts,amsthm,bbm} 
\usepackage{appendix}

\usepackage{amssymb}
\usepackage[all]{xy}
\usepackage{enumerate}
\usepackage{tikz}
\usepackage{mathrsfs}
\usepackage[english]{babel}
\usepackage{multirow}
\usepackage{float}
\usepackage{enumitem}
\usepackage{graphicx,accents}

\newcommand\dataset[1]{\textsc{\texttt{#1}}}

\newcommand{\edit}[1]{{\bf \textcolor{blue}
		{[#1\marginpar{\textcolor{red}{***}}]}}}
\newcommand{\HOX}[1]{\marginpar{\footnotesize #1}}
\newcolumntype{M}[1]{>{\centering\arraybackslash}m{#1}}
\newcolumntype{N}{@{}m{0pt}@{}}

\makeatletter

\let\c@table\c@figure
\makeatother

\makeatletter
\newcommand{\addresseshere}{%
	\enddoc@text\let\enddoc@text\relax
}
\makeatother

\def\one{\mathbf{1}}

\def\EE{\mathbb{E}}
\def\P{\mathbb{P}}
\def\eps{\varepsilon}
\def\diam{\textup{diam}}
\def\G{\mathcal{G}}

\def\x{\mathbf{x}}
\def\y{\mathbf{y}}
\def\z{\mathbf{z}}

\definecolor{hancolor}{rgb}{0.0 0.0, 1.0}
\newcommand{\commHL}[1]{{\textcolor{black}{#1}}} 

\newcommand{\facundo}[1]                {{ \textcolor{red} {#1}}}
\usepackage[normalem]{ulem}

\usepackage[normalem]{ulem}

\usepackage{theoremref}

\newtheorem{thm}{Theorem}[section]
\newtheorem{lemma}{Lemma}[section]
\newtheorem{prop}{Proposition}[section]
\newtheorem{cor}{Corollary}[section]
\newtheorem{problem}{Problem}[section]
\newtheorem{dfn}{Definition}[section]
\newtheorem{ex}{Example}[section]
\newtheorem{rmkk}{Remark}[section]

\usepackage[normalem]{ulem}

\begin{document}
	
	\title[Sampling random graph homomorphisms and applications to network data analysis]{Sampling random graph homomorphisms \\ and applications to network data analysis}

	\author{Hanbaek Lyu}\thanks{All codes are available at \url{https://github.com/HanbaekLyu/motif_sampling}}
	\address{Hanbaek Lyu, Department of Mathematics, University of California Los Angeles, CA 90095}
	\email{\texttt{colourgraph@gmail.com}}

	\author{Facundo M\'emoli}
	\address{Facundo Memoli, Department of Mathematics, The Ohio State University, Columbus, OH 43210}
	\email{\texttt{memoli@math.osu.edu}}

	\author{David Sivakoff}
	\address{David Sivakoff, Department of Statistics and Mathematics, The Ohio State University, Columbus, OH 43210}
	\email{\texttt{dsivakoff@stat.osu.edu}}
	
	\date{\today}
	

	\maketitle
	
	\begin{abstract}
		A graph homomorphism is a map between two graphs that preserves adjacency relations. We consider the problem of
		sampling a random graph homomorphism from a graph into a large network. We propose two complementary MCMC algorithms for sampling random graph homomorphisms and establish bounds on their mixing times and the concentration of their time averages. Based on our sampling algorithms, we propose a novel framework for network data analysis that circumvents some of the drawbacks in methods based on independent and neighborhood sampling. Various time averages of the MCMC trajectory give us various computable observables, including well-known ones such as homomorphism density and average clustering coefficient and their generalizations. Furthermore, we show that these network observables are stable with respect to a suitably renormalized cut distance between networks. We provide various examples and simulations demonstrating our framework through synthetic networks. We also \commHL{demonstrate the performance of} our framework on the tasks of network clustering and subgraph classification on the Facebook100 dataset and on Word Adjacency Networks of a set of classic novels. 
	\end{abstract}

	\section{Introduction}
	
	Over the past several decades, technological advances in data collection and extraction have fueled an explosion of network data from seemingly all corners of science -- from computer science to the information sciences, from biology and bioinformatics to physics, and from economics to sociology. These data sets come with a locally defined pairwise relationship, and the emerging and interdisciplinary field of Network Data Analysis aims at systematic methods to analyze such network data at a systems level, by combining various techniques \commHL{from probability, statistics, graph theory, geometry, and topology.}

	Sampling is an indispensable tool in the statistical analysis of large graphs and networks. Namely, we select a typical sample of the network and calculate its graph theoretical properties such as average degree, mean shortest path length, and expansion (see \cite{kolaczyk2014statistical} for a survey of statistical methods for network data analysis). One of the most fundamental sampling methods, which is called the \textit{independent sampling}, is to choose a fixed number of nodes independently at random according to some distribution on the nodes. One then studies the properties of the subgraph or subnetwork induced on the sample. Independent sampling is suitable for dense graphs, and closely connected to the class of network observables called the \textit{homomorphism density}, which were the central thread in the recent development of the theory of dense graph limits and graphons \citep{lovasz2006limits, lovasz2012large}.

	An alternative sampling procedure particularly suitable for sparse networks is called the \textit{neighborhood sampling} (or snowball sampling). Namely, one may pick a random node and sample its entire neighborhood up to some fixed radius, so that we are guaranteed to capture a connected local piece of the sparse network. We then ask what the given network looks like locally. For instance, the average clustering coefficient, first introduced in \cite{watts1998collective}, is a network observable that measures the extent to which a given network locally resembles complete graphs. Also, neighborhood sampling was used in 
	\cite{benjamini2001recurrence} to define the sampling distance between networks and to define the limit object of sequences of bounded degree networks.

	
	\commHL{Our primary concern in this work, roughly speaking, is to \textit{sample connected subgraphs from a possibly sparse network in a way such that certain minimal structure is always imposed.} A typical example is to sample $k$-node subgraphs with uniformly random Hamiltonian path, see Section \ref{section:sampling_hamiltonian}. }More generally, for a fixed `template graph' (motif) $F$ of $k$ nodes, we would like to sample $k$ nodes from the network $\G$ so that the induced subnetwork always contains a `copy' of $F$. This is equivalent to conditioning the independent sampling to contain a `homomorphic copy' of $F$. This conditioning enforces that we are always sampling some meaningful portion of the network, where the prescribed motif $F$ serves as a backbone. One can then study the properties of subnetworks of $\G$ induced on this random copy of $F$. \commHL{Clearly, neither  independent sampling nor  neighborhood sampling serve this purpose, as the former returns disconnected subgraphs with high probability (due to the sparsity of the network) and the latter has no control over the structure of the subgraphs  being sampled. We call this sampling scheme \textit{motif sampling} (see Figure \ref{fig:sampling}) and it should not be confused with sampling graphs from a random graph model. 
	}
	
	
	\begin{figure*}[h]
		\centering
		\includegraphics[width=0.7 \linewidth]{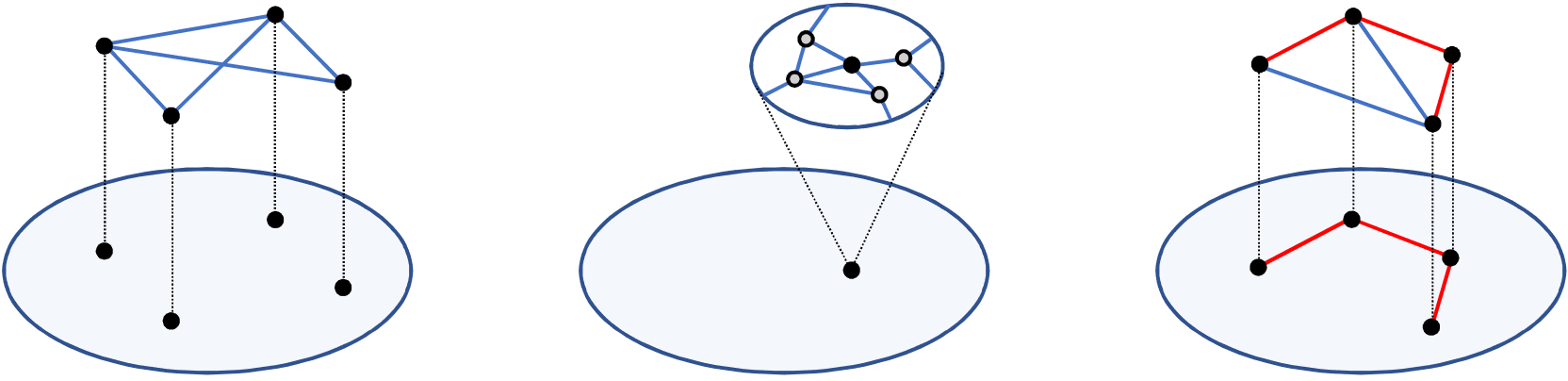}
		\caption{ Independent sampling (left), neighborhood sampling (middle), and motif sampling (right).
		}
		\label{fig:sampling}
	\end{figure*}
	
	{\color{black}
		Once we have developed sufficient mathematical and computational foundations for the motif sampling problem, we will use them to devise computationally efficient and stable network observables. As the typical size and complexity of network data  far exceed the capabilities of human perception, we need some lens through which we can study and analyze network data. Namely, given a network $\G$, we want to associate a much simpler object $f(\G)$, which we call a \textit{network observable}, such that it can be computed in a reasonable amount of time even when $\G$ is large and complex, and yet it retains substantial information about $\G$. These two desired properties of network observables are stated more precisely below:
		\begin{description}
			\item{(i)} (\textit{Computability}) The observable $f(\G)$ is computable in at most polynomial time in the size of the network $\G$. 
			
			\item{(ii)} (\textit{Stability}) For  two given networks $\G_{1},\G_{2}$, we have 
			\begin{align}\label{eq:stability_ineq}
				d(f(\G_{1}),f(\G_{2})) \le d(\G_{1},\G_{2}),
			\end{align}
			where $d$ on each side denotes a suitable distance metric between  observables and between networks, respectively.
		\end{description} 
		An inequality of type \eqref{eq:stability_ineq} is called a `stability inequality' for the observable $f(\G)$, which encodes the property that a small change in the network yields small change in the observable. 
	}

	\subsection{Our approach and contribution}
	{\color{black}
		We summarize our approach and contributions in the following bullet points. 
		
		\begin{description}
			\item{$\bullet$} We propose a new network sampling framework based on sampling a graph homomorphism from a small template network $F$ into a large target network $\G$. 
			
			\item{$\bullet$} We propose two complementary MCMC algorithms for sampling random graph homomorphisms and establish bounds on their mixing times and concentration of their time averages. 
			
			\item{$\bullet$} Based on our sampling algorithms, we propose a number of network observables that are both easy to compute (using our MCMC motif-sampling algorithms) and  provably stable. 
			
			\item{$\bullet$} We demonstrate the efficacy of our techniques through various synthetic and real-world networks. For instance, \commHL{for subgraph classification problems on Facebook social networks, our Matrix of Average Clustering Coefficient (MACC) achieves performance better than} 
			the benchmark methods (see Figure \ref{fig:intro_sim_MACC} and Section \ref{section:FB}). 
		\end{description}
	}
	
	\commHL{The key insight in our approach is to sample adjacency-preserving functions from small graphs to large networks, instead of directly sampling subgraphs. Namely, suppose $\G=(V,E_{\G})$ is a large and possibly sparse graph and $F=(\{1,\dots,k\}, E_{F})$ is a $k$-node template graph. A vertex map $\x:\{1,\dots,k\}\rightarrow V$ is said to be a (graph) \textit{homomorphism} $F\rightarrow \G$ if it preserves  adjacency relations, that is, $\x(i)$ and $\x(j)$ are adjacent in $\G$ if $i$ and $j$ are adjacent in $F$. Our main goal then becomes the following: 
		\begin{align}\label{eq:motif_sampling_intro}
			\textit{Sample a graph homomorphism $\x:F\rightarrow \G$ uniformly at random.}
		\end{align}
		We consider the above problem in the general context where $\G$ is a network with edge weights equipped with a probability distribution on the nodes. }

	
	\commHL{To tackle the homomorphism sampling problem \eqref{eq:motif_sampling_intro}, we propose two complementary Markov Chain Monte Carlo algorithms. In other words, the algorithms proceed by sampling a Markov chain of graph homomorphisms $\x_{t}:F\rightarrow \G$ in a way such that the empirical distribution of $\x_{t}$ converges to the desired target distribution. }
	
	\commHL{Our network observables based on motif sampling will be of the following form: 
		\begin{align}
			f(\G):=\P( \textit{A uniformly random homomorphism $\x:F\rightarrow \G$ satisfies a property $P$}).
		\end{align}
		For instance, the well-known \emph{average clustering coefficient}  network observable  can be realized in the form above (see Example \ref{ex:avg_clustering_coeff}), which we generalize to \textit{conditional homomorphism densities} (see Section \ref{section:observables_def}). By taking the expectation of some function of the random homomorphism $\x$, we can also define not only real-valued network observables, 
		but also function- (see Figure \ref{fig:intro_sim}), matrix- (see Figure \ref{fig:intro_sim_MACC}), and even network-valued observables. These observables can all be efficiently (and provably) computed by taking suitable time averages along the MCMC trajectory of \commHL{the MCMC motif sampling procedure} 
		(see Theorems \ref{thm:McDiarmids} and \ref{thm:vector_concentration}).
	} Furthermore, we establish that these network observables are stable in the sense that a small change in the network results in a small change in their values (see Section \ref{section:stability_inequalities}). Our new network observables are not vanishingly small for sparse networks and are able to capture multi-scale features. Moreover, they can directly be applied to comparing networks with different sizes without node labels (e.g., comparing two social networks with anonymous users or brain networks of two species) with low computational cost. 
	
	\begin{figure*}[h]
		\centering
		\includegraphics[width=1 \linewidth]{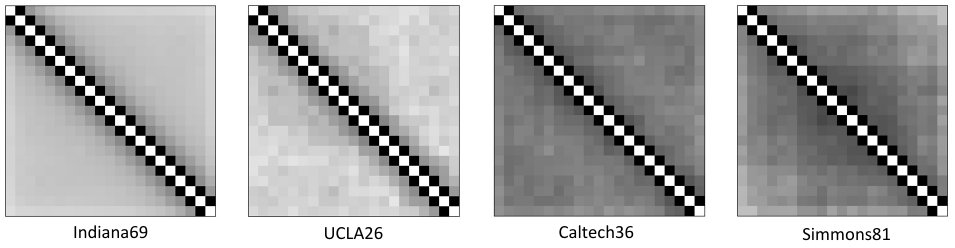}
		\caption{ Matrices of Average Clustering Coefficients of the Facebook network corresponding to four schools in the Facebook100 dataset \commHL{using the chain motif of $21$ nodes. The $21\times 21$ matrices are summarizing observables of the corresponding Facebook networks.} See Figure \ref{fig:FB_MACC} for more details. 
		}
		\label{fig:intro_sim_MACC}
	\end{figure*}

	To demonstrate our new sampling technique and Network Data analysis framework, we apply our framework for network clustering and classification problems using the Facebook100 dataset and Word Adjacency Networks of a set of classic novels. Our new matrix-valued network observable compresses a given network of arbitrary size without node label into a fixed size matrix, which reveals local clustering structures of the network in any desired scale (see Figure \ref{fig:intro_sim_MACC}). We use these low-dimensional representations to perform \commHL{subgraph classification} and hierarchical clustering of the 100 network data. \commHL{For the former supervised task, our proposed method shows significantly better performance than the baseline methods.} On the other hand, we analyze the hierarchical structure of weighted networks representing text data using our function-valued observable. The obtained observables indicate similar hierarchical structures among different texts by the same author that are distinctive between different authors.

	\begin{figure*}[h]
		\centering
		\includegraphics[width=1 \linewidth]{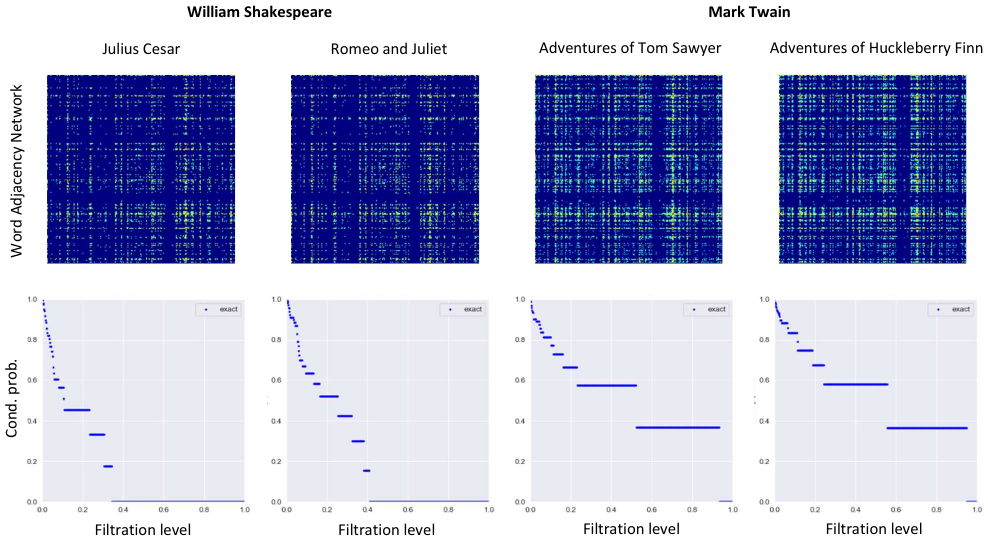}
		\caption{ Heat map of the Word Adjacency Networks of four novels and their CHD profiles corresponding to the pair of motifs $(H_{0,0},H_{0,0})$, $H_{0,0} = (\{0\}, \mathbf{1}_{\{(0,0)\}})$. \commHL{The non-increasing functions in the second row  summarize the observables of the networks shown in the first row. See Section \ref{subsection:CHD_WAN} for details.}
		}
		\label{fig:intro_sim}
	\end{figure*}

	\vspace{0.3cm}
	\subsection{Background and related work}
	
	\commHL{The motif sampling problem from \eqref{eq:motif_sampling_intro}  generalizes the well-known problem of sampling a proper coloring of a given graph uniformly at random. 
		Recall that a proper $q$-coloring of a simple graph $G=(V,E)$ is an assignment of colors $\x:V\rightarrow \{1,\dots,q\}$ such that $\x(u)\ne \x(v)$ whenever nodes $u$ and $v$ are adjacent in $G$. This is in fact a graph homomorphism $G\rightarrow K_{q}$, where $K_{q}$ is the complete graph of $q$ nodes. Indeed, in order to preserve the adjacency, any two adjacent nodes in $G$ should \textit{not} be mapped into the same node in $K_{q}$. A number of MCMC algorithms and their mixing times to sample a uniform $q$-coloring of a graph have been studied in the past few decades \citep{jerrum1995very, salas1997absence, vigoda2000improved, dyer2002very, frieze2007survey}. One of our MCMC motif sampling algorithms, the Glauber chain (see Definition \ref{def:glauber_chain}), is inspired by the standard Glauber dynamics for sampling proper $q$-coloring of graphs. }
	
	
	There is an interesting change of perspective between the graph coloring problem and motif sampling. Namely, in graph coloring $G\rightarrow K_{q}$, the problem becomes easier for large $q$ and hence the attention is toward sampling a random $q$-coloring for small $q$. On the other hand, for motif sampling, our goal is to analyze large network $\G$ through a random homomorphism $F\rightarrow \G$ from a relatively small motif $F$. It will be conceptually helpful to visualize a homomorphism $F\rightarrow \G$ as a graph-theoretic embedding of a motif $F$ into the large network $\G$.

	\commHL{Our work on constructing stable network observables from motif sampling algorithms is inspired by the graph homomorphism and graph limit theory (see, e.g., \cite{lovasz2006limits, lovasz2012large}), and by methods from Topological Data Analysis (see, e.g., \cite{carlsson2009topology, edelsbrunner2010computational}), which considers the hierarchical structure of certain observables and studies their stability properties. }

	\commHL{For an illustrative example,} let $G=(V,E)$ be a finite simple graph and let $K_{3}$ be a triangle. Choose three nodes $x_{1},x_{2},x_{3}$ independently from $V$ uniformly at random, and define an observable $\mathtt{t}(K_{3},G)$, which is called the homomorphism density of $K_{3}$ in $G$, by 
	\begin{align}\label{eq:hom_density_1}
		\mathtt{t}(K_{3},G) := \P(\text{there is an edge between $x_{i}$ and $x_{j}$ for all $1\le i<j\le 3$}).
	\end{align}  
	In words, this is the probability that three randomly chosen people from a social network are friends of each other. If we replace the triangle $K_{3}$ with an arbitrary simple graph $F$, a similar observable $\mathtt{t}(F,G)$ can be defined. Note that computing such observables can be done by repeated sampling and averaging. Moreover, a fundamental lemma due to  \cite{lovasz2006limits} asserts that the homomorphism densities are stable with respect to the cut distance between graphs (or graphons, in general):
	\begin{equation}\label{ineq:countinglemma}
		| \mathtt{t}(F,G_{1})-\mathtt{t}(F,G_{2}) | \le |E_{F}| \cdot \delta_{\square}(G_{1},G_{2}),
	\end{equation} 
	where $G_{1},G_{2}$ are simple graphs and $E_{F}$ is the set of edges in $F$. Hence by varying $F$, we obtain a family of observables that satisfy the computability and stability (note that we can absorb the constant $|E_{F}|$ into the cut distance $\delta_{\square}$). 
	
	However, there are two notable shortcomings of homomorphism densities as network observables. First, they provide no useful information for sparse networks, where the average degree is of order sublinear in the number of nodes (e.g., two-dimensional lattices, trees, most real-world social networks \citep{barabasi2013network, newman2018networks}). This is because for sparse networks the independent sampling outputs a set of non-adjacent nodes with high probability. In terms of the stability inequality \eqref{ineq:countinglemma}, this is reflected in the fact that the cut distance $\delta_{\square}$ between two sparse networks becomes asymptotically zero as the sizes of networks tend to infinity. \commHL{Second, homomorphism densities do not capture hierarchical features of weighted networks. Namely, we might be interested in how the density of triangles formed through edges of weights at least $t$ changes as we increase the parameter $t$. But the homomorphism density of triangles aggregates such information into a single numeric value, \commHL{which is independent of $t$.} }

	An entirely different approach is taken in the fields of Topological Data Analysis (TDA) in order to capture multi-scale features of data sets \citep{carlsson2009topology, edelsbrunner2010computational}. The essential workflow in TDA is as follows. First, a data set $X$ consisting of a finite number of points in Euclidean space $\mathbb{R}^{d}$ is given. In order to equip the data set with a topological structure, one constructs a filtration of simplicial complexes on top of $X$ by attaching a suitable set of high-dimensional cells according to the filtration parameter (spatial resolution). Then by computing the homology of the filtration (or the persistent homology of $X$), one can associate $X$ with a topological invariant $f(X)$ called the persistence diagram \citep{edelsbrunner2000topological} (or barcodes \citep{ghrist2008barcodes}). The stability of such observable is well-known \citep{cohen2007stability, chazal2009gromov}. Namely, it holds that 
	\begin{align}
		d_{B}(f(X),f(Y)) \le d_{GH}(X,Y),
	\end{align}
	where the distance metric on the left and right-hand side denotes the bottleneck distance between persistence diagrams and the Gromov-Hausdorff distance between data sets $X$ and $Y$ viewed as finite metric spaces. However, as is well known in the TDA community, computing persistence diagrams for large data sets is computationally expensive (see 
	\cite{edelsbrunner2000topological, zomorodian2005computing} for earlier algorithms and \cite{carlsson2009topology,edelsbrunner2012persistent,otter2017roadmap,memoli2019primer} for recent surveys).  
	
	\commHL{Whereas in the present work we  concentrate on \textit{symmetric networks}, where the edge weight between two nodes $x$ and $y$ does not depend on their ordering, we acknowledge that in the context of asymmetric networks, several possible observables $f$ and a suitable metric are studied in \citep{samir-dn-1,samir-dn-2,samir-dowker,turner2019rips,samir-pph,samir-dgw-nets}.}
	
	\commHL{We also remark that an earlier version of the present work has already found several applications in the literature of network data analysis. The MCMC motif sampling algorithms as well as their theoretical guarantees were used as a key component in the recent network dictionary learning methods of~\citep{lyu2021learning, lyu2020online, peng2022inferring}. Also, a MCMC $k$-path sampling algorithm was used to generate sub-texts within knowledge graphs for topic modeling applications \citep{alaverdian2020killed}. The same algorithm was used to benchmark stochastic proximal gradient descent algorithms for Markovian data in
		\citep{alacaoglu2022convergence}. }

	\subsection{Organization}
	{\color{black}
		We formally introduce the motif sampling problem on networks in Section \ref{subsection:random_hom} and discuss a concrete example of such sampling scheme in the form of subgraph sampling via Hamiltonian paths in Section \ref{section:sampling_hamiltonian}. In Section \ref{section:MCMC}, we introduce two Markov chain Monte Carlo (MCMC) algorithms for motif sampling. Their convergence is stated in Theorems \ref{thm:stationary_measure_Glauber} and \ref{thm:stationary_measure_pivot} and their mixing time bounds are stated in Theorems \ref{thm:gen_coloring_mixing_tree} and \ref{thm:pivot_chain_mixing}. We also deduce that the expected value of various functions of the random homomorphism can be efficiently computed by time averages of the MCMC trajectory (see Corollary \ref{cor:time_avg_observables}). Moreover, these estimates are guaranteed to be close to the expected value according to the concentration inequalities that we obtain in Theorems \ref{thm:McDiarmids} and \ref{thm:vector_concentration}.
		
		In Section \ref{section:observables_def}, we introduce four network observables (Conditional Homomormorphism Density, Matrix of Average Clustering Coefficients, CHD profile, and motif transform) by taking the expected value of suitable functions of random homomorphism $F\rightarrow \G$. We also provide some elementary examples. In Section \ref{section:stability_inequalities}, we state stability inequalities (Propositions \ref{prop:conditioned_counting_lemma}, \ref{prop:stability_Ftransform}, and Theorem \ref{thm:counting_filtration}) for our network observables using the language of graphons and the cut distance. 
		
		Sections \ref{section:examples} and \ref{section:applications} are devoted to examples and applications of our framework. In Section \ref{section:examples},  we provide various examples and simulations demonstrating our results on synthetic networks. In Section \ref{section:FB}, we apply our methods to the Facebook social network for the tasks of subgraph classification and hierarchical clustering. In Section \ref{section:applications}, we apply our framework to analyze Word Adjacency Networks of a set consisting of 45 novels  and propose an authorship attribution scheme using motif sampling and conditional homomorphism profiles.

		Finally, we provide additional discussions, examples, proofs, and figures in the appendices. In Appendix \ref{section:motif_transform_spectral}, we discuss the relationship between motif transforms and spectral analysis. In Appendices \ref{section:proofs_mixing} and \ref{section:proofs_stability}, we prove convergence, mixing time bounds, and concentration of the MCMC algorithms as well as the stability inequalities of our network observables. 
	}
	
	\subsection{Notation}
	
	For each integer $n\ge 1$, we write $[n]=\{1,2,\cdots, n\}$. Given a matrix $A:[n]^{2}\rightarrow [0,\infty)$, we call the pair $G=([n],A)$ an \textit{edge-weighted graph} with \textit{node set} $[n]$ and \text{edge weight} $A$. When $A$ is 0-1 valued, we call $G$ a \textit{directed graph} and we also write $G=([n],E)$, where $E=\{(i,j)\in [n]^{2}\,|\, A(i,j)=1\}$ is the set of all directed edges. If $A$ is 0-1 valued, symmetric, and has all diagonal entries equal to $0$, then we call $G$ a \textit{simple graph}. Given an edge-weighted graph $G=([n],A)$, define its \textit{maximum degree} by 
	\begin{align}\label{eq:def_max_deg_edge_weighted_graph}
		\Delta(G) = \max_{a\in [n]} \sum_{b\in [n]} \mathbf{1}\big(A(a,b)+A(b,a)>0\big).
	\end{align}  
	A sequence $(x_{j})_{j=0}^{m}$ of nodes in $G$ is called a \textit{walk of length $m$} if $A(x_{j},x_{j+1})>0$ for all $0\le j <m$. A walk is a \textit{path} if all nodes in the walk are distinct. We define the \textit{diameter} of $G$, which we denote by $\text{diam}(G)$, by 
	\begin{align}
		\text{diam}(G) = \max_{a,b\in [n]} \min \{ k\ge 0\,|\, \text{$\exists$ a path of length $k$ between $a$ and $b$} \}.
	\end{align}
	We let $\diam(G)=\infty$ if there is no path between some $x,y\in [n]$.

	{For an event $B$, we let $\mathbf{1}_{B}$ denote the indicator function of $B$, where $\mathbf{1}_{B}(\omega)=1$ if $\omega\in B$ and 0 otherwise. We also write $\mathbf{1}_{B}=\mathbf{1}(B)$ when convenient.}  For two real numbers $a,b\in \mathbb{R}$, we  write $a\lor b = \max(a,b)$ and $a\land b = \min(a,b)$.

	\vspace{0.4cm}
	\section{Motif sampling and MCMC sampling algorithms}
	\label{section:motif_sampling_MCMC}
	
	\subsection{Random homomorphism from motifs into networks}
	\label{subsection:random_hom}
	
	To describe motif sampling, we first give precise definitions of networks and motifs. A \textit{network} as a mathematical object consists of a triple $\G=(X,A,\alpha)$, where $X$, a finite set,  is the node set of individuals, $A:X^{2}\rightarrow [0,\infty)$ is a matrix describing interaction strength between individuals, and $\alpha:X\rightarrow (0,1]$ is a probability measure on $X$ giving the significance of each individual (cf. \citet{samir-dgw-nets}). Any given $(n\times n)$ matrix $A$ taking values from $[0,1]$ can be regarded as a network $([n],A,\alpha)$ where $\alpha(i)\equiv 1/n$ is the uniform distribution on $[n]$.

	Fix an integer $k\ge 1$ and a matrix $A_{F}:[k]^{2}\rightarrow [0,\infty)$. Let $F=([k],A_{F})$ denote the corresponding edge-weighted graph, which we also call a \textit{motif}. A motif $F=([k],A_{F})$ is said to be \textit{simple} if $A_{F}$ is 0-1 valued, has zero diagonal entries (no loops), and $A_{F}(i,j)+A_{F}(j,i)\in \{0,1\}$ for each $1\le i<j\le k$ (see Figure \ref{fig:motif} for an illustration). The fact that simple motifs have at most one directed edge between any pair of nodes is crucial in the proof of stability inequalities of the network observables stated in Section  \ref{section:stability_inequalities}. A particularly important motif for application purposes is the \textit{$k$-chain}, which is the pair $([k], \mathbf{1}_{\{(1,2),(2,3),\dots,(k-1,k)})$. It corresponds to the direct path on $k$ nodes, see Figure \ref{fig:motif} (c).

	\begin{figure*}[h]
		\centering
		\includegraphics[width=0.9 \linewidth]{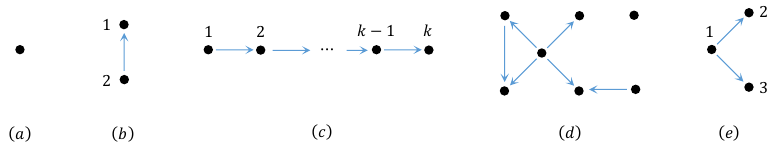}
		\caption{  \commHL{Examples of simple motifs. Motifs may contain no edge (a) or multiple connected components (d). The motif in (c) forms a directed path on $k$ nodes, which we call the `$k$-chain'.}
		}
		\label{fig:motif}
	\end{figure*}
	
	For a given motif $F=([k],A_{F})$ and a $n$-node network $\G=([n],A,\alpha)$, we introduce the following probability distribution $\pi_{F\rightarrow \G}$ on the set $[n]^{[k]}$ of all vertex maps $\mathbf{x}:[k]\rightarrow [n]$ by 
	\begin{equation}\label{eq:def_embedding_F_N}
		\pi_{F\rightarrow \G}( \mathbf{x} ) = \frac{1}{\mathtt{Z}}  \left( \prod_{1\le i,j\le k}  A(\mathbf{x}(i),\mathbf{x}(j))^{A_{F}(i,j)} \right)\, \alpha(\mathbf{x}(1))\cdots \alpha(\mathbf{x}(k)),
	\end{equation}  
	where the normalizing constant $\mathtt{Z}$ is given by 
	\begin{equation}\label{eq:def_partition_function}
		\mathtt{Z}=\mathtt{t}(F,\G):=\sum_{\x:[k]\rightarrow [n]} \left( \prod_{1\le i,j\le k}  A(\mathbf{x}(i),\mathbf{x}(j))^{A_{F}(i,j)} \right) \, \alpha(\mathbf{x}(1))\cdots \alpha(\mathbf{x}(k)).
	\end{equation}
	We call a random vertex map $\x:[k]\rightarrow [n]$ distributed as $\pi_{F\rightarrow \G}$ a \textit{random homomorphism} from $F$ to $\G$.  A vertex map $\mathbf{x}:[k]\rightarrow [n]$ is a (graph) \textit{homomorphism} $F\rightarrow \G$ if $\pi_{F\rightarrow \G}(\mathbf{x})>0$. Hence $\pi_{F\rightarrow \G}$ is a probability measure on the set of all homomorphisms $F\rightarrow \G$. The above quantity $\mathtt{t}(F,\G)$ is known as the \textit{homomorphism density} of $F$ in $\G$. We now formally introduce the problem of motif sampling.
	
	{\color{black}
		\begin{problem}[Motif sampling from networks]\label{prob:motif_sampling}
			For a given motif $F=([k],A_{F})$ and an  $n$-node network $\G=([n],A,\alpha)$, sample a homomorphism $\x:F\rightarrow \G$ according to the probability distribution $\pi_{F\rightarrow \G}$ in \eqref{eq:def_embedding_F_N}.  
		\end{problem}
		
		An important special case of \eqref{prob:motif_sampling} is when $\G$ is a simple graph. Let $G=([n],A)$ be a simple graph. Then for each vertex map $\x:[k]\rightarrow [n]$, note that 
		\begin{align}\label{eq:hom_product_simple_indicator}
			\prod_{1\le i,j\le k}  A(\mathbf{x}(i),\mathbf{x}(j))^{A_{F}(i,j)} = \mathbf{1}\left(\textup{for all $(i,j)$ with $A_{F}(i,j)=1$ and $A(\x(i),\x(j))=1$}\right). 
		\end{align}
		Whenever the indicator on the right-hand side above equals one, we say $\x$ is a homomorphism $F\rightarrow G$. That is, $\x$ maps an edge in $F$ to an edge in $G$. Note that $\x$ need not be injective, so different edges in $F$ can be mapped to the same edge in $G$. This leads us to the problem of motif sampling from graphs as described below.
		
		\begin{problem}[Motif sampling from graphs]\label{prob:motif_sampling_graph}
			For a given motif $F=([k],A_{F})$ and a $n$-node simple graph $G=([n],A)$, sample a homomorphism $\x:F\rightarrow G$ uniformly at random.
		\end{problem}
		
		The Problem \ref{prob:motif_sampling_graph} is indeed a special case Problem \ref{prob:motif_sampling} by identifying the simple graph $G=([n],A)$ with the network $\G=([n], A, \alpha)$, where $\alpha$ is the uniform node weight (i.e., $\alpha(i)\equiv 1/n$ for $i=1,\dots,n$). Then due to \eqref{eq:hom_product_simple_indicator}, the probability distribution $\pi_{F\rightarrow \G}$ in \eqref{eq:def_embedding_F_N} becomes the uniform distribution on the set of all homomorphisms $F\rightarrow \G$. }
	
	\subsection{Sampling subgraphs with uniform Hamiltonian path}\label{section:sampling_hamiltonian}
	
	{\color{black}
		In order to provide some concrete application contexts for the motif sampling problems posed above, here we consider the problem sampling connected subgraphs from sparse graphs. Computing a large number of $k$-node subgraphs from a given  network is an essential task in modern network analysis, such as in computing `network motifs'~\citep{milo2002network} and `latent motifs'~\citep{lyu2021learning, lyu2020online, peng2022inferring} and in topic modeling on knowledge graphs \citep{alaverdian2020killed}.
		
		We consider the  random sampling of $k$-node subgraphs that we obtain by uniformly randomly sampling a `$k$-path' from a network and taking the induced subgraph on the sampled nodes. This subgraph sampling procedure is summarized below. (See Figure \ref{fig:sampling} for an illustration.) Here, a \textit{$k$-path} is a subgraph that consists of $k$ distinct nodes, with the $i$th node adjacent to the $(i+1)$th node for all $i \in \{1,\ldots,k-1\}$. \commHL{A path $P$ in a graph $G$ is a \textit{Hamiltonian path} if $P$ contains all nodes of $G$.} 
		\begin{description}[itemsep=0.1cm]
			\item{\textit{Sampling subgraphs 
					via  uniform Hamiltonian paths}.} 
			\item{\quad Given a simple graph $G=([n], A)$ and an integer $1\le k\le \textup{diam}(G)$:}
			\item[\qquad (1)] Sample a $k$-path $P\subseteq G$  uniformly at random; 

			\item[\qquad (2)] Return the $k$-node induced subgraph $H$ of $G$ on the nodes in $P$. 
		\end{description}
		Above, sampling a subgraph induced by a $k$-path serves two purposes: (1) It ensures that the sampled $k$-node induced subgraph is connected with the minimum number of imposed edges; and (2) it induces a natural node ordering of the $k$-node induced subgraph. When applied to sparse networks, such $k$-node subgraphs are not likely to possess many other Hamiltonian paths, so ordering the nodes using the sampled Hamiltonian path provides a canonical representation of the subgraphs as their $k\times k$ adjacency matrices (e.g., see Figure \ref{fig:sampling} (c)). This is an important computational advantage of sampling $k$-node subgraphs via Hamiltonian paths over neighborhood sampling. In the latter, there is no canonical choice of node ordering out of $k!$ ways so there is a large number of equivalent adjacency matrix representations for the same subgraph.

		\begin{figure*}[h]
			\centering
			\includegraphics[width=1 \linewidth]{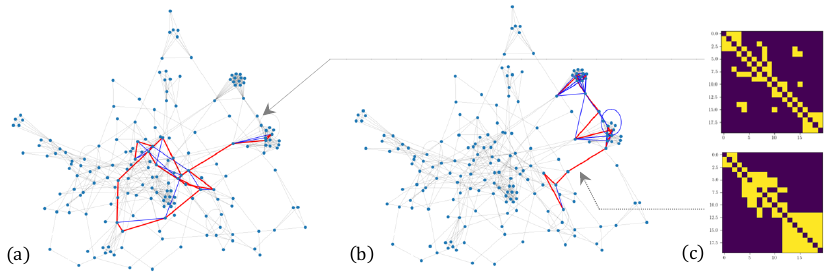}
			\caption{\commHL{Illustration of motif sampling with chain motif of $k=20$ nodes. Two instances of injective homomorphisms from a path of 20 nodes into the same network are shown in panels (a) and (b), which are depicted as paths of $k$ nodes with red edges.  Panel (c) shows the $k\times k$ adjacency matrix of the induced subgraph on these $k$-paths. }
			}
			\label{fig:sampling}
		\end{figure*}

		The $k$-node subgraph induced on such a uniformly random $k$-path is guaranteed to be connected and can exhibit diverse connection patterns (see Figure~\ref{fig:subgraphs_ex}), depending on the structure of the original network. 
		
		\begin{figure*}[h]
			\centering
			\includegraphics[width=1 \linewidth]{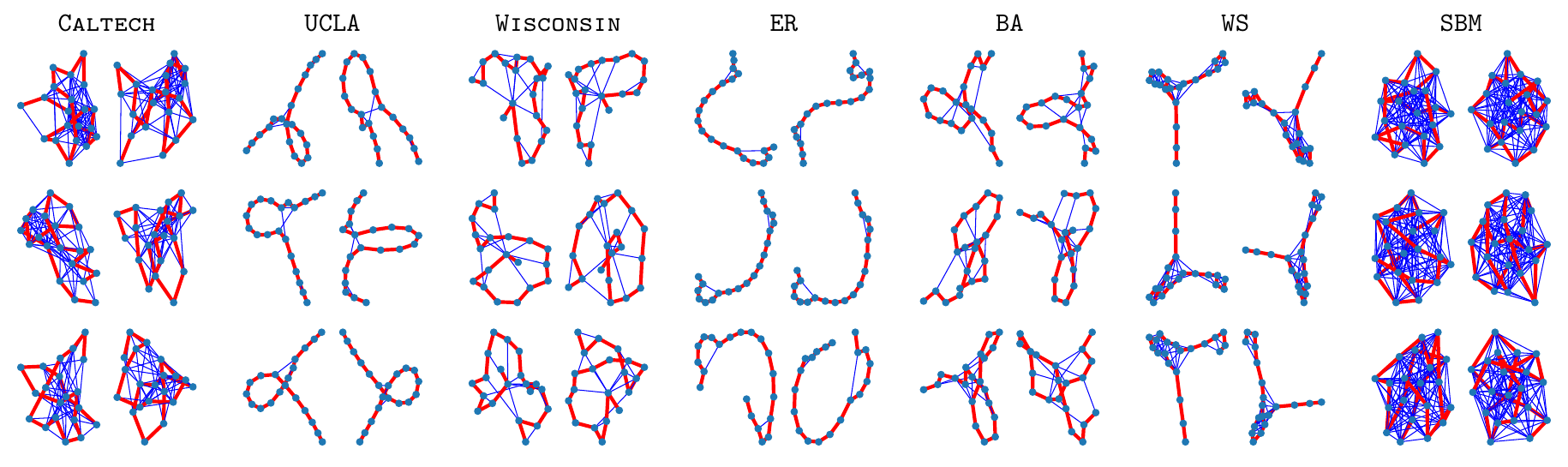}
			\caption{ \commHL{Examples of 20-node subgraphs  \commHL{induced through}  Hamiltonian paths on three Facebook social networks \citep{traud2012social} and on synthetic networks generated according to the following models: the Erd\H{o}s--R\'{e}nyi (\dataset{ER}) \citep{erdds1959random}, Barab\'{a}si--Albert (\dataset{BA}) \citep{barabasi1999emergence} Watts--Strogatz (\dataset{WS}) \citep{watts1998collective}, and stochastic-block-model (\dataset{SBM}) \citep{holland1983stochastic}.	
					For each subgraph, its Hamiltonian path (with red edges) is sampled uniformly at random by using the Glauber chain algorithm (see Def. \ref{def:glauber_chain}). }
			}
			\label{fig:subgraphs_ex}
		\end{figure*}
		
		The key sampling problem in the above subgraph sampling scheme is to sample a $k$-path uniformly at random from a graph $G$. A naive way to do so is to use rejection sampling together with independent sampling. That is, one can repeatedly sample a set $\{x_{1},\dots,x_{k}\}$ of $k$ distinct nodes in $G$ independently and uniformly at random until there is a path on the sequence $(x_{1},\dots,x_{k})$ (i.e., $x_{i}$ and $x_{i}$ are adjacent for $i=1,\dots,k-1$). However, when $G$ is sparse (i.e., when the number of edges in $G$ is much less than $n^{2}$), the probability that a randomly chosen node set $(x_{1},\dots,x_{k})$ forms a $k$-path is extremely small, so this procedure  might suffer a large number of rejections until finding a $k$-path.  
		
		We propose to use motif sampling with $k$-chain motifs to address this problem of sampling a $k$-path uniformly at random. Let $F=([k], \mathbf{1}_{\{ (1,2),(2,3),\dots,(k-1,k) \}})$ denote a $k$-chain motif. Consider the problem sampling a homomorphism $\x:F\rightarrow G$ uniformly at random with the additional constraint that $\x$ be injective, that is, the nodes $\x(1),\dots,\x(k)$ are distinct. {When $\x:F\rightarrow \G$ is an injective homomorphism}, we denote $\x:F\hookrightarrow \G$. This would give us a uniformly random $k$-path on the node set $\{\x(1),\dots,\x(k)\}$. Letting $\pi_{F\hookrightarrow \G}$ denote the probability distribution on the set of all injective homomorphisms $F\rightarrow G$, we can write 
		\begin{equation}\label{eq:def_embedding_F_N_inj}
			\pi_{F\hookrightarrow \G}( \mathbf{x} ) := C \,  \pi_{F\rightarrow \G}( \mathbf{x} )  \cdot \one(\textup{$\x(1),\ldots,\x(k)$ are distinct})\,,
		\end{equation}  
		where $C > 0$ is a normalization constant.  The probability distribution \eqref{eq:def_embedding_F_N_inj} is well-defined as long as there exists an injective homomorphism $\x: F \rightarrow \G$. For instance, if $F$ is a $k$-chain motif for $k\ge 4$ and if $\G$ is a star graph, then there is no injective homomorphism $\x: F \rightarrow \G$ and the probability distribution \eqref{eq:def_embedding_F_N_inj} is not well-defined. 
		
		The identity \eqref{eq:def_embedding_F_N_inj} suggests that, if we can sample a homomorphism $\x:F\rightarrow \G$ uniformly at random efficiently, then we can sample a sequence of homomorphisms $\x_{1},\dots,\x_{m}:F\rightarrow G$ uniformly at random until the first time $m$ such that $\x_{m}$ is injective. Note that the probability of uniformly random homomorphism $\x:F\rightarrow G$ being injective is not vanishingly small even if $G$ is sparse. In Section~\ref{section:MCMC}, we provide two MCMC 
		sampling algorithms for sampling a homomorphism $F\rightarrow G$ uniformly at random. We remark that this sampling scheme is a crucial component in the recent development of network dictionary learning methods~\citep{lyu2021learning, lyu2020online, peng2022inferring}.
	}

	\subsection{MCMC algorithms for motif sampling}
	\label{section:MCMC}
	
	Note that computing the measure $\pi_{F\rightarrow \G}$ according to its definition is computationally expensive, especially when the network $\G$ is large. In this subsection, we give efficient randomized algorithms to sample a random homomorphism $F\rightarrow \G$ from the measure $\pi_{F\rightarrow \G}$ by a Markov chain Monte Carlo method. Namely, we seek for a Markov chain $(\x_{t})_{t\ge 0}$ evolving in the space $[n]^{[k]}$ of vertex maps $[k]\rightarrow [n]$ such that each $\x_{t}$ is a homomorphism $F\rightarrow \G$ and the chain $(\x_{t})_{t\ge 0}$ has a unique stationary distribution given by (\ref{eq:def_embedding_F_N}). We call such a Markov chain a \textit{dynamic embedding} of $F$ into $\G$. We propose two complementary dynamic embedding schemes.

	Observe that equation (\ref{eq:def_embedding_F_N}) suggests considering a spin model on $F$ where each site $i\in [k]$ takes a discrete spin $\x(i)\in [n]$ and the probability of such discrete spin configuration $\x:[k]\rightarrow [n]$ is given by (\ref{eq:def_embedding_F_N}). This spin model interpretation naturally leads us to the following dynamic embedding in terms of the Glauber chain. \commHL{See Figure \ref{fig:Glauber} for an illustration.}

	\begin{dfn}[Glauber chain]\label{def:glauber_chain}
		Let $F=([k],A_{F})$ be a simple motif and $\G=([n],A,\alpha)$ be a network. Suppose $\mathtt{t}(F,\G)>0$ and fix a homomorphism $\x_{0}:F \rightarrow \G$.  Define a Markov chain $\x_{t}$ of homomorphisms $F\rightarrow \G$ as below.   	
		\begin{description}
			\item{(i)} Choose a node $i \in [k]$ of $F$ uniformly at random. 
			\item{(ii)} Set $\x_{t+1}(j)=\x_{t}(j)$ for $j \ne i$. Update $\x_{t}(i)=a$ to $\x_{t+1}(i)=b$ according to the transition kernel   
			\begin{align}\label{eq:Glauber_kernel}
				G(a,b)= \frac{\left( \prod_{j\ne i} A(\x_{t}(j),b)^{A_{F}(j,i)} A(b,\x_{t}(j))^{A_{F}(i,j)}    \right)  A(b,b)^{A_{F}(i,i)} \alpha(b) }{\sum_{1\le c \le n} \left( \prod_{j\ne i} A(\x_{t}(j),c)^{A_{F}(j,i)} A(c,\x_{t}(j))^{A_{F}(i,j)}    \right)  A(c,c)^{A_{F}(i,i)} \alpha(c)  },
			\end{align}  
			where the product is overall $1\le j \le k$ such that $j\ne i$. 
		\end{description}
	\end{dfn}

	\begin{figure*}[h]
		\centering
		\includegraphics[width=0.8 \linewidth]{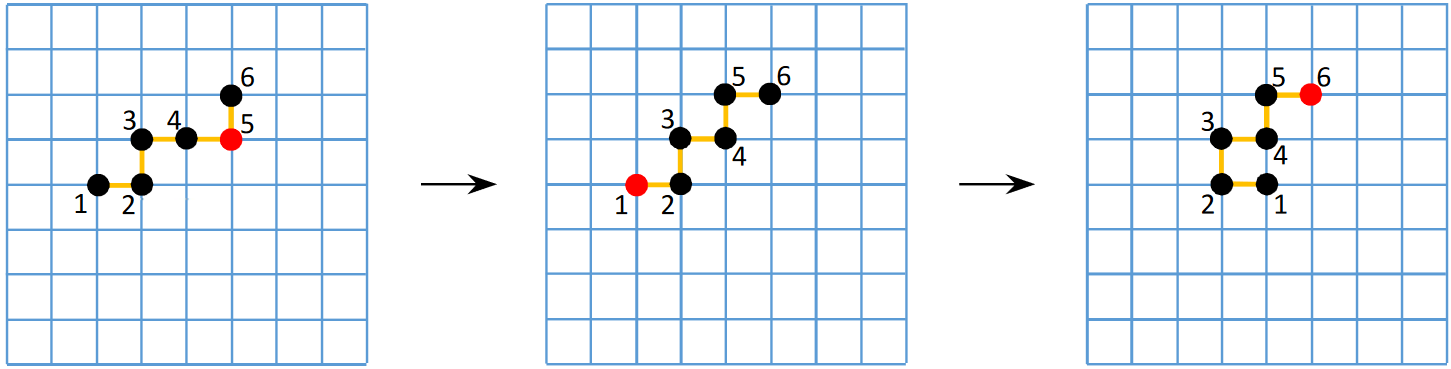}
		\caption{ Glauber chain of homomorphisms $\x_{t}:F \rightarrow \G$, where  $\G$ is the $(9\times 9)$ grid with uniform node weights and $F=([6],\mathbf{1}_{\{(1,2),(2,3),\cdots,(5,6)\} })$ is a $6$-chain. The orientations of the edges $(1,2), \ldots, (5,6)$ are suppressed in the figure. During the first transition, node $5$ is chosen with probability $1/6$ and $\x_{t}(5)$ is moved to the top left common neighbor of $\x_{t}(4)$ and $\x_{t}(6)$ with probability $1/2$. During the second transition, node $1$ is chosen with probability $1/6$ and $\x_{t+1}(1)$ is moved to the right neighbor of $\x_{t+1}(2)$ with probability $1/4$.   
		}
		\label{fig:Glauber}
	\end{figure*}

	Note that in the case of the Glauber chain, since all nodes in the motif try to move in all possible directions within the network, one can expect that it might take a long time to converge to its stationary distribution, $\pi_{F\rightarrow \G}$. To break the symmetry, we can designate a special node in the motif $F$ as \commHL{the} `pivot', and let it \commHL{`carry'} the rest of the homomorphism as it performs a simple random walk on $\G$. A canonical random walk kernel on $\G$ can be modified by the Metropolis-Hastings algorithm \commHL{(see, e.g.,~\cite[Sec. 3.2]{levin2017markov})} so that its unique stationary distribution agrees with the correct marginal distribution from the joint distribution $\pi_{F\rightarrow \G}$. We can then successively sample the rest of the embedded nodes (see Figure \ref{fig:pivot}) after each move of the pivot. We call this alternative dynamic embedding the \textit{pivot chain}. 
	
	To make a precise definition of the pivot chain, we restrict the motif $F=([k],A_{F})$ to be an edge-weighted directed tree rooted at node $1$ without loops. More precisely, suppose $A_{F}=0$ if $k=1$ and for $k\ge 2$, we assume that for each $2\le i \le k$, $A_{F}(j,i)>0$ for some unique $1\le j \le k$, $j\ne i$. In this case, we denote $j=i^{-}$ and call it the \textit{parent} of $i$. We may also assume that the other nodes in $\{ 2,\cdots,k \}$ are in a depth-first order, so that $i^{-}<i$ for all $2\le i \le k$. We can always assume such ordering \commHL{is given} by suitably permuting the vertices, if necessary. In this case, we call $F$ a \textit{rooted tree motif}.

	Now we introduce the pivot chain. See Figure \ref{fig:pivot} for an illustration.

	\begin{dfn}[Pivot chain]
		Let $F=([k],A_{F})$ be a rooted tree motif and let $\G=([n],A,\alpha)$ be a network such that for each $i\in [n]$, $A(i,j)>0$ for some $j\in [n]$. Let $\x_{0}:[k]\rightarrow [n]$ be an arbitrary homomorphism. Define a Markov chain $\x_{t}$ of homomorphisms $F\rightarrow \G$ as follows.   	
		\begin{description}
			\item{(i)} Given $\x_{t}(1)=a$, sample a node $b\in [n]$ according to the distribution $\Psi(a,\cdot\,)$, where the kernel $\Psi:[n]^{2}\rightarrow [0,1]$ is defined by  
			\begin{align}\label{eq:RWkernel_G}
				\Psi(a,b) := \frac{\alpha(a) \max(A(a,b), A(b,a))\alpha(b)}{ \sum_{c\in [n]} \alpha(a) \max(A(a,c), A(c,a))\alpha(c)  } \qquad a,b\in [n].
			\end{align}
			
			\item{(ii)} Let $\pi^{(1)}$ denote the projection of the probability distribution $\pi_{F\rightarrow \G}$ (defined at \eqref{eq:def_embedding_F_N}) onto the location of node 1. Then accept the update $a\mapsto b$ and set $\x_{t+1}(1)=b$ or reject the update and set $\x_{t+1}(1)=a$ independently with probability $\lambda$ or $1-\lambda$, respectively, where 
			\begin{align}\label{eq:pivot_acceptance_prob_gen}
				\lambda :=  \left[ \frac{\pi^{(1)}(b)}{\pi^{(1)}(a)}	\frac{\Psi(b,a)}{\Psi(a,b)} \land 1\right].
			\end{align}

			\item{(iii)} Having sampled $\x_{t+1}(1),\cdots,\x_{t+1}(i-1)\in [n]$, inductively, sample $\x_{t+1}(i)\in [n]$ according to the following conditional probability distribution 
			\begin{align}
				&\mathbb{P}(\x_{t+1}(i)=x_{i}\,|\,\x_{t+1}(1)=x_{1}, \cdots, \x_{t+1}(i-1)=x_{i-1}) \\
				&\hspace{3cm} = \frac{\left(\prod_{2\le j < i} A(x_{j^{-}},x_{j})\alpha(j) \right)A(x_{i^{-}}, x_{i})\alpha(x_{i})}{\sum_{c\in [n]} \left(\prod_{2\le j < i} A(x_{j^{-}},x_{j})\alpha(j) \right)A(x_{i^{-}}, c)\alpha(c)}.
			\end{align}  
		\end{description}
	\end{dfn}

	\begin{figure*}[h]
		\centering
		\includegraphics[width=0.8 \linewidth]{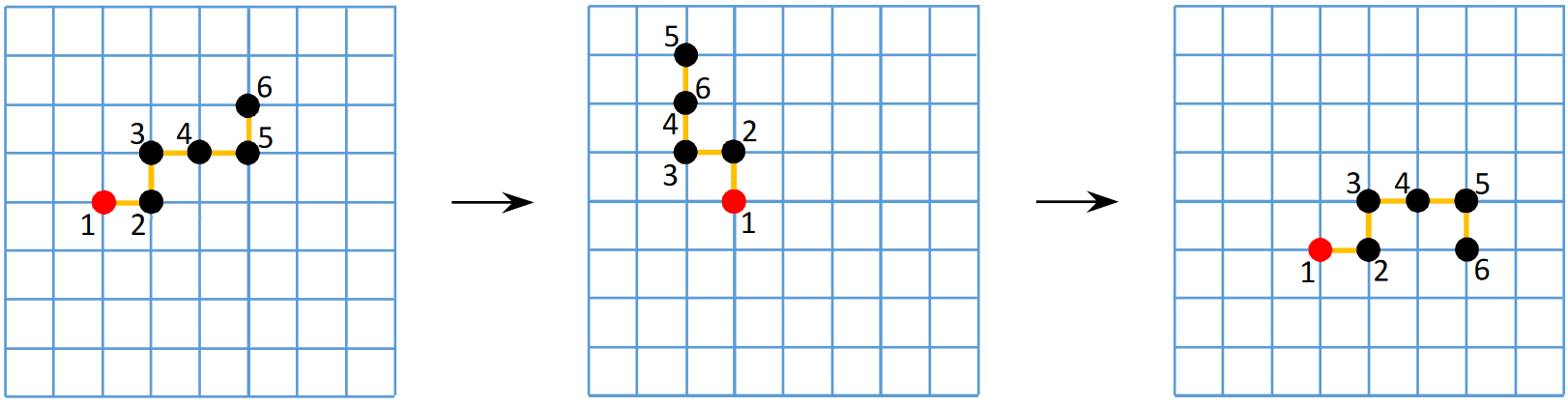}
		\caption{ Pivot chain of homomorphisms $\x_{t}:F \rightarrow \G$, where $\G$ is the $(9\times 9)$ grid with uniform node weight and $F=([6],\mathbf{1}_{\{(1,2),(2,3),\cdots,(5,6)\} })$ is a \commHL{$6$-chain}. The orientations of the edges $(1,2), \ldots, (5,6)$ are suppressed in the figure. During the first transition, the pivot $\x_{t}(1)$ moves to its right neighbor with probability $1/4$, and $\x_{t+1}(i)$ is sampled uniformly among the four neighbors of $\x_{t+1}(i-1)$ for $i=2$ to $6$. Note that $\x_{t+1}(4)=\x_{t+1}(6)$ in the middle figure. In the second transition, the pivot moves down with probability $1/4$, and again $\x_{t+1}(i)$ is sampled uniformly among the four neighbors of $\x_{t+1}(i-1)$ for $i=2$ to $6$.
		}
		\label{fig:pivot}
	\end{figure*}
	
	The tree structure of the motif $F$ is crucially used both in \commHL{steps} (ii) and (iii) of the pivot chain. Namely, computing the acceptance probability $\lambda$ in step (ii) involves computing the marginal distribution $\pi^{(1)}$ on the location of the pivot from the joint distribution $\pi_{F\rightarrow \G}$. This can be done recursively due to the tree structure of $F$, admitting a particularly simple formula when $F$ is a star or a path, see Examples \ref{ex:pivot_star} and \ref{ex:pivot_path}. 
	
	In order to explain the construction of the pivot chain, we first note that the simple random walk on $\G$ with kernel $\Psi$ defined at \eqref{eq:RWkernel_G} has the following canonical stationary distribution 
	\begin{align}
		\pi_{\G}(a) := \frac{\sum_{c\in [n]}\Psi(a,c)}{\sum_{b,c\in [n]} \Psi(b,c)} \qquad a\in [n]. 
	\end{align}
	When this random walk is irreducible, $\pi_{\G}$ is its unique stationary distribution. If we draw a random homomorphism  $\x:F\rightarrow \G$ from a rooted tree motif $F=([k],A_{F})$ into a network $\G=([n],A,\alpha)$ according to the distribution $\pi_{F\rightarrow \G}$, then for each $x_{1}\in [n]$, 
	\begin{align}\label{eq:tree_pivot_marginal}
		\pi^{(1)}(x_{1})&:=\P_{F\rightarrow \G}(\x(1)=x_{1}) \\
		& = \frac{1}{\mathtt{t}(F,\G)} \sum_{1\le x_{2},\cdots, x_{k}\le n} A(x_{2^{-}},x_{2})\cdots A(x_{k^{-}},x_{k})\alpha(x_{1})\alpha(x_{2})\cdots \alpha(x_{k}).
	\end{align}
	Hence, we may use Metropolis-Hastings algorithm \citep{liu2008monte, levin2017markov} to modify the random walk kernel $\Psi$ to $P$ so that its stationary distribution becomes $\pi^{(1)}$, where 
	\begin{align}\label{eq:MH_RW_kernel_G}
		P(a,b) =
		\begin{cases}
			\Psi(a,b) \left[ \frac{\pi^{(1)}(b)\Psi(b,a)}{\pi^{(1)}(a)\Psi(a,b)} \land 1 \right] & \text{if $b\ne a$} \\
			1 - \sum_{c:c\ne a} \Psi(a,c)\left[ \frac{\pi^{(1)}(c)\Psi(c,a)}{\pi^{(1)}(a)\Psi(a,c)} \land 1 \right] & \text{if $b=a$.}
		\end{cases}
	\end{align}
	This new kernel $P$ can be executed by steps (i)-(ii) in the definition of the pivot chain. 

	\commHL{In the following examples, we consider particular instances of the pivot chain for embedding paths and stars and compute the corresponding acceptance probabilities.}
	
	\begin{ex}[Pivot chain for embedding stars]
		\label{ex:pivot_star}
		\textup{
			Let $F=([k], \mathbf{1}_{\{ (1,2),(1,3),\cdots, (1,k) \}})$ be the `star' motif centered at node $1$ (e.g., (a)-(c) in Figure \ref{fig:motif}). Embedding a star into a network gives important network observables such as the transitivity ratio and average clustering coefficient (see Example \ref{ex:avg_clustering_coeff}). In this case, the marginal distribution $\pi^{(1)}$ of the pivot in \eqref{eq:tree_pivot_marginal} simplifies into
			\begin{align}
				\pi^{(1)}(x_{1}) = \frac{\alpha(x_{1})}{\mathtt{t}(F,\G)} \left( \sum_{c\in [n]} A(x_{1},c)\alpha(c) \right)^{k-1}.
			\end{align}
			Accordingly, the acceptance probability $\lambda$ in \eqref{eq:pivot_acceptance_prob_gen} becomes 
			\begin{align}\label{eq:pivot_acceptance_prob}
				\lambda =  \left[ \frac{\alpha(b) \left( \sum_{c\in [n]} A(b,c)\alpha(c) \right)^{k-1}}{\alpha(a) \left( \sum_{c\in [n]} A(a,c)\alpha(c) \right)^{k-1}}	\frac{\Psi(b,a)}{\Psi(a,b)} \land 1\right].
			\end{align}
			For a further simplicity, suppose that the network $\G=([n],A,\alpha)$ is such that $A$ is symmetric and $\alpha\equiv 1/n$. In this case, the random walk kernel $\Psi$ and the acceptance probability $\lambda$ for the pivot chain simplify as 
			\begin{align}
				\Psi(a,b) = \frac{ A(a,b)}{ \sum_{c\in [n]}  A(a,c)  } \quad a,b\in [n], \qquad  \lambda = \left[ \frac{\left( \sum_{c\in [n]}  A(b,c) \right)  ^{k-2} } {\left( \sum_{c\in [n]}  A(a,c) \right)  ^{k-2}} \land 1 \right].
			\end{align} 
			In particular, if $F=([2],\mathbf{1}_{\{(0,1)\}})$, then $\lambda\equiv 1$ and the pivot $\x_{t}(1)$ performs the simple random walk on $\G$ given by the kernel $\Psi(a,b)\propto A(a,b)$ \commHL{without} rejection. }$\hfill \blacktriangle$
	\end{ex}

	\begin{ex}[Pivot chain for embedding paths]
		\label{ex:pivot_path}
		\textup{
			Suppose for simplicity that the node weight $\alpha$ on the network $\G=([n],A,\alpha)$ is uniform and let  $F=([k], \mathbf{1}_{\{ (1,2),(2,3),\cdots, (k-1,k) \}})$ be  a $k$-chain motif. Draw a random homomorphism  $\x:F\rightarrow \G$ from the distribution $\pi_{F\rightarrow \G}$. Then the marginal distribution $\pi^{(1)}$ of the pivot in \eqref{eq:tree_pivot_marginal} simplifies into
			\begin{align}
				\pi^{(1)}(x_{1}) & = \frac{n^{-k}}{\mathtt{t}(F,\G)} \sum_{c\in [n]} A^{k-1}(x_{1},c).
			\end{align}
			Hence the acceptance probability in step (ii) of the pivot chain becomes
			\begin{align}\label{eq:pivot_acceptance_prob_rmk}
				\lambda =  \left[ \frac{ \sum_{c\in [n]} A^{k-1}(b,c) }{ \sum_{c\in [n]} A^{k-1}(a,c) }\frac{\Psi(b,a)}{\Psi(a,b)} \land 1\right],
			\end{align}
			which involves computing powers of the matrix $A$ up to the length of the path $F$.} $\hfill \blacktriangle$
	\end{ex}	
	
	{\color{black}
		\begin{rmkk}[Comparison between the Glauber and the pivot chains]\label{rmk:pivot_chain_computational_cost}
			\normalfont
			Here we compare various aspects of the Glauber and the pivot chains. 
			\begin{description}
				\item{\textit{(Per-iteration complexity)}} The Glauber chain is much cheaper than the pivot chain per iteration for bounded degree networks. Note that in each step of the Glauber chain, the transition kernel in \eqref{eq:Glauber_kernel} can be computed in at most $O(\Delta(\G)k^{2})$ steps in general, where $\Delta(\G)$ denotes the `maximum degree' of $\G$, which we understand as the maximum degree of the edge-weighted graph $([n],A)$ as defined at \eqref{eq:def_max_deg_edge_weighted_graph}.
				
				For the pivot chain, from the computations in Examples \ref{ex:pivot_star} and \ref{ex:pivot_path}, one can easily generalize the formula for the acceptance probability $\lambda$ recursively when $F$ is a general directed tree motif. This will involve computing powers of $A$ up to the depth of the tree. More precisely, the computational cost of each step of the pivot chain is of order $\Delta(\G)^{\ell \Delta(F)}$, where $\Delta(\G)$ and $\Delta(F)$ denote the maximum degree of $\G$ and $F$ (defined at \eqref{eq:def_max_deg_edge_weighted_graph}) and $\ell$ denotes the depth of $F$. Unlike the Glauber chain, this could be exponentially large in the depth of $F$ even when $\G$ and $F$ have bounded maximum degrees. 
				
				\item{\textit{(Iteration complexity (or mixing time))}} The pivot chain requires much less iterations to mix to the stationary distribution than the Glauber chain for sparse networks. In Theorem \ref{thm:pivot_chain_mixing}, we show that the mixing time of the pivot chain is about the same as the standard random walk on networks. In Theorem \ref{thm:gen_coloring_mixing_tree}, we show that the Glauber chain mixes fast for dense networks. However, if $\G$ is sparse, we do not have a good mixing bound and we expect the chain may mix slowly.   
				
				\item{\textit{(Sampling $k$-chain motifs from sparse networks)}} For the problem of sampling $k$-chain motifs from sparse networks, we recommend using the pivot chain but with an approximate computation of the acceptance probability. For instance, taking only a bounded number of powers of the weight matrix $A$ in \eqref{eq:pivot_acceptance_prob_rmk} seems to work well in practice. 
				
				\item{\textit{(Sampling motifs from dense networks)}} For sampling general motifs (not necessarily trees) from dense networks, we recommend to use the Glauber chain. 
			\end{description}
		\end{rmkk}
	}

	\subsection{Convergence and mixing of Glauber/pivot chains}

	In this subsection, we state convergence results for the Glauber and pivot chains.

	We say a network $\G=([n],A,\alpha)$ is \textit{irreducible} if the random walk on $\G$ with kernel $\Psi$ defined at \eqref{eq:RWkernel_G} visits all nodes in $\G$ with positive probability. Note that since $\Psi(a,b)>0$ if and only if $\Psi(b,a)>0$, each proposed move $a\mapsto b$ is never rejected with probability 1. Hence $\G$ is irreducible if and only if the random walk on $\G$ with the modified kernel $P$ is irreducible. Moreover, we say $\G$ is \textit{bidirectional} if $A(i,j)>0$ if and only if $A(j,i)>0$ for all $i,j\in [n]$. Lastly, we associate a simple graph $G=([n],A_{G})$ with the network $\G$, where $A_{G}$ is its adjacency matrix given by $A_{G}(i,j) = \mathbf{1}(\min(A(i,j),A(j,i))>0)$. We call $G$ the \textit{skeleton} of $\G$. 
	
	\begin{thm}[Convergence of Glauber chain]\label{thm:stationary_measure_Glauber}
		Let $F=([k], A_{F})$ be a motif and $\G=([n],A,\alpha)$ be an irreducible network. Suppose $\mathtt{t}(F,\G)>0$ and let $(\x_{t})_{t\ge 0}$ be the Glauber chain $F\rightarrow \G$. 
		\begin{description}
			\item[(i)] $\pi_{F\rightarrow \G}$ is a stationary distribution for the Glauber chain. 
			\item[(ii)] Suppose $F$ is a rooted tree motif, $\G$ is bidirectional. \commHL{Then the Glauber chain is irreducible if and only if $\G$ is not bipartite. If $\G$ is not bipartite, then  $\pi_{F\rightarrow \G}$ is the unique stationary distribution for the Glauber chain. }
		\end{description}  
	\end{thm}
	
	The proof of Theorem \ref{thm:stationary_measure_Glauber} (i) uses a straightforward computation. For (ii), since $F$ is a rooted tree, one can argue that for the irreducibility of the Glauber chain for homomorphisms $F\rightarrow \G$, it suffices to check irreducibility of the Glauber chain $\x_{t}:K_{2}\rightarrow \G$, where $K_{2}$ is the 2-chain motif, which has at most two communicating classes depending on the `orientation' of $x_{t}$. Recall that $\G$ is not bipartite if and only if its skeleton $G$ contains an odd cycle. An odd cycle in $G$ can be used to construct a path between arbitrary two homomorphisms $K_{2}\rightarrow \G$. See Appendix \ref{section:proofs_mixing} for more details.
	
	We also have the corresponding convergence results for the pivot chain in Theorem \ref{thm:stationary_measure_pivot} below.
	
	\begin{thm}[Convergence of pivot chain]\label{thm:stationary_measure_pivot}
		Let $\G=([n],A,\alpha)$ be an irreducible network with $A(i,j)>0$ for some $j\in [n]$ for each $i\in [n]$. $F = ([k], A_{F})$ be a rooted tree motif. Then pivot chain $F\rightarrow \G$ is irreducible with unique stationary distribution $\pi_{F\rightarrow \G}$.
	\end{thm}

	Since both the Glauber and pivot chains evolve in the finite state space $[n]^{[k]}$, when given the irreducibility condition, both chains converge to their unique stationary distribution $\pi_{F\rightarrow \G}$. Then the Markov chain ergodic theorem implies the following corollary. 
	
	\begin{thm}[Computing stationary mean by ergodic mean]\label{thm:observable_time_avg}
		Let $F=([k],A_{F})$ be a rooted tree motif and $\G=([n],A,\alpha)$ be an irreducible network. Let $g:[n]^{[k]}\rightarrow \mathbb{R}^{d}$ be any function for $d\ge 1$. Let $\,\x:[k]\rightarrow [n]$ denote a random homomorphism $F\rightarrow \G$ drawn from $\pi_{F\rightarrow \G}$.    
		\begin{description}
			\item[(i)] If $(\x_{t})_{t\ge 0}$ denotes the pivot chain $F\rightarrow \G$, then 
			\begin{equation}\label{eq:conditional_density_LLN}
				\mathbb{E}[ g(\x) ] = \lim_{N\rightarrow \infty}\frac{1}{N}\sum_{t=1}^{N} g(\x_{t}).
			\end{equation}
			\item[(ii)]  If $\G$ is bidirectional and its skeleton is not bipartite, then \eqref{eq:conditional_density_LLN} also holds for the Glauber chain $(\x_{t})_{t\ge 0}:F\rightarrow \G$. 
		\end{description}  
	\end{thm}

	\vspace{0.3cm}


	
	Next, we address the question of how long we should run the Markov chain Monte Carlo in order to get a precise convergence to the target measure $\pi_{F\rightarrow \G}$. Recall that the \textit{total deviation distance} between two probability distributions $\mu,\nu$ on a finite set $\Omega$ is defined by 
	\begin{equation}\label{eq:def_TV_distance}
		\lVert\mu-\nu \rVert_{\text{TV}} := \frac{1}{2}\sum_{x\in \Omega} |\mu(x)-\nu(x)|.
	\end{equation}  
	If $(X_{t})_{t\ge 0}$ is any Markov chain on finite state space $\Omega$ with transition kernel $P$ and unique starionay distribution $\pi$, then its \textit{mixing time} $t_{mix}$ is defined to be the function 
	\begin{equation}
		t_{mix}(\eps) = \inf\left\{t\ge 0\,:\,   \max_{x\in \Omega}\lVert P^{t}(x, \cdot)-\pi \rVert_{\text{TV}}\le \eps   \right\}. 
	\end{equation}

	In Theorems \ref{thm:gen_coloring_mixing_tree} and \ref{thm:pivot_chain_mixing} below, we give bounds on the mixing times of the Glauber and pivot chains when the underlying motif $F$ is a tree. For the Glauber chain, let $\x:F\rightarrow \G$ be a homomorphism and fix a node $j\in [k]$. Define a probability distribution $\mu_{\x,j}$ on $[n]$ by 
	\begin{align}
		\mu_{\x,i}(b)=\frac{\left( \prod_{j\ne i} A(\x(j),b)^{A_{F}(j,i)} A(b,\x(j))^{A_{F}(i,j)}    \right)  A(b,b)^{A_{F}(i,i)} \alpha(b) }{\sum_{1\le c \le n} \left( \prod_{j\ne i} A(\x(j),c)^{A_{F}(j,i)} A(c,\x(j))^{A_{F}(i,j)}    \right)  A(c,c)^{A_{F}(i,i)} \alpha(c)  },
	\end{align}
	This is the conditional distribution that the Glauber chain uses to update $\x(j)$. 
	
	For each integer $d\ge 1$ and network $\G=([n],A,\alpha)$, define the following quantity 
	\begin{equation}\label{eq:def_glauber_mixing_constant}
		c(d,\G) := \max_{\substack{\x, \x':S_{d}\rightarrow \G \\ \text{$\x\sim\x'$ and $\x(1)= \x'(1)$} }} \left( 1 - 2d\lVert \mu_{\x,1} - \mu_{\x',1} \rVert_{\text{TV}}\right),  
	\end{equation}
	where $S_{d}=([d+1],E)$ is the star with $d$ leaves where node $1$ is at the center, and $\x\sim \x'$ means that they differ by at most one coordinate. For a motif $F=([k],A_{F})$, we also recall its \textit{maximum degree} $\Delta(F)$ defined in \eqref{eq:def_max_deg_edge_weighted_graph}.

	\begin{thm}[Mixing time of Glauber chain]\label{thm:gen_coloring_mixing_tree}
		Suppose $F=([k],A_{F})$ is a rooted tree motif and $\G$ is an irreducible and bidirectional network. Further, assume that the skeleton $G$ of $\G$ contains an odd cycle. If $c(\Delta(F),\G)>0$, then the mixing time $t_{mix}(\eps)$ of the Glauber chain $(\x_{t})_{t\ge 0}$ of homomorphisms $F\rightarrow \G$ satisfies 
		\begin{equation}
			t_{mix}(\eps) \le \lceil 2c(\Delta(F),\G) k\log (2k/\eps)(\diam(G)+1) \rceil.
		\end{equation}
	\end{thm}

	On the other hand, we show that the pivot chain mixes at the same time that the single-site random walk on network $\G$ does. An important implication of this fact is that the mixing time of the pivot chain does not depend on the size of the motif. However, the computational cost of performing each step of the pivot chain does increase in the size of the motif (see Remark \ref{rmk:pivot_chain_computational_cost}). 

	It is well-known that the mixing time of a random walk on $\G$ can be bounded by the absolute spectral gap of the transition kernel in \eqref{eq:MH_RW_kernel_G} (see \cite[Thm. 12.3, 12.4]{levin2017markov}). Moreover, a standard coupling argument shows that the mixing time is bounded above by the meeting time of two independent copies of the random walk. Using a well-known cubic bound on the meeting times due to \cite{coppersmith1993collisions}, we obtain the following result.
	
	\begin{thm}[Mixing time of pivot chain]\label{thm:pivot_chain_mixing}
		Let $F=([k],E_{F})$ be a directed rooted tree and $\G=([n],A,\alpha)$ be an irreducible network. Further assume that for each $i\in [n]$, $A(i,j)>0$ for some $j\in [n]$.   Let $P$ denote the transition kernel of the random walk on $\G$ defined at \eqref{eq:MH_RW_kernel_G}. Then the mixing time $t_{mix}(\eps)$ of the pivot chain $(\x_{t})_{t\ge 0}$ of homomorphisms $F\rightarrow \G$ satisfies the following. 
		\begin{description}
			\item[(i)] Let $t^{(1)}_{mix}(\eps)$ be the mixing time of the pivot with kernel $P$. Then 
			\begin{align}
				t_{mix}(\eps) = t^{(1)}_{mix}(\eps).
			\end{align}
			
			\item[(ii)] Let $\lambda_{\star}$ be the eigenvalue of $P$ with largest modulus that is less than $1$. Then 
			\begin{equation}
				\frac{\lambda_{\star} \log (1/2\eps)}{1-\lambda_{\star}} \le t_{mix}(\eps) \le \max_{x\in [n]}\frac{\log( 1/\alpha(x)  \eps)}{1-\lambda_{\star}}.
			\end{equation} 
			
			\item[(iii)] Suppose $n\ge 13$, $A$ is the adjacency matrix of some simple graph, and $\alpha(i)\propto \deg(i)$ for each $i\in [n]$. Then 
			\begin{align}
				t_{mix}(\eps) \le \log_{2}(\eps^{-1})\left( \frac{4}{27}n^{3}+\frac{4}{3}n^{2} + \frac{2}{9}n - \frac{296}{27}\right).
			\end{align}
		\end{description}
	\end{thm}
	
	\vspace{0.2cm}
	\subsection{Concentration and statistical inference}
	
	Suppose $(\x_{t})_{t\ge 0}$ is the pivot chain of homomorphisms $F\rightarrow \G$, and let $g:[k]^{[n]}\rightarrow \mathbb{R}^{d}$ be a function for some $d\ge 1$. In the previous subsection, we observed that various observables on the network $\G$ can be realized as the expected value $\mathbb{E}[g(\x)]$ under the stationary distribution $\pi_{F\rightarrow \G}$, so according to Corollary \ref{cor:time_avg_observables}, we can approximate them by time averages of increments $g(\x_{t})$ for a suitable choice of $g$. A natural question to follow is that if we take the time average for the first $N$ steps, is it possible to \commHL{infer the stationary expectation} $\mathbb{E}[g(\x)]$? 
	
	The above question can be addressed by applying  McDiarmid's inequality for Markov chains (see, e.g., \cite[Cor. 2.11]{paulin2015concentration}) together with the upper bound on the mixing time of pivot chain provided in Theorems \ref{thm:pivot_chain_mixing}.
	
	\begin{thm}[Concentration bound for real-valued observables]\label{thm:McDiarmids}
		Let $F=([k],E_{F})$, $\G=([n],A,\alpha)$, $(\x_{t})_{t\ge 0}$, and $t^{(1)}_{mix}(\eps)$ be as in Theorem \ref{thm:pivot_chain_mixing}. Let $g:[k]^{[n]}\rightarrow \mathbb{R}$ be any functional. Then for any $\delta>0$, 
		\begin{align}\label{eq:thm:McDiarmids_1}
			\mathbb{P}\left( \left| \mathbb{E}_{\pi_{F\rightarrow \G}}[g(\x)] - \frac{1}{N}\sum_{t=1}^{N} g(\x_{t})  \right| \ge \delta \right) <  2\exp\left( \frac{-2\delta^{2}N}{9 t_{mix}^{(1)}(1/4)} \right).
		\end{align}
	\end{thm}

	A similar result for the Glauber chain (with $t_{mix}^{(1)}(1/4)$ at \eqref{eq:thm:McDiarmids_1} replaced by $t_{mix}(1/4)$) can be derived from the mixing bounds provided in Theorem \ref{thm:gen_coloring_mixing_tree}.
	\begin{rmkk}
		One can reduce the requirement for running time $N$ in Theorem \ref{thm:McDiarmids} by a constant factor in two different ways. First, if the random walk of pivot on $\G$ exhibits a cutoff, then the factor of $9$ in (\ref{eq:thm:McDiarmids_1}) can be replaced by 4 (see \cite[Rmk. 2.12]{paulin2015concentration}). Second, if we take the partial sum of $g(\x_{t})$ after a `burn-in period' a multiple of mixing time of the pivot chain, then thereafter we only need to run the chain for a multiple of the relaxation time $1/(1-\lambda_{\star})$ of the random walk of pivot (see \cite[Thm. 12.19]{levin2017markov}).
	\end{rmkk}
	
	Next, we give a concentration inequality for the vector-valued partial sums process. This will allow us to construct confidence intervals for CHD profiles and motif transforms. The key ingredients are the use of burn-in period as in \cite[Thm. 12.19]{levin2017markov} and a concentration inequality for vector-valued martingales \citep{hayes2005large}. 
	
	\begin{thm}[Concentration bound for vector-valued observables]\label{thm:vector_concentration}
		Suppose $F=([k],A_{F})$, $\G=([n],A,\alpha)$, $(\x_{t})_{t\ge 0}$, and $t^{(1)}_{mix}(\eps)$ be as in Theorem \ref{thm:pivot_chain_mixing}. Let $\mathcal{H}$ be any Hilbert space and let $g:[n]^{[k]}\rightarrow \mathcal{H}$ be any function such that $\lVert g \rVert_{\infty}\le 1$. Then for any $\eps,\delta>0$, 
		\begin{align}\label{eq:thm:vector1}
			\mathbb{P}\left( \left\Vert \mathbb{E}_{\pi_{F\rightarrow \G}}[g(\x)] - \frac{1}{N}\sum_{t=1}^{N} g(\x_{r+t})  \right\Vert \ge \delta \right) \le 2 \exp\left( 2- \frac{-\delta^{2}N}{2} \right) + \eps,
		\end{align}
		provided $r\ge t_{mix}^{(1)}(\eps)$.
	\end{thm}

	\section{Network observables based on motif sampling}
	
	\commHL{In Section \ref{section:motif_sampling_MCMC}, we  introduced the motif sampling problem and proposed MCMC algorithms for the efficient computational solution of this problem. In that section we also established various theoretical guarantees. Specifically, we have shown that the stationary expectation of an arbitrary vector-valued function of a random homomorphism can be computed through an ergodic average along MCMC trajectories (see Theorems \ref{thm:McDiarmids} and \ref{thm:vector_concentration}). 
		In this section, we introduce specific network observables that can be efficiently computed in this way and also establish their stability properties}.
	
	\subsection{Definitions and computation}
	\label{section:observables_def}
	In this section, we introduce \commHL{four} network observables based on the random embedding of motif $F$ into a network $\G$. The first one is a conditional version of the well-known homomorphism density \cite{lovasz2012large}.
	
	\begin{dfn}[Conditional homomorphism density]
		Let $\G=([n],A,\alpha)$ be a network and fix two motifs $H=([k],A_{H})$ and $F=([k],A_{F})$. Let $H+F$ denote the motif $([k], A_{H}+A_{F})$. We define the \textit{conditional homomorphism density} (CHD) of $H$ in $\G$ given $F$ by 
		\begin{equation}\label{eq:def_cond.hom.dens.}
			\mathtt{t}(H,\G|F) = \frac{\mathtt{t}(H+ F,\G)}{\mathtt{t}(F,\G)},
		\end{equation} 
		which is set to zero when the denominator is zero.
	\end{dfn}	
	
	When $\G$ is a simple graph with uniform node weight, the above quantity equals the probability that all edges in $H$ are preserved by a uniform random homomorphism $\x:F\rightarrow \G$. \commHL{As a notable special case, we 
		describe a quantity closely related to the the average clustering coefficient as a conditional homomorphism density.}

	\begin{ex}[Average clustering coefficient]\label{ex:avg_clustering_coeff}
		\textup{A notable special case is when $F$ is the wedge motif $W_{3}=([3],\mathbf{1}_{\{(1,2),(1,3)\}})$ (see Figure \ref{fig:motif} (e)) and $H = ([3], \mathbf{1}_{\{(2,3)\}})$  and $\G$ is a simple graph. Then $\mathtt{t}(H,\G\,|\, W_{3})$ is the conditional probability that a random sample of three nodes $x_{1},x_{2},x_{3}$ in $\G$ induces a copy of the triangle motif $K_{3}$, given that there are edges from $x_{1}$ to each of $x_{2}$ and $x_{3}$ in $\G$. If all three nodes are required to be distinct, such a conditional probability is known as the \textit{transitivity ratio} \citep{luce1949method}. }
		
		\textup{A similar quantity with different averaging leads to the \textit{average clustering coefficient}, which was introduced to measure how a given network locally resembles a complete graph and used to define small-world networks by \citet{watts1998collective}. Namely, we may write 
			\begin{align}\label{eq:avg_clustering_coff}
				\mathtt{t}(H,\G\,|\, W_{3}) = \sum_{x_{1}\in [n]}  \frac{\sum_{x_{2},x_{3}\in [n]}A(x_{1},x_{2})A(x_{2},x_{3})A(x_{1},x_{3}) \alpha(x_{2})\alpha(x_{3}) }{ \left( \sum_{x_{2}\in [n]}A(x_{1},x_{2}) \alpha(x_{2}) \right)^{2}} \frac{\alpha(x_{1})}{\sum_{x_{1}\in [n]}\alpha(x_{1})}.
			\end{align} 
			If $\G$ is a simple graph with uniform node weight $\alpha\equiv 1/n$, then we can rewrite the above equation as 
			\begin{align}\label{eq:cond.hom.pivot2}
				\mathtt{t}(H,\G\,|\, W_{3}) = \sum_{x_{1}\in [n]} \frac{ \#(\text{edges between neighbors of $x_{1}$ in $\G$}) }{\deg_{\G}(x_{1})(\deg_{\G}(x_{1})-1)/2} \frac{\deg_{\G}(x_{1})-1}{n\deg_{\G}(x_{1})}.
			\end{align} 
			If the second ratio in the above summation is replaced
			by $1/n$, then it becomes the average clustering
			coefficient of $\G$ \citep{watts1998collective}. Hence
			the conditional homomorphism density
			$\mathtt{t}(H,\G\,|\, W_{3})$ can be regarded as a
			a variant of the generalized average clustering coefficient, which lower bounds the average clustering coefficient of $\G$ when it is a simple graph. We also remark that a direct generalization of the average clustering coefficient in terms of higher-order cliques was introduced recently by \cite{yin2018higher}. See also \cite{cozzo2015structure} for a related discussion for multiplex networks. }$\hfill \blacktriangle$
	\end{ex}
	
	Motivated by the connection between the conditional homomorphism density and the average clustering coefficient discussed above, we introduce the following generalization of the average clustering coefficient. (See Figures \ref{fig:intro_sim_MACC} and \ref{fig:FB_MACC} for examples.)
	
	\begin{dfn}[Matrix of Average Clustering Coefficients]
		\label{def:MACC}
		Let $\G=([n],A,\alpha)$ be a network and fix a motif $F=([k],A_{F})$. For each $1\le i\le j \le k$, let $H_{ij}=([k], A_{F}+\mathbf{1}_{\{(i,j)\}}\mathbf{1}(A_{F}(i,j)=0))$ be the motif obtained by `adding' the edge $(i,j)$ to $F$. We define the \textit{Matrix of Average Clustering Coefficient} (MACC) of $\G$ given $F$ by the $k\times k$ matrix whose $(i,j)$ coordinate is given by 
		\begin{equation}\label{eq:def_cond.hom.dens.}
			\mathtt{MACC}(\G|F)(i,j) = \frac{\mathtt{t}( H_{ij},\G)}{\mathtt{t}(F,\G)},
		\end{equation} 
		which is set to zero when the denominator is zero.
	\end{dfn}

	Next, instead of looking at the conditional homomorphism density of $H$ in $\G$ given $F$ at a single scale, we could look at how the conditional density varies at different scales as we threshold $\G$ according to a parameter $t\ge 0$. Namely, we draw a random homomorphism $\x:F\rightarrow \G$, and ask if all the edges in $H$ have weights $\ge t$ in $\G$. This naturally leads to the following function-valued observable.
	
	\begin{dfn}[CHD profile]
		Let $\G=([n],A,\alpha)$ be a network and fix two motifs $H=([k],A_{H})$ and $F=([k],A_{F})$. We define the \textit{CHD (Conditional Homomorphism Density) profile} of a network $\G$ for $H$ given $F$ by the function $\mathtt{f}(H,\G\,|\, F):[0,1]\rightarrow [0,1]$, 
		\begin{align}
			\mathtt{f}(H,\G\,|\, F)(t)&=\P_{F\rightarrow \G}\left( \min_{1\le i,j\le k} A(\x(i),\x(j))^{A_{H}(i,j)}\ge t  \right), \label{eq:def_CHDF}
		\end{align}
		where $\x:F\rightarrow \G$ is a random embedding drawn from the distribution $\pi_{F\rightarrow \G}$ defined at \eqref{eq:def_embedding_F_N}. 
	\end{dfn}
	

	\commHL{We give examples of CHD profiles involving two-armed paths, singleton, and self-loop motifs. }
	
	\begin{ex}[CHD profiles involving two-armed paths]\label{ex:motifs_Ex}
		\normalfont
		
		For integers $k_{1},k_{2}\ge 0$, we define a \textit{two-armed path motif}  $F_{k_{1},k_{2}}=(\{0,1,\cdots, k_{1}+k_{2}\}, \mathbf{1}(E))$ where its set $E$ of directed edges are given by  
		\begin{align}
			E = \left\{ \begin{matrix} (0,1),(1,2),\cdots, (k_{1}-1,k_{1}),   \\ 
				(0,k_{1}+1), (k_{1}+1, k_{1}+2), \cdots, (k_{1}+k_{2}-1, k_{1}+k_{2}) \end{matrix} \right\}.
		\end{align}
		This is also the rooted tree consisting of two directed paths of lengths $k_{1}$ and $k_{2}$ from the root 0. Also, we denote $H_{k_{1},k_{2}}=(\{0,1,\cdots, k_{1}+k_{2}\}, \mathbf{1}_{\{ (k_{1}, k_{1}+k_{2})\}})$. This is the motif on the same node set as $F_{k_{1},k_{2}}$ with a single directed edge between the ends of the two arms. (See Figure \ref{fig:cycle_motif}.)

		\begin{figure*}[h]
			\centering
			\includegraphics[width=0.85\linewidth]{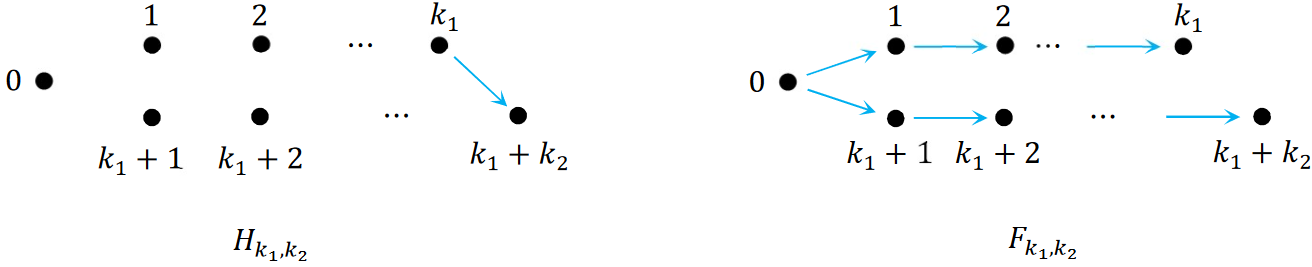}
			\caption{ Plots of the motif $H_{k_1,k_2}$, which contains a single directed edge from $k_1$ to $k_1+k_2$ (left), and the two-armed path motif $F_{k_1, k_2}$ on the right.
			}
			\label{fig:cycle_motif}
		\end{figure*}
		
		When $k_{1}=k_{2}=0$, then $F_{0,0}$ and $H_{0,0}$ become the `singleton motif' $([0],\mathbf{0})$ and the `self-loop motif' $([0], \mathbf{1}_{(0,0)})$, \commHL{respectively.} In this case the corresponding homomorphism and conditional homomorphism densities have simple expressions involving only the diagonal entries of the edge weight matrix of the network. Namely, for a given network $\G=([n],A,\alpha)$, we have 
		\begin{align}
			\mathtt{t}(H_{0,0}, \G) = \sum_{k=1}^{n} A(k,k) \alpha(k),\qquad \mathtt{t}(F_{0,0}, \G) = \sum_{k=1}^{n} \alpha(k) = 1.
		\end{align} 
		The former is also the weighted average of the diagonal entries of $A$ with respect to the node weight $\alpha$. For the conditional homomorphism densities, observe that 
		\begin{align}
			\mathtt{t}(H_{0,0},\G\,|\, F_{0,0}) = \sum_{k=1}^{n} A(k,k) \alpha(k),\qquad \mathtt{t}(H_{0,0},\G\,|\, H_{0,0}) = \frac{\sum_{k=1}^{n} A(k,k)^{2} \alpha(k) } {\sum_{k=1}^{n} A(k,k) \alpha(k)}.
		\end{align}
		The latter is also the ratio between the first two moments of the diagonal entries of $A$. The corresponding CHD profile is given by 
		\begin{align}
			\mathtt{f}(H_{0,0},\G\,|\, F_{0,0})(t) &= \sum_{k=1}^{n} \mathbf{1}(A(k,k)\ge t)  \alpha(k), \\
			\mathtt{f}(H_{0,0},\G\,|\, H_{0,0})(t) &= \frac{\sum_{k=1}^{n} \mathbf{1}(A(k,k)\ge t) A(k,k) \alpha(k)}{\sum_{k=1}^{n} A(k,k) \alpha(k)}.
		\end{align}
		\commHL{The above two} quantities can be interpreted as the probability that the self-loop intensity $A(k,k)$ is at least $t$, when $k\in [n]$ is chosen with probability proportional to $\alpha(k)$ or $A(k,k)\alpha(k)$, respectively.  $\hfill \blacktriangle$
	\end{ex}
	
	\vspace{0.1cm}
	Lastly, we define network-valued observables from motif sampling. Recall that motif sampling gives the $k$-dimensional probability measure $\pi_{F\rightarrow \G}$ on the set $[n]^{[k]}$. Projecting this measure onto the first and last coordinates gives a probability measure on $[n]^{\{1,k\}}$. This can be regarded as \commHL{the weight matrix} $A^{F}:[n]^{2}\rightarrow [0,1]$ of another network $\G^{F}:=([n],A^{F},\alpha)$. The precise definition is given below. 
	
	\begin{dfn}[Motif transform]
		Let $F=([k],A_{F})$ be a motif for some $k\ge 2$ and $\G=([n],A,\alpha)$ be a network. The \textit{motif transform} of $\G$ by $F$ is the network $\G^{F}:=([n],A^{F},\alpha)$, where  
		\begin{equation}
			A^{F}(x,y) = \P_{F\rightarrow \G}\left( \x(1)=x,\, \x(k)=y \right),
		\end{equation}    
		where $\x:F\rightarrow \G$ is a random embedding drawn from the distribution $\pi_{F\rightarrow \G}$ defined at \eqref{eq:def_embedding_F_N}.
	\end{dfn}

	Motif transforms can be used to modify a given network so that certain structural defects are remedied without perturbing the original network \commHL{aggressively}. For instance, suppose $\G$ consists two large cliques $C_{1}$ and $C_{2}$ connected by a thin path $P$. When we perform the single-linkage clustering on $\G$, it will perceive $C_{1}\cup   P \cup C_{2}$ as a single cluster, even though the linkage $P$ is not significant. To overcome such an issue, we could instead perform single-linkage clustering on the motif transform $\G^{F}$ where $F$ is a triangle. Then the thin linkage $P$ is suppressed by the transform, and the two cliques $C_{1}$ and $C_{2}$ will be detected as separate clusters. See Example \ref{ex:barbell} for more details. 
	
	\begin{rmkk}
		\normalfont
		Transformations of networks analogous to motif transforms have been studied in the context of clustering of metric spaces and networks by \cite{classif,representable-nets,excisive-nets} and \cite{clust-motivic}.
	\end{rmkk}

	Next, we discuss how to compute the network observables we introduced in this section. Given a motif $F$ and network $\G$, the CHD, CHD profile, and motif transform are all defined by the expectation of a suitable function of a random homomorphism $\x:F\rightarrow \G$. While computing this expectation directly is computationally challenging, we can efficiently compute them by taking time averages along MCMC trajectory $\x_{t}:F\rightarrow \G$ of a dynamic embedding (see Theorems \ref{thm:McDiarmids}, \ref{thm:vector_concentration}, and \ref{thm:McDiarmids}). This is more precisely stated in the following corollary.
	
	
	\begin{cor}\label{cor:time_avg_observables}
		Let $F=([k],A_{F})$ be a rooted tree motif, $H=([k],A_{H})$ another motif, and $\G=([n],A,\alpha)$ an irreducible network. Let $(\x_{t})_{t\ge 0}$ be the pivot chain $F\rightarrow \G$. Then the followings hold: 
		\begin{align}
			\mathtt{t}(H,\G|F) &= \lim_{N\rightarrow \infty} \frac{1}{N} \sum_{t=1}^{N}  \prod_{1\le i,j\le k} A(\x_{t}(i),\x_{t}(j))^{A_{H}(i,j)}, \label{eq:cond_hom_ergodic} \\
			\mathtt{f}(H,\G\,|\, F) (t) &= \lim_{N\rightarrow \infty} \frac{1}{N} \sum_{t=1}^{N} \prod_{1\le i,j\le k} \mathbf{1}\left(A(\x_{t}(i),\x_{t}(j))^{A_{H}(i,j)}\ge t\right) \quad t\in [0,1], \label{eq:profile_ergodic}\\
			\mathtt{t}(H,\G\,|, F) A^{H} &= \lim_{N\rightarrow \infty} \frac{1}{N} \sum_{t=1}^{N} \left( \prod_{1\le i,j\le k}  A(\x_{t}(i),\x_{t}(j))^{A_{H}(i,j)}\right) E_{\x_{t}(1),\x_{t}(k)}, \label{eq:transform_ergodic}
		\end{align}
		where $E_{i,j}$ denotes the $(n\times n)$ matrix with zero entries except 1 at $(i,j)$ entry. Furthermore, $\G$ is bidirectional and its skeleton contains an odd cycle, then the above equations also hold for the Glauber chain $(\x_{t})_{t\ge 0}:F\rightarrow \G$. 
	\end{cor}

	\begin{rmkk}
		\normalfont
		When we compute $A^{H}$ using \eqref{eq:transform_ergodic}, we do not need to approximate the conditional homomorphism density $\mathtt{t}(H,\G,|\, F)$ separately. Instead, we compute the limiting matrix on the right-hand side of (\ref{eq:transform_ergodic}) and normalize by its 1-norm so that $\lVert A^{H} \rVert_{1}=1$. 
	\end{rmkk}

	\begin{rmkk}
		\normalfont
		\commHL{We emphasize that all the network observables that we introduced in this section can be expressed as the expected value of some function of a random homomorphism $F\rightarrow \G$, and that any  network observable defined in 
			this manner
			can be computed efficiently by taking suitable time averages along MCMC trajectory of homomorphisms $\x_{t}:F\rightarrow \G$ as in Corollary \ref{cor:time_avg_observables}. It would be interesting to investigate other  network observables that \commHL{can be expressed as the expectation of some function of a random homomorphism $F\rightarrow \G$.} } 
	\end{rmkk}

	\vspace{0.3cm}
	
	\section{Stability inequalities}
	\label{section:stability_inequalities}
	
	\commHL{In this section, we establish stability properties of the network observables we introduced in Section \ref{section:observables_def}. Roughly speaking, our aim is to show that these observable change little when we change the underlying network little. In order to do so, we need to introduce a notion of distance between networks.  }
	
	We introduce two commonly used notions of distance between networks as viewed as `graphons'. A \textit{kernel} is a measurable integrable function $W:[0,1]^{2}\rightarrow [0,\infty)$. We say a kernel $W$ is a \textit{graphon} if it takes values from $[0,1]$. Note that we do not require the kernels and graphons to be symmetric, in contrast to the convention use \commHL{in} \cite{lovasz2012large}. For a given network $\G=([n],A,\alpha)$, we define a `block kernel' $U_{\G}:[0,1]^{2}\rightarrow [0,1]$ by 
	\begin{equation}\label{eq:def_block_kernel}
		U_{\G}(x,y) = \sum_{1\le i,j \le n } A(i,j) \mathbf{1}(x\in I_{i}, y\in I_{j}),
	\end{equation} 
	where $[0,1]=I_{1}\sqcup I_{2}\sqcup \cdots \sqcup I_{n}$ is a partition such that each $I_{i}$ is an interval with Lebesgue measure $\mu(I_{i})=\alpha(i)$. (For more discussion on kernels and graphons, see \citep{lovasz2012large}.)

	For any integrable function $W:[0,1]^{2}\rightarrow \mathbb{R}$, we define its $p$-norm by 
	\begin{equation}
		\lVert W \rVert_{p} = \left( \int_{0}^{1}\int_{0}^{1} |W(x,y)|^{p}\,dx\,dy \right)^{1/p},
	\end{equation}
	for any real $p\in (0,\infty)$, and its \textit{cut norm} by 
	\begin{equation}
		\lVert W \rVert_{\square} =  \sup_{A,B\subseteq [0,1]} \left| \int_{A}\int_{B} W(x,y) \,dx\,dy \right|,
	\end{equation}
	where the supremum is taken over Borel-measurable subsets of $A,B\subseteq [0,1]$.
	Now for any two networks $\G_{1}$ and $\G_{2}$, we define their \textit{$p$-distance} by 
	\begin{equation}\label{eq:ntwk_p_distance}
		\delta_{p}(\G_{1},\G_{2}) = \inf_{\varphi} \lVert U_{\G_{1}} - U_{\varphi(\G_{2})} \rVert_{p},
	\end{equation}
	where the infimum is taken over all bijections $\varphi:[n]\rightarrow [n]$ and $\varphi(\G_{2})$ is the network $([n],A^{\varphi},\alpha\circ\varphi)$, $A^{\varphi}(x,y) = A(\varphi(x),\varphi(y))$. Taking infimum over $\varphi$ ensures that the similarity between two networks does not depend on relabeling of nodes. We define \textit{cut distance} between $\G_{1}$ and $\G_{2}$ similarly and denote it  by $\delta_{\square}(\G_{1},\G_{2})$. We emphasize that the cut norm and cut distance are well-defined for possibly asymmetric kernels. 
	
	The cut distance is \commHL{more} conservative than the 1-norm in the sense that 
	\begin{equation}
		\delta_{\square}(W_{1},W_{2})\le \delta_{1}(W_{1},W_{2})
	\end{equation}
	for any two kernels $W_{1}$ and $W_{2}$. This follows from the fact that 
	\begin{equation}
		\lVert W \rVert_{\square} \le \lVert\, |W| \,\rVert_{\square}=   \lVert W \rVert_{1}. 
	\end{equation}
	for any kernel $W$.

	Now we state stability inequalities for the network observables we introduced in Section \ref{section:observables_def} in terms of kernels and graphons. The homomorphism density of a motif $F=([k],A_{F})$ in a kernel $U$ is defined by (see, e.g.,  \citet[Subsection 7.2]{lovasz2006limits}) 
	\begin{align}\label{eq:hom_density_graphon}
		\mathtt{t}(F,U) = \int_{[0,1]^{k}} \prod_{1\le i,j\le k} U(x_{i},x_{j})^{A_{F}(i,j)}\,dx_{1}\cdots dx_{k}.
	\end{align} 
	For any other motif $H=([k],A_{H})$, we define the conditional homomorphism density of $H$ in $U$ given $F$ by $\mathtt{t}(H,U\,|\, F)=\mathtt{t}(H+F,U)/\mathtt{t}(F,U)$, where $F+E = ([k], A_{E}+A_{F})$ and we set $\mathtt{t}(H,U\,|\, F)=0$ if $\mathtt{t}(F,U)=0$. It is easy to check that the two definitions of conditional homomorphism density for networks and graphons agree, namely  $\mathtt{t}(H,\G\,|\, F)=\mathtt{t}(H,U_{\G}\,|\, F)$. Also, CHD  for kernels is defined analogously to \eqref{eq:def_CHDF}. {That is, we define the \textit{CHD  profile} of a kernel $U:[0,1]^{2}\rightarrow [0,1]$ for $H$ given $F$ by the function $\mathtt{f}(H,U\,|\, F):[0,1]\rightarrow [0,1]$,
		\begin{align}
			\mathtt{f}(H,U\,|\, F)(t)&= \int_{[0,1]^{k}} \mathbf{1}\left( \min_{1\le i,j\le k} U(\x(i),\x(j))^{A_{H}(i,j)}\ge t   \right) \, dx_{1},\dots,dx_{k}.
		\end{align}
	}
	Finally, we define the motif transform $U^{F}:[0,1]^{2}\rightarrow [0,\infty)$ of a kernel $U$ by a motif $F=([k],A_{F})$ for $k\ge 2$ by 
	\begin{align}
		U^{F}(x_{1},x_{k}) = \frac{1}{\mathtt{t}(F,U)} \int_{[0,1]^{k-2}} \prod_{1\le i,j\le k} U(x_{i},x_{j})^{A_{F}(i,j)}\,dx_{2}\cdots dx_{k-1}.
	\end{align}

	The well-known stability inequality for homomorphism densities is due to  \cite{lovasz2006limits}, which reads 
	\begin{equation}\label{ineq:countinglemma1}
		| \mathtt{t}(F,U)-\mathtt{t}(F,W) | \le |E_{F}| \cdot \delta_{\square}(U,W)
	\end{equation} 
	for any two graphons $U,W:[0,1]^{2}\rightarrow [0,1]$ and a motif $F=([k],E_{F})$. A simple application of this inequality shows that conditional homomorphism densities are also stable with respect to the cut distance up to normalization.
	
	\begin{prop}\label{prop:conditioned_counting_lemma}
		Let $H=([k],A_{H})$ and $F=([k],A_{F})$ be motifs such that $H+F=([k], A_{H}+A_{F})$ is simple. Let $U,V:[0,1]^{2}\rightarrow [0,1]$ be graphons. Then 
		\begin{equation}\label{ineq:conditioned_counting_lemma}
			|\mathtt{t}(H,U|F) - \mathtt{t}(H,W|F)|\le  \frac{2|E_{H}| \cdot \delta_{\square}(U,W)}{\max( \mathtt{t}(F,U), \mathtt{t}(F,W))} .
		\end{equation}
	\end{prop}
	As a corollary, this also yields a similar stability inequality for the MACC (see Definition \ref{def:MACC}). A similar argument shows that motif transforms are also stable with respect to cut distance.  
	\begin{prop}\label{prop:stability_Ftransform}
		Let $F=([k],A_{F})$ be a simple motif and let $U,W:[0,1]^{2}\rightarrow [0,1]$ be  graphons. Then 
		\begin{equation}
			\delta( U^{F}, W^{F})_{\square} \le  \left( 1+ \frac{1}{\max(\mathtt{t}(F,U), \mathtt{t}(F,W))}   \right) |E(F)|\cdot \delta_{\square}( U,W )
		\end{equation}
	\end{prop}
	
	\commHL{Lastly, we state a stability inequality for the CHD profiles in Theorem \ref{thm:counting_filtration}. While the proof of Propositions \ref{prop:conditioned_counting_lemma} and \ref{prop:stability_Ftransform} is relatively straightforward using the stability inequality for the homomorphism density \eqref{ineq:countinglemma1}, the proof of Theorem \ref{thm:counting_filtration} is more involved and requires new analytical tools. The main idea is to define a notion of cut distance between `filtrations of graphons' and to show that this new distance interpolates between the cut distance and the 1-norm-distance (see Proposition \ref{prop:filtration_cutnormbound}). See Appendix \ref{section:proofs_stability} for more details.}
	
	\begin{thm}\label{thm:counting_filtration}
		Let $H=([k],A_{H})$ and $F=([k],A_{F})$ be simple motifs such that $H+F=([k], A_{H}+A_{F})$ is simple. Then for any graphons $U,W:[0,1]^{2}\rightarrow [0,1]$, 
		\begin{equation}
			\lVert \mathtt{f}(H,U\,|\, F) - \mathtt{f}(H,W\,|\, F) \rVert_{1} \le \frac{2 \lVert A_{F} \rVert_{1} \cdot \delta_{\square}(U, W) + \lVert A_{H} \rVert_{1} \cdot \delta_{1}(U,W)}{\max(\mathtt{t}(F,U), \mathtt{t}(F,W))}.
		\end{equation}
	\end{thm}

	\section{Examples}
	\label{section:examples}
	
	In this section, we provide various computational examples to demonstrate our techniques and results. Throughout this section  we use the motifs $H_{k_{1},k_{2}}$ and $F_{k_{1},k_{2}}$ introduced in Example \ref{ex:motifs_Ex}. In Subsection \ref{subsection:chd_ex}, we compute explicitly and numerically various homomorphism densities for the network given by a torus graph plus some random edges. In Subsection \ref{subsection:chd_profile_SBM}, we compute various CHD profiles for stochastic block networks. Lastly, in Subsection \ref{subsection:motif_transform_ex}, we discuss motif transforms in the context of hierarchical clustering of networks and illustrate this using a barbell network.

	\subsection{Conditional homomorphism densities}\label{subsection:chd_ex}

	\begin{ex}[Torus]\label{ex:torus}
		\textup{Let $\G_{n} = ([n]\times [n],A,\alpha)$ be the $(n\times n)$ torus $\mathbb{Z}_{n}\times \mathbb{Z}_{n}$ with nearest neighbor edges and uniform node weight $\alpha\equiv 1/n^{2}$.  
			Consider the conditional homomorphism density $\mathtt{t}(H_{k,0},\G_{n}\,|\, F_{k,0})$. Since $A$ binary and symmetric, note that $\P_{F_{k,0}\rightarrow \G_{n}}$ is the uniform probability distribution on the sample paths of simple symmetric random walk on $\G_{n}$ for the first $k$ steps. Hence if we denote this random walk by $(X_{t})_{t\ge 0}$, then 
			\begin{align}
				\mathtt{t}(H_{k,0},\G_{n}\,|\, F_{k,0}) &= \P(\lVert X_{k}-(0,0) \rVert_{\infty}=1 \,|\, X_{0}=(0,0)) \\ 			
				&= 4\P(X_{k+1}=(0,0)\,|\, X_{0}=(0,0))\\
				&= \frac{1}{4^{k}}\sum_{\substack{a,b\ge 0\\ 2(a+b)=k+1}} \frac{(k+1)!}{a!a!b!b!}.
			\end{align}
			For instance, we have $\mathtt{t}(H_{3,0},\G_{n}\,|\, F_{3,0})=9/16=0.5625$ and 
			\begin{align}
				\mathtt{t}(H_{9,0},\G_{n}\,|\, F_{9,0})= \frac{2\cdot 10!}{4^{9}}\left(\frac{1}{5! 5!} + \frac{1}{4! 4!} + \frac{1}{3!3!2!2!} \right) = \frac{3969}{16384} \approx 0.2422.
			\end{align}
			See Figure \ref{fig:torus_CHD} for a simulation of Glauber and Pivot chains $F_{k,0}\rightarrow \G_{n}$. As asserted in Corollary \ref{cor:time_avg_observables}, time averages of these dynamic embeddings converge to the correct values of the conditional homomorphism density $\mathtt{t}(H_{k,0},\G_{n}\,|\, F_{k,0})$. The simulation indicates that for sparse networks like the torus, the Glauber chain takes longer to converge than Pivot chain does.  }$\hfill \blacktriangle$	
	\end{ex}

	\begin{ex}[Torus with long-range edges]\label{ex:torus_longedge}
		\textup{Fix parameters $p\in [0,1]$ and $\alpha\in [0,\infty)$. Let $\G_{n}=\G_{n}^{p,\alpha}$ be the $n\times n$ torus $\mathbb{Z}_{n}\times \mathbb{Z}_{n}$ with additional edges added randomly to each non-adjacent pair $(a,b)$ and $(c,d)$, independently with probability $p(|a-c|+|b-d|)^{-\alpha}$. When $\alpha=0$, this reduces to the standard Watts-Strogatz model \cite{watts1998collective}. }
		
		\textup{See Figure \ref{fig:torus_CHD} for some simulation of Glauber and Pivot chains $F_{k,0}\rightarrow \G_{50}$ for $p=0.1$ and $\alpha=0$. Time averages of these dynamic embeddings converge to the correct values of the conditional homomorphism density $\mathtt{t}(H_{k,0},\G_{n}\,|\, F_{k,0})$, which is approximately the ambient edge density $0.1$. This is because if we sample a copy of $F_{k,0}$, it is likely to use some ambient `shortcut' edges so that the two ends of $F_{k,0}$ are far apart in the usual shortest path metric on the torus. Hence the chance that these two endpoints are adjacent in the network $\G_{n}^{p,0}$ is roughly $p$.  }
		
		\textup{In the next example, we use the tree motif $F$ on six nodes and $H$ is obtained from $F$ by adding two extra edges, as described in Figure \ref{fig:torus_long_edges_motif_CHD}. A similar reasoning to the one used above tells us that the probability that a random copy of $F$ from $\G_{n}^{p,0}$ has edges $(2,5)$ and $(3,6)$ should be about $p^{2}$. Indeed, both the Glauber and Pivot chains in Figure \ref{fig:torus_long_edges_motif_CHD} converge to $0.01$.  }
		$\hfill \blacktriangle$
	\end{ex}

	\subsection{CHD profiles of stochastic block networks}
	\label{subsection:chd_profile_SBM}

	Let $\G=([n],A,\alpha)$ be a network. For each integer $r\ge 1$ and a real number $\sigma>0$, we will define a `stochastic block network'  $\mathfrak{X}=([nr],B^{(r)}(A,\sigma^{2}),\beta)$ by replacing each node of $\G$ by a community with $r$ nodes. The node weight $\beta:[nr]\rightarrow [0,1]$ of the block network is inherited from $\alpha$ by the relation $\beta(x) = \alpha(\lfloor x/r \rfloor+1)$. For the edge weight, we define $B^{(r)}(A,\sigma^{2}) = \Gamma^{(r)}(A,\sigma^{2})/\max(\Gamma^{(r)}(A,\sigma^{2}))$, where $\Gamma^{(r)}(A,\sigma^{2})$ is the $(nr\times nr)$ random matrix obtained from $A$ by replacing each of its positive entries $a_{ij}>0$ by an $(r\times r)$ matrix of i.i.d. entries following a Gamma distribution with mean $a_{ij}$ and variance $\sigma^{2}$. Recall that the Gamma distribution with parameters $\alpha$ and $\beta$ has the following probability distribution function 
	\begin{align}
		f_{\alpha,\beta}(x) = \frac{\beta^{\alpha}}{\Gamma(\alpha)} x^{\alpha-1} e^{-\beta x} \mathbf{1}(x\ge 0). 
	\end{align}
	Since the mean and variance of the above distribution are given by $\alpha/\beta$ and $\alpha/\beta^{2}$, respectively, we may set $\alpha= a_{ij}^{2}/\sigma^{2}$ and $\beta =a_{ij}/\sigma^{2}$ for the $(r\times r)$ block corresponding to $a_{ij}$.
	
	For instance, consider two networks $\G_{1}=([6],A_{1},\alpha)$, $\G_{2}=([6],A_{2},\alpha)$ where $\alpha\equiv 1/6$ and 
	\begin{align}\label{eq:block_matrices_template}
		A_{1} = \begin{bmatrix}
			5 & 1 & 1 & 1 & 1 & 1 \\
			1 & 5 & 1 & 1 & 1 & 1 \\
			1 & 1 & 5 & 1 & 1 & 1 \\
			1 & 1 & 1 & 5 & 1 & 1 \\
			1 & 1 & 1 & 1 & 5 & 1 \\
			1 & 1 & 1 & 1 & 1 & 5 
		\end{bmatrix}
		,\qquad 
		A_{2} = \begin{bmatrix}
			1 & 1 & 1 & 5 & 5 & 1 \\
			1 & 1 & 1 & 1 & 1 & 5 \\
			5 & 1 & 1 & 5 & 1 & 5 \\
			5 & 1 & 1 & 1 & 1 & 2 \\
			1 & 5 & 1 & 1 & 1 & 1 \\
			1 & 1 & 5 & 10 & 1 & 1 
		\end{bmatrix}.
	\end{align}
	Let $B_{1}=B^{(10)}(A_{1},1)$, $B_{2}=B^{(10)}(A_{2},1.5)$, and $B_{3}=B^{(10)}(A_{2},0.5)$. Consider the stochastic block networks $\mathfrak{X}_{1}=([60], B_{1},\beta), \mathfrak{X}_{2}=([60], B_{2},\beta)$, and $\mathfrak{X}_{3}=([60], B_{3},\beta)$. The plots of matrices $B_{1}$ and $B_{2}$ are given in Figure \ref{fig:block_networks_pic}.

	In Figure \ref{fig:block_filtration} below, we plot the CHD profiles $\mathtt{f}:=\mathtt{f}(H_{k_{1},k_{2}},\mathfrak{X}\,|\, F_{k_{1},k_{2}})$ for $\mathfrak{X}=\mathfrak{X}_{1}, \mathfrak{X}_{2}$, and $\mathfrak{X}_{3}$.	The first row in Figure \ref{fig:block_filtration} shows the CHD profiles for $k_{1}=k_{2}=0$. At each filtration level $t\in [0,1]$, the value $\mathtt{f}(t)$ of the profile, in this case, means the proportion of diagonal entries in $B_{i}$ at least $t$ (see Example \ref{ex:motifs_Ex}). The CHD profiles for $\mathfrak{X}_{2}$ and $\mathfrak{X}_{3}$ drop quickly to zero by level $t=0.3$, as opposed to the profile for $\mathfrak{X}_{1}$, which stays close to height $1$ and starts dropping around level $t=0.4$. This is because, as can be seen in Figure \ref{fig:block_networks_pic}, entries in the diagonal blocks of the matrix $B_{1}$ are large compared to that in the off-diagonal blocks, whereas for the other two matrices $B_{1}$ and $B_{2}$, diagonal entries are essentially in the order of the Gamma noise with standard deviation $\sigma$.

	\begin{figure*}[h]
		\centering
		\includegraphics[width=0.7 \linewidth]{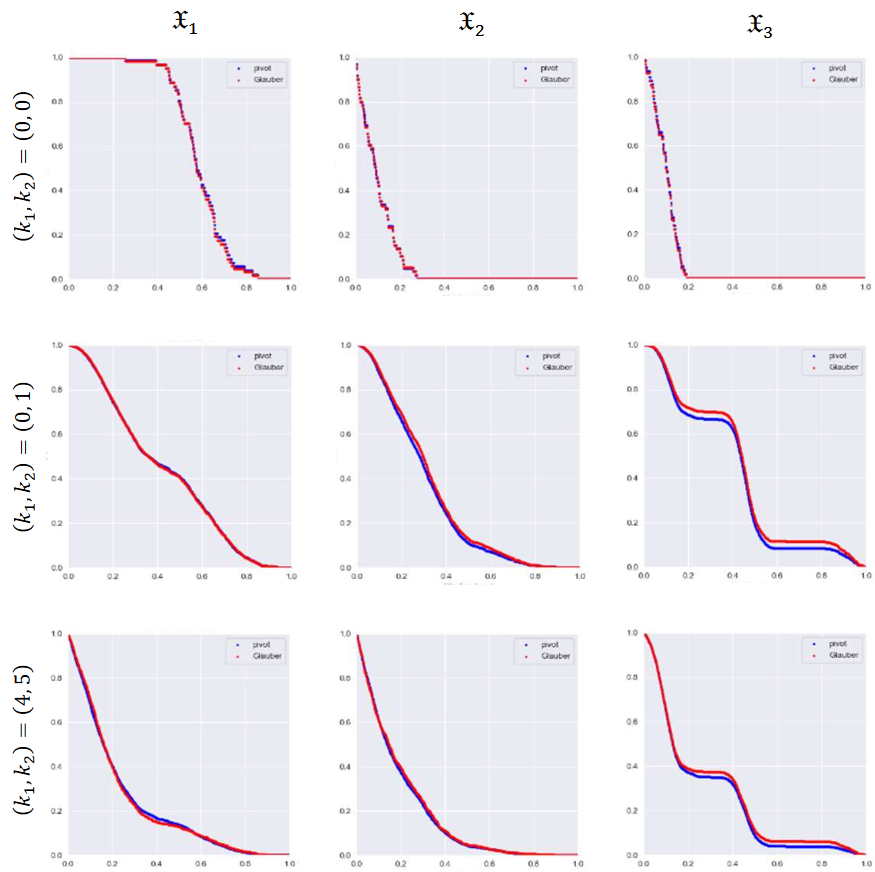}
		\caption{ Plots of CHD profiles $\mathtt{f}(H_{k_{1},k_{2}},\mathfrak{X}\,|\, F_{k_{1},k_{2}})$ for $\mathfrak{X}=\mathfrak{X}_{1}$ (first row), $\mathfrak{X}_{2}$ (second row), and $\mathfrak{X}_{3}$ (third row). To compute each profile, both the Glauber (red) and pivot (blue) chains are run up to $10^{5}$ iterations.
		}
		\label{fig:block_filtration}
	\end{figure*}

	For $\max(k_{1},k_{2})\ge 1$, note that the value of the profile $\mathtt{f}(t)$ at level $t$ equals the probability that the extra edge in $H_{k_{1},k_{2}}$ has weight $\ge t$ in $\mathfrak{X}$, when we sample a random copy of $F_{k_{1},k_{2}}$ from $\mathfrak{X}$. For instance, if $(k_{1},k_{2})=(0,1)$, this quantity is almost the density of edges in $\mathfrak{X}$ whose weights are at least $t$. But since the measure of random homomorphism $\x:F_{1,0}\rightarrow \mathfrak{X}$ is proportional to the edge weight $B_{i}(\x(0),\x(1))$, we are in favor of sampling copies of $F_{0,1}$ with large edge weight.

	In the second row of Figure \ref{fig:block_filtration}, the profile for $\mathfrak{X}_{3}$ differs drastically from the other two, which gradually decays to zero. The small variance in the Gamma noise for sampling $B_{3}$ makes the two values of $5$ and $10$ in $A_{2}$ more pronounced with respect to the `ground level' $1$. Hence we see two plateaus in its profile. As noted in the previous paragraph, the height of the first plateau (about 0.7), is much larger than the actual density (about 0.25) of the edges sample from blocks of value 5. A similar tendency could be seen in the third row of Figure \ref{fig:block_filtration}, which shows the CHD profiles for $(k_{1},k_{2})=(4,5)$. Note that the first plateau in the profile for $\mathfrak{X}$ now appears at a lower height (about 0.4). This indicates that sampling a copy of $F_{4,5}$ is less affected by the edge weights than sampling a copy of $F_{0,1}$.

	\subsection{Hierarchical clustering for networks and motif transform}
	\label{subsection:motif_transform_ex}
	
	\commHL{In this section, we discuss the application of the  motif transform to the setting of hierarchical clustering of networks.}
	

	\paragraph{Hierarchical clustering and dendrogram.}
	\commHL{
		Suppose we are given a finite set $X$ and fixed `levels' $0=t_{0}<t_{1}<\dots<t_{m}$ for some integer $m\ge 0$. Let $\mathcal{H}:=(\mathcal{F}_{t})_{t\in \{t_{0},\dots,t_{m}\}}$ be a sequence of collection of subsets of $X$, that is, $\mathcal{F}_{t_{k}}\subseteq 2^{X}$ with $2^{X}$ denoting the power set of $X$. We call  $\mathcal{H}$ a  \textit{hierarchical clustering} of $X$ if (1) $\mathcal{F}_{0}=X$ and $\mathcal{F}_{t_{m}}=\{ X \}$ and (2) for each $0\le a \le b \le m$ and for each $A\in \mathcal{F}_{t_{a}}$, there exists a unique $B\in \mathcal{F}_{t_{b}}$ with $A\subseteq B$. For each $t\in \{t_{0},\dots,t_{M}\}$, we call each $A\in \mathcal{F}_{t}$ a \textit{cluster} of $X$ at \textit{level $t$}. One can associate a tree $T=(V,E)$ to $\mathcal{H}$ by setting $V = \bigsqcup_{k=0}^{m} \mathcal{F}_{t_{k}}$ and letting the edge set $E$ consist of all pairs $(A,B)$ such that $A\in \mathcal{F}_{t_{k}}$, $B\in \mathcal{F}_{t_{k+1}}$, and $A\subseteq B$ for some $k=0,1,\dots, m-1$. The graph $T=(V,E)$ defined in this way is indeed a tree with root $\{X\}$ at level $t_{m}$ and a set of leaves $X$ at level $0$. We call $T$ a \textit{dendrogram} of $\mathcal{H}$. See \citep{jardine1971mathematical, carlsson2010characterization} for more details on hierarchical clustering and dendrograms. }
	
	\paragraph{Single-linkage hierarchical clustering for finite metric spaces.} 
	
	\commHL{
		\textit{Single-linkage hierarchical clustering} is a standard mechanism to obtaining a hierarchical clustering of a finite matrix space. Suppose $(X,d)$ is a finite metric space. 
		Let $0=t_{0}<t_{1}<\dots<t_{m}$ be the result of ordering all distinct values of the pairwise distances $d(x,y)$ for $x,y\in X$. For each $t\ge 0$, define the equivalence relation  $\simeq_{t}$ on $X$ as the \emph{transitive closure} of the relation  $d(x,y)\le t$, that is, 
		\begin{align}\label{eq:metric_equiv_relation}
			x\simeq_{t} y \quad \textup{$\Longleftrightarrow$} \quad \begin{matrix} 
				\textup{exists an integer $m\ge 1$ and $z_{0},\dots,z_{m}\in X$ s.t.} \\  \textup{$d(z_{i},z_{i+1})\leq t$ for $i=0,\dots,m-1$ and $z_{0}=x$, $z_{m}=y$.}
			\end{matrix} 
		\end{align}
		Then $\mathcal{H}:=(U_{t_{k}})_{0\le k \le m}$ is a hierarchical clustering of $X$. The associated dendrogram $T$ of $\mathcal{H}$ is called the \textit{single-linkage (hierarchical clustering) dendrogram} of $(X,d)$. }

	\paragraph{Single-linkage hierarchical clustering for networks.}

		
		{	\color{black}
			
			We introduce a method for computing the hierarchical clustering of networks based on a metric on the node set. Let $\G=([n],A,\alpha)$ be a network.  We view the weight $A(x,y)$ between distinct nodes as representing the magnitude of the relationship between them, so the larger $A(x,y)$ is, the stronger the nodes $x,y$ are associated.  Hence it would be natural to interpret the off-diagonal entries of $A$ as a measure of similarity between the nodes, as opposed to a metric $d$ on a finite metric space, which represents `dissimilarity' between points. 
			
			In order to define a metric $d_{A}$ on the node set $[n]$, first transform the pairwise similarity matrix $A$ into a pairwise \textit{dissimilarity} matrix $A'$ as 
			\begin{align}\label{eq:A_modified}
				A'(x,y) = \begin{cases}
					0 & \text{if $x=y$} \\
					\infty & \text{if $A(x,y)=0$ and $x\ne y$}\\
					\max(A) - A(x,y) & \text{otherwise}.
				\end{cases}
			\end{align}		
			We then define a metric $d_{A}$ on the node set $[n]$ by letting $d_{A}(x,y)$ be the smallest sum of all $A'$-edge weights of any sequence of nodes starting from $x$ and ending at $y$: 
			\begin{align}\label{eq:d_A_def}
				d_{A}(x,y):= \inf \left\{ \sum_{i=1}^{m} A'(x_{i},x_{i+1}) \,\bigg|\, x_{1}=x,\, x_{m+1}=y,\, x_{1},\dots,x_{m+1}\in [n] \right\}.
			\end{align}
			This defines a metric space $([n], d_{A})$ associated with the network $\G=([n], A, \alpha)$. We let $\mathcal{H}=\mathcal{H}(\G)$ to denote the hierarchical clustering of $([n], d_{A})$. We call the dendrogram $T=(V,E)$ of $\mathcal{H}(\G)$ the \textit{single-linkage heirarchical clustering dendrogram} of the network $\G$, or simply the \textit{dendrogram} of $\G$. Computing the metric $d_{A}$ in \eqref{eq:d_A_def} can be easily accomplished in $O(n^{3})$ time by using the  Floyd-Warshall algorithm \citep{floyd1962algorithm, warshall1962theorem}. See Figures \ref{fig:dendrogram_1}, \ref{fig:dendrogram_2}, and \ref{fig:austen_shakespeare_dendro} for network dendrograms computed in this way.

			The hierarchical clustering $\mathcal{H}$ of $\G$ defined above is closely related to the following notion of network capacity function. Given a network $\G=([n],A,\alpha)$, consider the  `capacity function' $T_{\G}:[n]^{2}\rightarrow [0,\infty)$ defined by
			\begin{align}\label{eq:def_capacity}
				T_{\G}(x,y) = \sup_{t\ge 0}\left\{ t\ge 0\,\bigg|\, \begin{matrix}
					\text{$\exists x_{0},x_{1},\cdots,x_{m}\in [n]$ s.t. $(x_{0},x_{m})=(x,y)$ or $(y,x)$ } \\
					\text{and $\min_{0\le i <m} A(x_{i},x_{i+1}) > t$}. 		
				\end{matrix}
				\right\}. 
			\end{align}
			That is, $T_{\G}(x,y)$ is the minimum edge weight of all possible walks from $x$ to $y$ in $\G$, where a \textit{walk} in $\G$ is a sequence of nodes $x_{0},\dots,x_{m}$ such that $A(x_{i},x_{i+1})>0$ for $i=1,\dots,m-1$. Let $\simeq_{t}$ denote the equivalence relation induced by $d_{A}$ as in \eqref{eq:metric_equiv_relation}. Then one can see that 
			\begin{align} \label{eq:network_dendro_equiv_capacity}
				x\simeq_{t} y \quad \textup{$\Longleftrightarrow$} \quad T_{\G}(x,y) \ge \max(A)-t \quad \textup{or} \quad x=y.
			\end{align}
			Thus, $x$ and $y$ merge into the same cluster in $\mathcal{H}$ at level $\max(A)-T_{\G}(x,y)$. Furthermore, it is easy to see that $T_{\G}$ satisfies the following `dual' to ultrametric condition (see \cite{clust-um} for how the ultrametric condition relates to dendrograms) for distinct nodes: 
			\begin{align}\label{eq:ultrametric}
				T_{\G}(x,y) \ge \min( T_{\G}(x,z), T_{\G}(z,y)) \qquad \text{$\forall$ $x,y,z\in [n]$ s.t. $x\ne y$.}
			\end{align}

			Note that $T_{\G}(x,x) = A(x,x)$ for all $x\in [n]$. Hence \eqref{eq:ultrametric} may not hold if $x=y$, as $T_{\G}(x,y)=A(x,x)$ could be less than the minimum of $T_{\G}(x,z)$ and $T_{\G}(z,x))$ (e.g., $\G$ a simple graph). If we modify the capacity function on the diagonal by setting $T_{\G}(x,x)\equiv \max(A)$ for all $x$, then \eqref{eq:ultrametric} is satisfied for all choices $x,y,z\in [z]$. This modification corresponds to setting $A'(x,x)=0$ in \eqref{eq:A_modified}.
			
			The above construction of the hierarchical clustering $\mathcal{H}$ of $\G=([n], A, \alpha)$ does not use diagonal entries of $A$. One can slightly modify the definition of hierarchical clustering of $\G$ in a way that it also uses the diagonal entries of $A$ by allowing each node $x$ to `appear' in the dendrogram at different times depending on its `self-similarity' $A(x,x)$. More precisely, define a relation $\simeq_{t}'$ on the node set $[n]$ by  $x \simeq_{t}' y$ if and only if $T_{\G}(x,y) \ge \max(A)-t$ for all $x,y\in [n]$ (not necessarily for distinct $x,y$). Then $x \simeq_{t}' x$ if and only if $t\ge \max(A)-A(x,x)$. Hence in order for the relation $\simeq_{t}'$ to be an equivalence relation, we need to restrict its domain to $\{x\in [n]\,|\, \max(A)-A(x,x)\le t\}$ at each filtration level $t$. The resulting dendrogram is called a \textit{treegram}, since its leaves may appear at different heights \citep{treegrams}. 

			Note that the capacity function in \eqref{eq:def_capacity} can be defined for graphons $U$ instead of networks. Hence by using \eqref{eq:network_dendro_equiv_capacity}, we can also define hierarchical clustering dendrogram for graphons in a similar manner.   The following example illustrates single-linkage hierarchical clustering of the three-block graphons from Example  \ref{ex:graphon_example}. 
			
		}
		
		\begin{ex}\label{ex:dendro_graphon}
			\normalfont
			
			Recall the graphons $U$, $U^{\circ 2}$, and $U\cdot U^{\circ 2}$ discussed in Example \ref{ex:graphon_example}. Note that $U$ is the graphon $U_{\G}$ associated to the network $\G=([3],A,\alpha)$ in Example \ref{ex:matrix_spectral}. Single-linkage hierarchical clustering dendrograms of the three networks corresponding to the three graphons are shown in Figure \ref{fig:graphon_dendro} (in solid + dotted blue lines), which are solely determined by the off-diagonal entries. Truncating each vertical line below the corresponding diagonal entry (dotted blue lines), one obtains the treegrams for the three networks (solid blue lines). 

			\begin{figure*}[h]
				\centering
				\includegraphics[width=0.85 \linewidth]{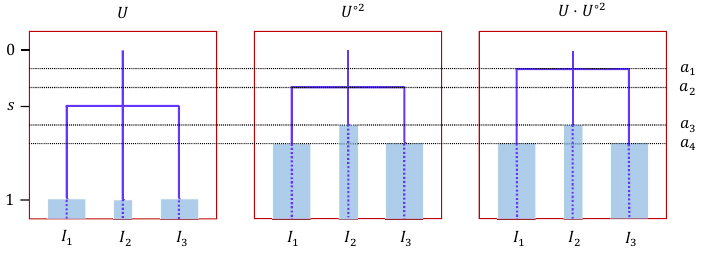}
				\caption{Dendrograms and treegrams of the networks associated to the graphons $U=U_{\G}$ (left), $U^{\circ 2}$ (middle), and $U\cdot U^{\circ 2 }$ (right) in  Example \ref{ex:graphon_example}. \commHL{The vertical axis show the values of the capacity function from \eqref{eq:def_capacity}. } $a_{1}=s^{2}(1+\epsilon)/2$, $a_{2}=s(1+\epsilon)/2$, $a_{3}=s^{2}(1-\epsilon/4)+\epsilon$, and $a_{4}=(1/2)+(s^{2}-1/2)\epsilon$. See the text in Example \ref{ex:dendro_graphon} for more details.
				}
				\label{fig:graphon_dendro}
			\end{figure*}
			
			\textup{Furthermore, one can also think of hierarchical clustering of the graphons by viewing them as networks with continuum node set $[0,1]$. The resulting dendrogram is shown in Figure \ref{fig:graphon_dendro} (solid blue lines + shaded rectangles) }$\hfill \blacktriangle$.
		\end{ex}

		\commHL{In the following example, we illustrate how motif transforms can be used to suppress weak connections between two communities in order to improve the recovery of the hierarchical clustering structure of a given network. }

		\begin{ex}[Barbell networks]\label{ex:barbell}
			\normalfont 
			
			\textup{In this example, we consider `barbell networks', which are obtained by connecting two networks by a single edge of weight 1. When the two networks are $\mathcal{H}_{1}$ and $\mathcal{H}_{2}$, we denote the resulting barbell network by $\mathcal{H}_{1}\oplus \mathcal{H}_{2}$, and we say $\mathcal{H}_{1}$ and $\mathcal{H}_{2}$ are the two components of $\mathcal{H}_{1}\oplus \mathcal{H}_{2}$.}
			
			\begin{figure*}[h]
				\centering
				\includegraphics[width=0.25 \linewidth]{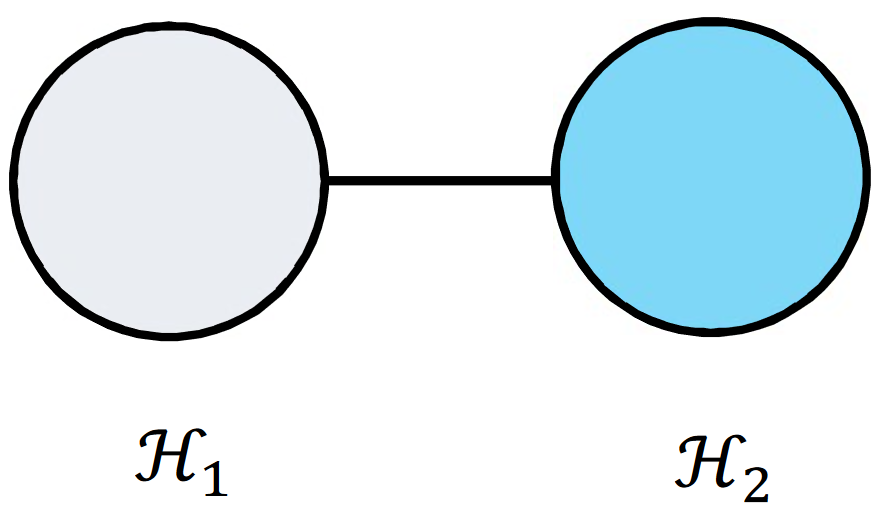}
				\caption{ Depiction of a barbell network.
				}
				\label{fig:barbell_depiction}
			\end{figure*}
			
			\textup{
				Recall the network $\G_{n}^{p,\alpha}$ defined in Example \ref{ex:torus_longedge}, which is the $(n\times n)$ torus with long-range edges added according to the parameters $p$ and $\alpha$. Also let $\mathfrak{X}=([nr],B_{r}(A,\sigma^{2}),\beta)$ denote the stochastic block network constructed from a given network $\G=([n],A,\alpha)$ (see Subsection \ref{subsection:chd_profile_SBM}). Denote the stochastic block network corresponding to $\G_{n}^{p,\alpha}$ with parameters $r$ and $\sigma$ by $\G_{n}^{p,\alpha}(r,\sigma)$. }
			
			\begin{figure*}[h]
				\centering
				\includegraphics[width=1 \linewidth]{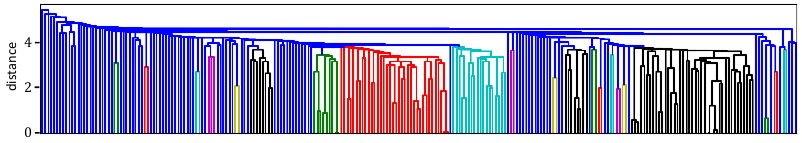}
				\caption{ Single-linkage dendrogram for the barbell network $\G_{2}$. 
				}
				\label{fig:dendrogram_1}
			\end{figure*}
			
			\textup{Now define barbell networks $\G_{1}:=\G_{10}^{0,0}\oplus \G_{10}^{0.2,0}$ and $\G_{2} = \G_{5}^{0,0}(5,0.6)\oplus \G_{5}^{0.2,0}(5,0.2)$. \commHL{Also, let $\G_{3}:=\G_{2}^{C_{3}}$ be the network obtained from $\G_{2}$ by the motif transform using the} triangle motif $C_{3}:=([3], \mathbf{1}_{\{(1,2),(2,3),(3,1) \}})$ (here the orientation of the edges of $C_{3}$ is irrelevant since the networks are symmetric). In each barbell network, the two components are connected by the edge between node 80 and node 53 in the two components.  For each $i\in \{1,2,3\}$, let $A_{i}$ denote the edge weight matrix corresponding to $\G_{i}$. The plots for $A_{i}$'s are given in Figure \ref{fig:barbell_pic}.}

			We are going to consider the single-linkage dendrograms of each barbell network for their hierarchical clustering analysis. \commHL{We omit the dendrogram of the simple graph $\G_{1}$.} For $\G_{2}$, the Gamma noise prevents all nodes from merging at the same level. Instead, we expect to have multiple clusters forming at different levels and they all merge into one cluster at some positive level $t>0$. Indeed, in the single-linkage dendrogram for $\G_{2}$ shown in Figure \ref{fig:dendrogram_1}, we do observe such hierarchical clustering structure of $\G_{2}$. 
			
			\begin{figure*}[h]
				\centering
				\includegraphics[width=1 \linewidth]{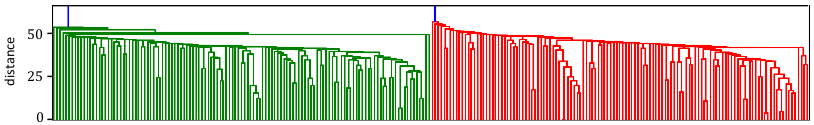}
				\caption{ Single-linkage dendrogram for barbell network $\G_{3}$, which is obtained by motif-transforming $\G_{2}$ using a triangle $C_{3}$. 
				}
				\label{fig:dendrogram_2}
			\end{figure*}
			
			\textup{However, the `single linkage' between the two main components of $\G_{2}$ is very marginal compared to the substantial interconnection within the components. We \commHL{may} use motif transforms prior to single-linkage clustering in order to better separate the two main components. The construction of $\G_{3} = \G_{2}^{C_{3}}$ using triangle motif transform and its dendrogram in Figure \ref{fig:dendrogram_2} demonstrate this point.
			}


			\textup{In the dendrogram of $\G_{3}$ shown in Figure \ref{fig:dendrogram_2}, we see that the two main clusters still maintain internal hierarchical structure, but they are separated at all levels $t\ge 0$.  A similar motif transform may be used to suppress weak connections in the more general situation in order to emphasize the clustering structure within networks, but without perturbing the given network too much.} $\hfill\blacktriangle$
			

		\end{ex}

		\section{Application I: Subgraph classification and Network clustering with Facebook networks}
		\label{section:FB}

		In this section, we apply our methods to a problem consisting
		of clustering given a data set of networks. In our experiment,
		we use the \textit{Facebook100 dataset}, which consists of the
		snapshots of the Facebook network of 100 schools in the US in
		Sep. of 2005. Since it was first published and analyzed by
		Traud, Mucha, and Porter in \cite{traud2012social}, it has
		been regarded as a standard data set in the field of social
		network analysis. In the dataset, each school's social network
		is encoded as a simple graph of anonymous nodes corresponding
		to the users, and nodes $i$ and $j$ are adjacent if and only
		if the corresponding users have a friendship relationship on the
		Facebook network. The networks have varying sizes: Caltech36
		is the smallest with 769 nodes and 16,656 edges, whereas
		Texas84 is the largest with 36,371 nodes and 1,590,655 edges. The
		lack of shared node labels and varying network sizes make it
		difficult to directly compare the networks and perform
		clustering tasks. For instance, for directly computing a
		distance between two networks of different sizes without a shared
		node labels, one needs to find optimal correspondence between
		the node sets  (as in \eqref{eq:ntwk_p_distance}), which is
		computationally very expensive. We overcome this difficulty by
		using our motif sampling for computing the Matrix of Average
		Clustering Coefficients (MACC) (see Definition \ref{def:MACC}) for each network. This
		the compressed representation of social networks will then
		be used for performing  hierarchical clustering on the whole
		collection of 100 networks.

		\begin{figure*}[h]
			\centering
			\includegraphics[width=0.97 \linewidth]{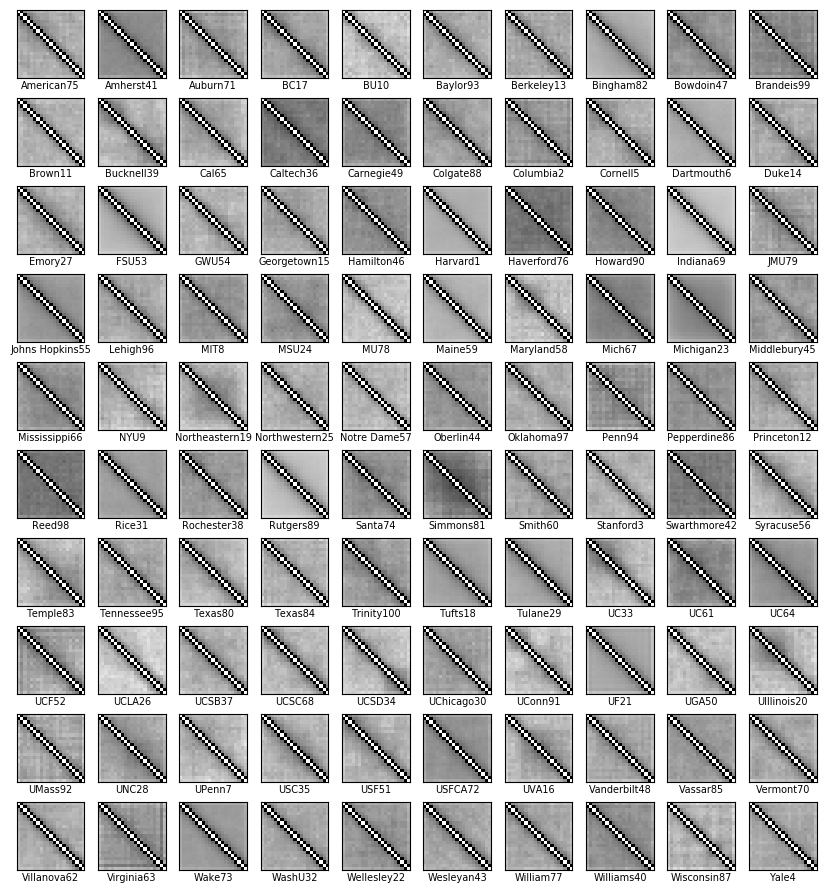}
			\caption{ Matrices of Average Clustering Coefficients
				(MACC) for the 100 Facebook network data set using
				the chain motif $F = ([21],
				\mathbf{1}_{\{(1,2),(2,3),\cdots,(20,21)
					\}})$. Values of the entries are shown in
				greyscale with black = 1 and white = 0. The two main diagonals correspond to the edges in the motif $F_{0,20}$ and always have a value of 1.  Each entry $(i,j)$ for $|i-j|>1$ equals to the probability of seeing the corresponding `long-range' edge along a uniformly sampled copy of the chain motif. 
			}
			\label{fig:FB_MACC}
		\end{figure*}
		
		\subsection{MACCs for the Facebook100 dataset}
		
		The full MACCs of size $21$ for the 100 facebook networks are shown in Figure \ref{fig:FB_MACC}. We used the chain motif \commHL{$F = ([21], \mathbf{1}_{\{(1,2),(2,3),\cdots,(20,21) \}})$} of 20 edges, which we sampled from each network by Glauber chain (see Definition \ref{def:glauber_chain}) for $2n\log n$ iterations, where $n$ denotes the number of nodes in the given network, which we denote by $\mathcal{G}$. Each entry $(i,j)$ of the MACC is computed by taking the time average in \eqref{eq:cond_hom_ergodic} with motifs \commHL{$F$ and $H=H_{ij} := ([21], \mathbf{1}_{ \{(i,j)\}})$.} This time average along the Glauber chain $F\rightarrow \G$ converges to a $21\times 21$ matrix as shown in Figure \ref{fig:FB_MACC}. Note that the two main diagonals on $|i-j|=1$ are always 1 as they correspond to the edges of the chain motif $F$ embedded in the network. For $|i-j|>1$, the $(i,j)$ entry equals the conditional homomorphism density $\mathtt{t}(H_{ij}, \G\,|\, F)$, which is the probability of seeing the corresponding `long-range' edge $(i,j)$ along a uniformly sampled copy of the chain motif $F$ from $\G$. We note that in Figure \ref{fig:FB_MACC}, in order to emphasize the off-diagonal entries, MACCs are plotted after the square-root transform.

		MACC gives a natural and graphical generalization of the network clustering coefficient (see Example \ref{ex:avg_clustering_coeff}). For instance, consider the MACCs of \dataset{Caltech}, \dataset{Harverford}, \dataset{Reed}, \dataset{Simmons}, and \dataset{Swarthmore} in Figure \ref{fig:FB_MACC}. These are relatively small private or liberal arts schools, so one might expect to see stronger clustering among a randomly sampled chain of 21 users in their Facebook network. In fact, their MACCs show large values (dark) off of the two main diagonals, indicating that it is likely to see long-range connections along a randomly sampled chain \commHL{$F$} of 21 friends. On the other hand, the MACCs of \dataset{Indiana}, \dataset{Rutgers}, and \dataset{UCLA} show relatively small (light) values away from the two main diagonals, indicating that it is not very likely to see long-range friendships among a chain of 21 friends in their Facebook network. Indeed, they are large public schools so it is reasonable to see less clustered friendships in their social network.

		\subsection{Subgraph classification}
		
		\commHL{In this section, we consider the subgraph classification problem in order to compare the performance of MACCs to that of  classical network summary statistics such as edge density, diameter, and average clustering coefficient. }
		
		
		\commHL{
			The problem setting is as follows. Suppose we have two networks $\G_{1}$ and $\G_{2}$, not necessarily of the same size. From each network $\G_{i}$, we sample 100 connected subgraphs of 30 nodes by running a simple symmetric random walk on $\G_{i}$ until it visits 30 distinct nodes and then taking the induced subgraph on the sampled nodes. Subgraphs sampled from network $\G_{i}$ get label $i$ (see Figure \ref{fig:subgraphs_ex_classification} for examples of subgraphs subject to classification). Out of the total of 200 labeled subgraphs, we use 100  (50 from each network) to train \commHL{several} logistic regression classifiers corresponding to the input features consisting of various network summary statistics of the subgraphs --- edge density, minimum degree, maximum degree, (shortest-path) diameter, degree assortativity coefficient \citep{newman2002assortative}, number of cliques, and average clustering coefficient --- as well as MACCs at four scales $k\in \{5,10,15,20\}$. \commHL{Each trained logistic classifier} is used to classify the remaining 100 labeled subgraphs (50 from each network). The performance is measured by using the area-under-curve (AUC) metric for the receiver-operating characteristic (ROC) curves.} 
		
		\commHL{We compare the performance of a total of 11 logistic classifiers trained on the various summary statistics of subgraphs described above using the seven Facebook social networks \dataset{Caltech}, \dataset{Simmons}, \dataset{Reed}, \dataset{NYU}, \dataset{Virginia}, \dataset{UCLA}, and \dataset{Wisconsin}. There are total 21 pairs of distinct networks $(\G_{1},\G_{2})$ we consider for the subgraph classification task. } For each pair of distinct networks, we repeated the same experiment ten times and reported the average AUC scores together with their standard deviations. 
		
		\begin{figure*}[h]
			\centering
			\includegraphics[width=1 \linewidth]{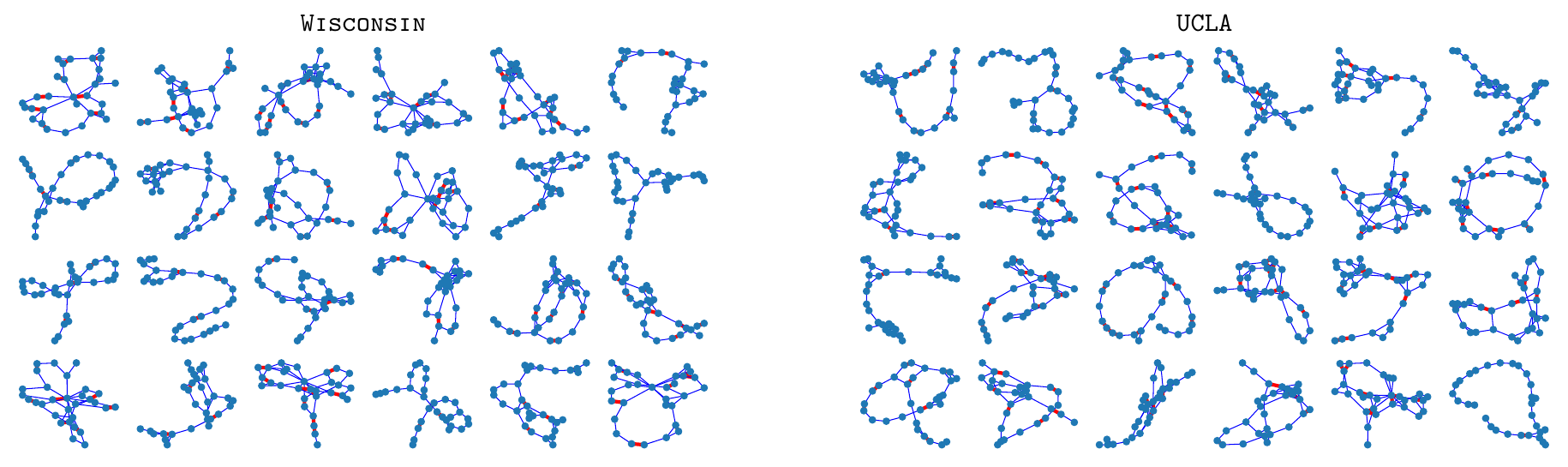}
			\caption{ \commHL{Examples of 30-node connected subgraphs from two Facebook social networks \dataset{Wisconsin} and \dataset{UCLA} 
					Each subgraph is sampled by running a simple symmetric random walk on the network until visiting 30 distinct nodes and then taking the induced subgraph on the sampled nodes.} 
			}
			\label{fig:subgraphs_ex_classification}
		\end{figure*}

		\commHL{
			As we can see from the results reported in Table \ref{table:subgraph_classification1}, classification using MACCs achieves the best performance in all 21 cases 
			This indicates that MACCs are network summary statistics that are more effective in capturing structural information shared among subgraphs from the same network than the benchmark network statistics. Furthermore, observe that the classification performance using MACCs is mostly improved by increasing the scale parameter $k$. This show that MACCs do capture not only local scale information (recall that the average clustering coefficient is closely related to MACC with $k=2$, see Example \ref{ex:avg_clustering_coeff}), but also the mesoscale (intermediate between local and global scales) structure of networks \citep{milo2002network, alon2007network, schwarze2020motifs}. 
		}

		\begin{table*}[h]
			\centering
			\includegraphics[width=1 \linewidth]{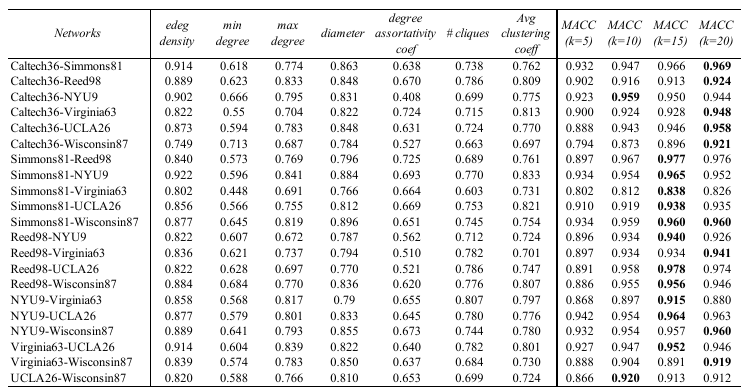}
			\caption{  \commHL{Performance benchmark on subgraph classification tasks. From the two Facebook social networks mentioned in \commHL{in each row}, 100 subgraphs of 30 nodes are sampled. \commHL{Different logistic regression classifiers are then trained using eleven different statistics} 
					of the sampled subgraphs on a 50\% training set. The classification performance on the other 50\% test set is reported as the area-under-curve (AUC) for the receiver-operating characteristic (ROC) curves. The table shows the mean AUC over ten independent trials. The best performance in each case is marked in bold. The standard deviations are reported in Table \ref{table:subgraph_classification_std} in the appendix.}
			}
			\label{table:subgraph_classification1}
		\end{table*}

		\subsection{Clustering the Facebook networks via MACCs}
		
		For a more quantitative comparison of the MACCs in Figure \ref{fig:FB_MACC}, we show a multi-dimensional scaling of the MACCs together with cluster labels obtained by the $k$-means algorithm with $k=5$ in Figure \ref{fig:MDS2d} Each school's Facebook network is represented by its $21\times 21$ MACC, and mutual distance between two networks are measured by the Frobenius distance between their MACCs. Note that, as we can see from the corresponding MACCs, the five schools in the top left cluster in Figure \ref{fig:MDS2d} are private schools with a high probability of long-range connections, whereas all schools including \dataset{UCLA}, \dataset{Rutgers}, and \dataset{Indiana} in the bottom right cluster in Figure \ref{fig:MDS2d} have relatively sparse long-range edges. For a baseline comparison, we show the result of the $k$-means clustering of the same dataset only using the number of nodes and average degree for each network in Figure \ref{fig:baseline_clustering}. The two clustering results have some similarities but also some interesting differences: The cluster that contains small private schools \dataset{Caltech} and \dataset{Reed}; The location of \dataset{UCLA} and \dataset{USFCA} with respect to other clusters. This shows qualitatively different clustering results can be obtained by using the local clustering structure of the networks encoded in their MACCs instead of the macroscopic information counting nodes and edges.

		\begin{figure*}[h]
			\centering
			\includegraphics[width=1 \linewidth]{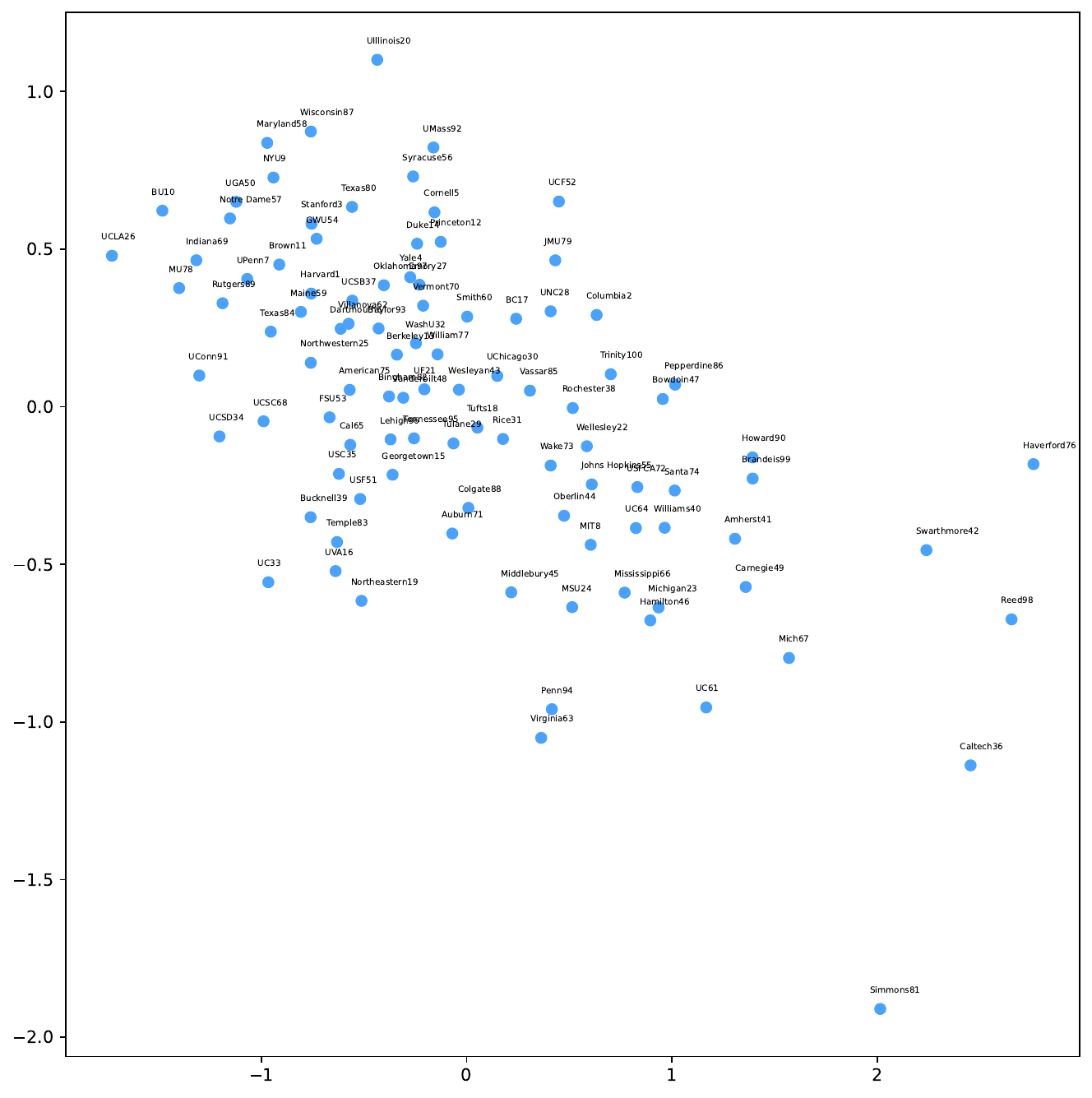}
			\caption{\commHL{Multi-dimensional scaling plot of the MACCs
					of the Facebook100 dataset in Figure
					\ref{fig:FB_MACC}. \commHL{We measured distance between two MACCs using the matrix Frobenius norm.}} 
		}
		\label{fig:MDS2d}
	\end{figure*}
	
	\begin{figure*}[h]
		\centering
		\includegraphics[width=1 \linewidth]{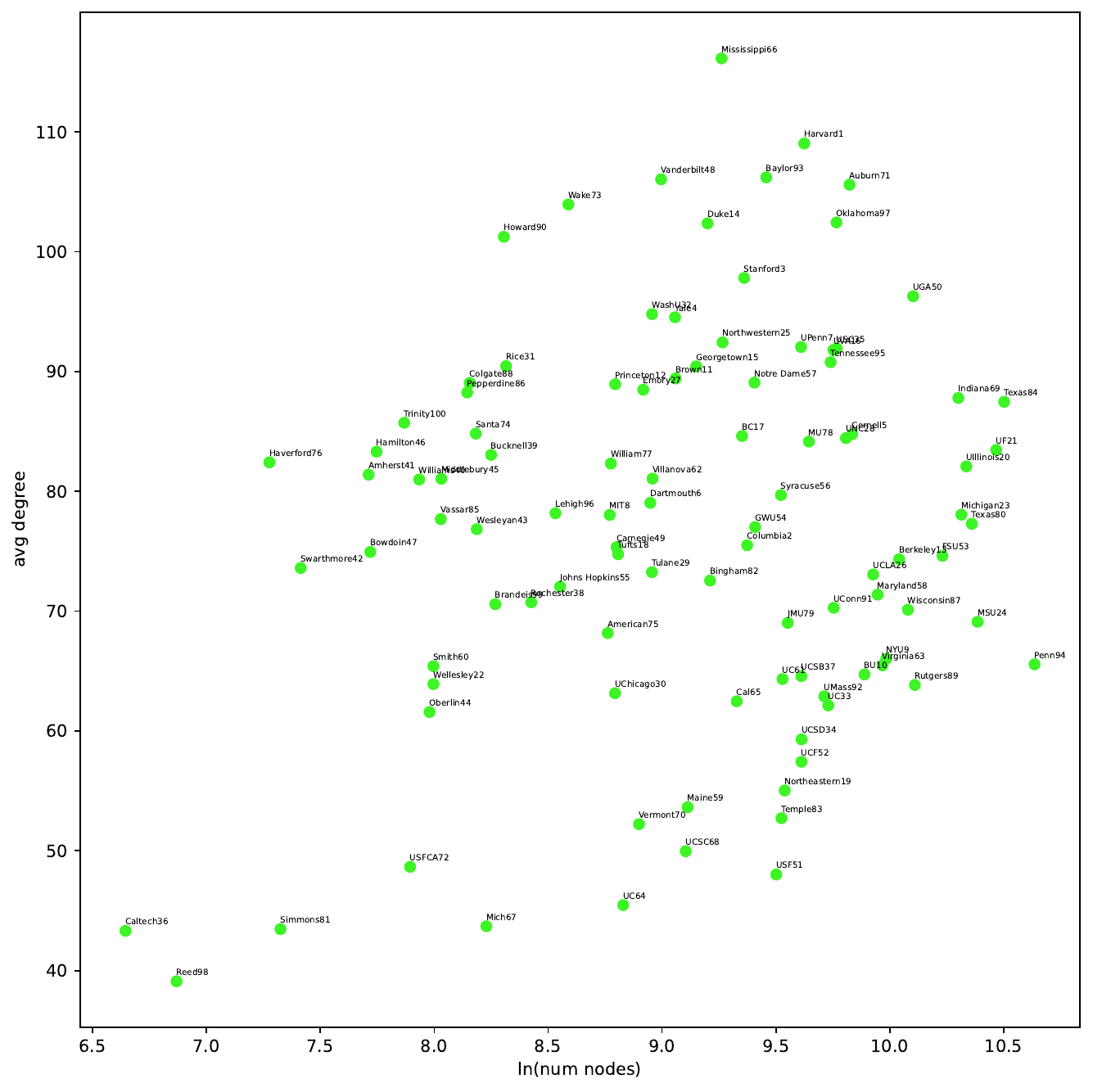}
		\caption{\commHL{A two-dimensional scatter plot of the Facebook100 dataset using the average degree and the natural logarithm of the number of nodes.} 
		}
		\label{fig:baseline_clustering}
	\end{figure*}

	\begin{figure*}[h]
		\centering
		\includegraphics[width=1 \linewidth]{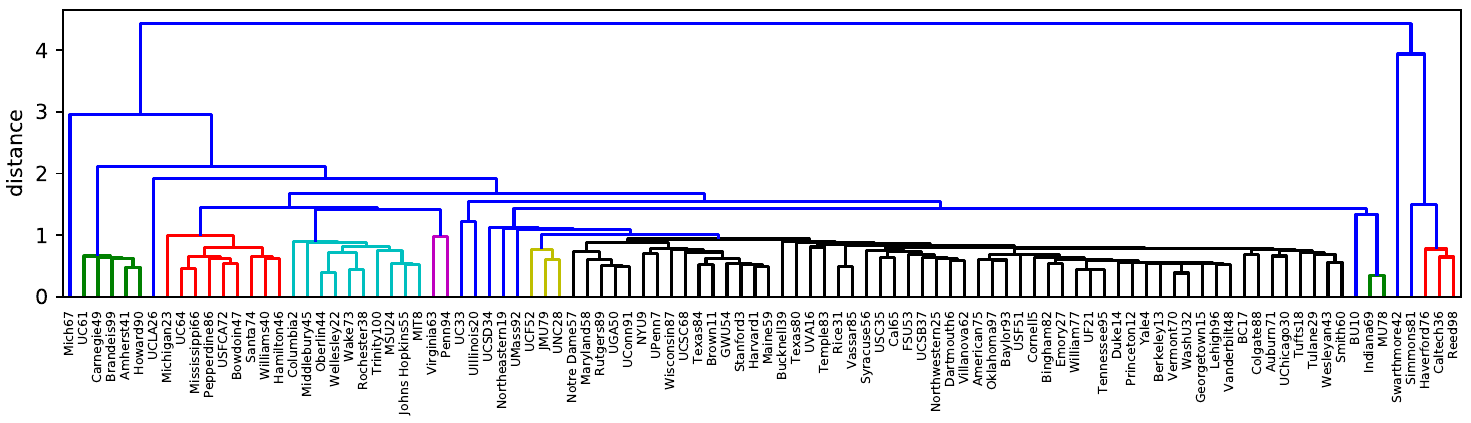}
		\caption{ Single-linkage hierarchical clustering
			dendrogram of the Facebook100 dataset using the
			$21\times 21$ matrices of average clustering
			coefficients (MACC) shown in Figure \ref{fig:FB_MACC}. Two schools with similar MACCs merge early in the dendrogram. Clusters emerging before level $1$ are shown in non-blue colors. 
		}
		\label{fig:FB_dendro}
	\end{figure*}
	
	\begin{figure*}[h]
		\centering
		\includegraphics[width=1 \linewidth]{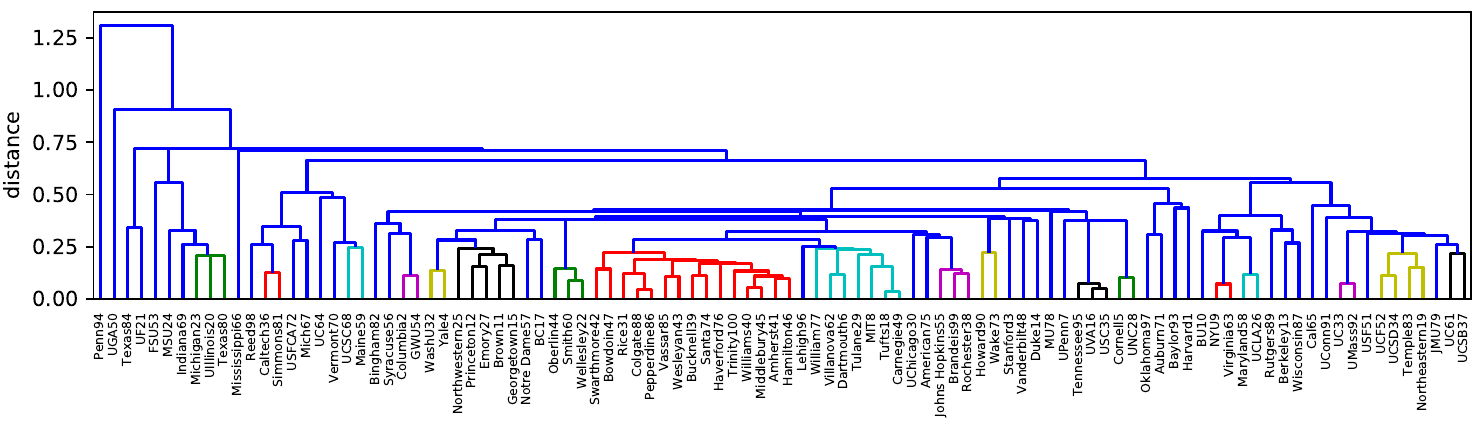}
		\caption{ Single-linkage hierarchical clustering
			dendrogram of the Facebook100 dataset using the normalized $L_{2}$-distance using the number of nodes and average degree used in Figure \ref{fig:baseline_clustering}. Clusters emerging before level $0.25$ are shown in non-blue colors.  
		}
		\label{fig:FB_dendro}
	\end{figure*}

	We also show a single-linkage hierarchical clustering dendrogram of the MACCs in Figure \ref{fig:FB_dendro}, where two schools whose MACCs are Frobenius distance $d$ away from each other merge into the same cluster at height $d$. Similarly, as we have seen in Figure \ref{fig:MDS2d}, the rightmost cluster consisting of the private schools \dataset{Simmons}, \dataset{Haverford}, \dataset{Caltech}, and \dataset{Reed} is separated by any other schools by distance as least 4; In the middle, we also observe the cluster of public schools including \dataset{Maryland}, \dataset{Rutgers}, and \dataset{UConn}. Lastly, we also provide a dendrogram using the baseline network metric using the normalized number of nodes and average degrees as used in Figure \ref{fig:baseline_clustering}. 
	
	Lastly, we also remark that an entirely different approach for network data clustering as well as an application to the Facebook100 dataset is presented in \cite{onnela2012taxonomies}. There, a given network's community structure is encoded as a profile of a 3-dimensional `mesoscopic response function' (MRF), which is computed by the multiresolution Potts model for community detection with varying scaling parameters.  MRFs encode the global community structure of a network, whereas MACCs capture local community structure at a chosen scale. 
	
	\subsection{Computational complexity and remarks}
	
	We can estimate the
	the complexity of the MACC-based method as follows: each step of the Glauber chain update $\mathbf{x}\mapsto \mathbf{x}'$ has
	the complexity of order $O(k \Delta(F) \Delta(\mathbf{x}))$, where $k$ denotes the number of nodes in the motif $F$, and $\Delta(\cdot)$ denotes the maximum degree of a simple graph, and $\Delta(\cdot)$ denotes the maximum degree of the nodes in the image of the homomorphism $\x$ in $\G$. Note the trivial bound $\Delta(\mathbf{x})\le \Delta(\G)$. By adding up these terms for a given number of iterations, the average time complexity of a single Glauber chain update is approximately $O(k \Delta(F) \cdot \texttt{avgdeg}(\G))$, where $\texttt{avgdeg}(\G)$ denotes the average degree of $\G$.  For dense networks, $\textup{avgdeg}(\G)$
	maybe large but the mixing time of the Glauber chain is small (see Theorem \ref{thm:gen_coloring_mixing_tree}); for sparse
	networks, the Glauber chain takes longer to converge but
	$\texttt{avgdeg}(\G)$ is small. In our experiments, we used $2n\log n$
	steps of the Glauber chain for each network $\G$ with $n$ nodes
	resulting in the total computational cost of $O(n\log n\cdot \textup{avgdeg}(\G))$. One could use fewer iterations to quickly get a crude estimate. 
	
	We used a modest computational resource for our experiments: A quad-core 10th Gen. Intel Core i5-1035G7 Processor and 8 GB LPDDR4x RAM with Windows 10 operating system. The actual times for computing the MACCs shown in Figure \ref{fig:FB_MACC} are shown in Table \ref{fig:computation_time_table}. The
	computation can easily be parallelized even for computing MACC for a single network. Indeed, since the MACC of a given
	network is computed by a Monte Carlo integration, one can use multiple Glauber chains on different cores and average the individual results to reduce the computation time by a large factor. All scripts for replicating the experiments can be obtained from our GitHub repository \url{https://github.com/HanbaekLyu/motif_sampling}.
	
	We close this section by pointing out some of the key
	advantages of our method for the network clustering
	problem. Our method can efficiently handle networks of
	different sizes without node labels.  Indeed, note that the MACC of a given network is invariant
	under node relabeling, and regardless of the size of the
	network, we obtain a low-dimensional representation in the
	form of a MACC of fixed size, which can be tuned as a user-defined parameter (by making different choices of the
	underlying motif $F$). Also, as we have discussed in the previous paragraph, MACCs are interpretable in terms of the average clustering structure of the network so we can interpret the result of a clustering algorithm based on MACCs.
	
	

	\vspace{0.1cm}
	
	\section{Application II: Textual analysis and Word Adjacency Networks}
	\label{section:applications}

	\textit{Function word adjacency networks} (WANs) are weighted networks introduced by Segarra, Eisen, and Ribeiro in the context of authorship attribution \citep{segarra2015authorship}. Function words are words that are used for the grammatical purpose and do not carry lexical meaning on their own, such as \textit{the}, \textit{and}, and \textit{a} (see \cite{segarra2015authorship} for the full list of 211 function words). After fixing a list of $n$ function words, for a given article $\mathcal{A}$, we construct a $(n\times n)$  \textit{frequency matrix} $M(\mathcal{A})$ whose $(i,j)$ entry $m_{ij}$ is the number of times that the $i$th function word is followed by the $j$th function word within a forward window of $D=10$ consecutive words (see  \cite{segarra2015authorship} for details). For a given article $\mathcal{A}$, we associate a network $\G(\mathcal{A})=([n], A, \alpha)$, where $\alpha\equiv 1/n$ is the uniform node weight on the function words and $A$ is a suitable matrix obtained from normalizing the frequency matrix $M(\mathcal{A})$. In \cite{segarra2015authorship}, the authors used row-wise normalization to turn the frequency matrix into a Markov transition kernel and then used Kullback–Leibler (KL) divergence to compare them for a classification task. Use the same normalization for the same purpose (see Table \ref{table:WAN_classificaiton_rate}). In all other simulations, we use the global normalization $A = M(\mathcal{A})/\max(M(\mathcal{A}))$ as it leads to more visually distinctive CHD profiles among different authors (see, e.g., Figure \ref{fig:WAN_profile_00}).

	The particular data set we will analyze in this section consists of the following 45 novels of the nine authors listed below:  
	\begin{description}
		\item{1. \textit{Jacob Abbott}:} Caleb in the Country, Charles I, Cleopatra, Cyrus the Great, and Darius the Great
		
		\item{2. \textit{Thomas Bailey Aldrich}:} Marjorie Daw, The Cruise of the Dolphin, The Little Violinist, Mademoiselle Olympe Zabriski, and A Midnight Fantasy
		
		\item{3. \textit{Jane Austen}:} Northanger Abbey, Emma, Mansfield Park, Pride and Prejudice, and Sense and Sensibility
		
		\item{4. \textit{Grant Allen}:} The British Barbarians, Biographies of Working Men, Anglo-Saxon Britain, Charles Darwin, and An African Millionaire

		\item{5. \text{Charles Dickens}:} A Christmas Carol, David Copperfield, Bleak House, Oliver Twist, and Holiday Romance
		
		\item{6. \textit{Christopher Marlowe}:} Edward the Second, The Tragical History of Doctor Faustus, The Jew of Malta, Massacre at Paris, and Hero and Leander and Other Poems
		
		\item{7. \textit{Herman Melville}:} Israel Potter, The Confidence-Man, Moby Dick; or The Whale, Omoo: Adventures in the South Seas, and Typee

		\item{8. \textit{William Shakespeare}:} Hamlet, Henry VIII, Julius Cesar, King Lear, and Romeo and Juliet
		
		\item{9. \textit{Mark Twain}:} Adventures of Huckleberry Finn, A Horse's Tale, The Innocents Abroad,  The Adventures of Tom Sawyer, and A Tramp Abroad
	\end{description}
	
	
	The frequency matrices corresponding to the above novels are recorded using a list of $n=211$ function words (see supplementary material of \cite{segarra2015authorship}). These matrices are sparse and spiky, meaning that most entries are zero and that there are a few entries that are very large compared to the others. For their visual representation, in the first row in Figure \ref{fig:intro_sim}, we plot the heat map of some of the frequency matrices after a double `log transform' $A\mapsto \log(A+\mathbf{1})$ and then normalization $B\mapsto B/\max(B)$.

	Next, we find that the corresponding WANs contain one large connected component and a number of isolated nodes. This can be seen effectively by performing single-linkage hierarchical clustering on these networks. In Figure \ref{fig:austen_shakespeare_dendro} we plot the resulting single-linkage dendrograms for two novels: "Jane Austen - Pride and Prejudice" and "William Shakespeare  - Hamlet". In both novels, the weight between the function words "of" and "the" is the maximum and they merge at level 1 (last two words in Figure \ref{fig:austen_shakespeare_dendro} top and the fourth and fifth to last in Figure \ref{fig:austen_shakespeare_dendro} bottom). On the other hand, function words such as "yet" and "whomever" are isolated in both networks (first two words in Figure \ref{fig:austen_shakespeare_dendro} top and bottom).

	\subsection{CHD profiles of the novel data set}
	\label{subsection:CHD_WAN}
	
	We compute various CHD profiles of the WANs corresponding to our novel data set. We consider the following three pairs of motifs: (see Example \ref{ex:motifs_Ex})
	\begin{align}
		(H_{0,0},H_{0,0}): &\qquad  H_{0,0} = (\{0\}, \mathbf{1}_{\{(0,0)\}}) \\
		(F_{0,1},F_{0,1}): &\qquad F_{0,1} = (\{0,1\}, \mathbf{1}_{\{(0,1)\}}) \\
		(H_{1,1},F_{1,1}): &\qquad H_{1,1} = (\{0,1,2\}, \mathbf{1}_{\{(1,2)\}}),\qquad  F_{1,1} = (\{0,1,2\}, \mathbf{1}_{\{(0,1), (0,2)\}}).
	\end{align}
	The CHD profiles of WANs corresponding to the 45 novels are given in Figures \ref{fig:WAN_profile_00}, \ref{fig:WAN_profile_01}, and \ref{fig:WAN_profile_11}. 
	
	Figure \ref{fig:WAN_profile_00} as well as the second row of Figure \ref{fig:intro_sim} below show the CHD profiles $\mathtt{f}(H_{0,0},\G\,|\, H_{0,0})$ for the pair of `self-loop' motifs. At each filtration level $t\in [0,1]$, the value $\mathtt{f}(t)$ of the profile, in this case, means roughly the density of self-loops in the network $\G_{\mathcal{A}}$ whose edge weight exceed $t$. In terms of the function words, the larger value of $\mathtt{f}(t)$ indicates that more function words are likely to be repeated in a given $D=10$ chunk of words. All of the five CHD profiles for Jane Austen drop to zero quickly and vanishes after $t=0.4$. This means that in her five novels, function words are not likely to be repeated frequently in a short distance. This is in a contrast to the corresponding five CHD profiles for Mark Twain. The rightmost long horizontal bars around height $0.4$ indicate that, among the function words that are repeated within a 10-ward window at least once, at least $40\%$ of them are repeated almost with the maximum frequency. In this regard, from the full CHD profiles given in Figure \ref{fig:WAN_profile_00}, the nine authors seem to divide into two groups. Namely, Jane Austen, Christopher Marlowe, and William Shakespeare have their $(0,0)$ CHD profiles vanishing quickly (less frequent repetition of function words), and the other five with persisting $(0,0)$ CHD profiles (more frequent repetition of function words).

	Figure \ref{fig:WAN_profile_01} shows the CHD profiles $\mathtt{f}(F_{0,1},\G\,|\, F_{0,1})$. The value $\mathtt{f}(t)$ of the CHD profile in this case can be viewed as the tail probability of a randomly chosen edge weight in the network, where the probability of each edge $(i,j)$ is proportional to the weight $A(i,j)$. The CHD profiles for Mark Twain seem to persist longer than that of Jane Austen as in the self-loop case, the difference is rather subtle in this case. 
	
	Lastly, Figure \ref{fig:WAN_profile_11} shows the CHD profiles $\mathtt{f}(H_{1,1},\G\,|\, F_{1,1})$. The value $\mathtt{f}(t)$ of the CHD profile in this case can be regarded as a version of the average clustering coefficient for the corresponding WAN (see Example \ref{ex:avg_clustering_coeff}). Namely, the value $\mathtt{f}(t)$ of the profile at level $t$ is the conditional probability that two random nodes with a common neighbor are connected by an edge with intensity $\ge t$. In terms of function words, this is the probability that if we randomly choose three function words $x,y$, and $z$ such that $x$ and $y$ are likely to appear shortly after $z$, then $y$ also appear shortly after $x$ with more than a proportion $t$ of all times. The corresponding profiles suggest that for Jane Austen, two function words with commonly associated function words are likely to have a very weak association. On the contrary, for Mark Twain, function words tend to be more strongly clustered. From Figure \ref{fig:WAN_profile_11}, one can see that the $(1,1)$ CHD profile of Shakespeare exhibits fast decay in a manner similar to Jane Austen's CHD profiles. While the five CHD profiles of most authors are similar, Grant Allan and Christopher Marlowe show somewhat more significant differences in their CHD profiles among different novels.

	\subsection{Authorship attribution by CHD profiles}
	
	In this subsection, we analyze the CHD profiles of the dataset of novels more quantitatively by computing the pairwise $L^{1}$-distances between the CHD profiles. Also, we discuss an application in authorship attribution. 
	
	In order to generate the distance matrices, we partition the 45 novels into `validation set' and `reference set' of sizes 9 and 36, respectively, by randomly selecting a novel for each author. Note that there are a total $5^9$ such partitions. For each article $i$ in the validation set, and for each of the three pairs of motifs, we compute the $L^{1}$-distance between the corresponding CHD profile of the article $i$ and the mean CHD profile of each of the nine authors, where the mean profile for each author is computed by averaging the four profiles in the reference set. This will give us a $9\times 9$ matrix of $L^{1}$-distances between the CHD profiles of the nine authors. We repeat this process for $10^4$ iterations to obtain a $9\times 9 \times 10^4$ array. The average of all $10^4$ distance matrices for each of the three pairs of motifs are shown in  Figure \ref{fig:WAN_dist_mx}.
	
	\begin{figure*}[h]
		\centering
		\includegraphics[width=1 \linewidth]{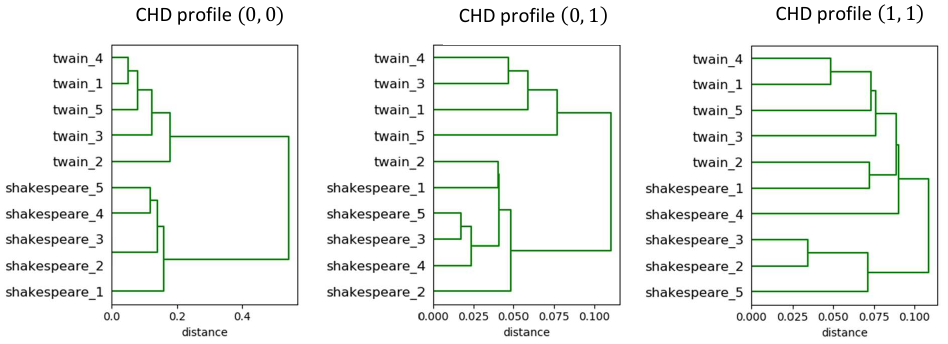}
		\caption{ Single-linkage dendrogram of the $L^{1}$-distance matrices between the CHD profiles of the 10 novels of Shakespeare and Twain for the pair of motifs $(H_{00}, F_{00})$ (left), $(H_{01},F_{01})$ (middle), and $(H_{11},F_{11})$ (right). Most texts of the same author fall into the same cluster.  
		}
		\label{fig:WAN_dendro}
	\end{figure*}
	
	For instance, consider the middle plot in Figure \ref{fig:WAN_dist_mx}. 
	The plot suggests that Jane Austen and William Shakespeare have small $L^{1}$-distance with respect to their CHD profiles $\mathtt{f}(F_{0,1},\G\,|\, F_{0,1})$. From the full list of CHD profiles given in Figure \ref{fig:WAN_profile_01}, we can see `why' this is so: while their CHD profiles drop to zero quickly around filtration level 0.5, all the other authors have more persisting CHD profiles. 
	
	Next, note that if we restrict the three distance matrices to the last two authors (Twain and Shakespeare), then the resulting $2\times 2$ matrices have small diagonal entries indicating that the two authors are well-separated by the three profiles. Indeed, in Figure \ref{fig:WAN_dendro}, we plot the single-linkage hierarchical clustering dendrogram for the $L^{1}$-distances across the 10 novels by the two authors. In the first dendrogram in Figure \ref{fig:WAN_dendro}, we see the five novels from each author form perfect clusters according to the $(0,0)$-CDH profile. For the other two dendrograms, we observe near-perfect clustering using $(0,1)$- and $(1,1)$-CHD profiles.

	Lastly, in Table \ref{table:WAN_classificaiton_rate}, we apply our method as a means of authorship attribution and compare its correct classification rate with other baseline methods. We choose five novels for each author as described at the beginning of this section. In this experiment, for each article $\mathcal{A}$, we normalize its frequency matrix $M(\mathcal{A})$ row-wise to make it a Markov transition kernel and then calculate pairwise distances between them by three methods -- $L^{1}$-the distance between $(0,0)$-CHD profiles, the KL-divergence, and the Frobenius distance. This normalization is used in the original article \citep{segarra2015authorship}, and we find that this generally leads to a higher classification rate than the global normalization $M(\mathcal{A})\mapsto M(\mathcal{A})/\max(M(\mathcal{A}))$. For the classification test, we first choose $k\in \{1,2,3,4\}$ known texts and one text with unknown authorship from each author. For each unknown text $X$, we compute its distance from the $5k$ texts of known authorship and attribute $X$ to the author of known texts of the minimum distance. The classification rates after repeating this experiment 1000 times are reported in Table \ref{table:WAN_classificaiton_rate}.

	\begin{table*}[h]
		\centering
		\includegraphics[width=0.95 \linewidth]{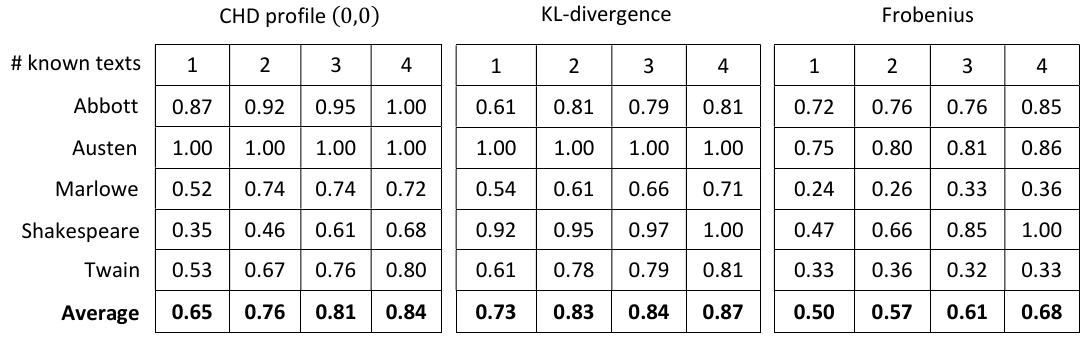}
		\vspace{0.3cm}
		\caption{ Success rate of authorship attribution by using CHD profiles, the KL divergence, and the Frobenius metric for various numbers of known texts per author.  
		}
		\label{table:WAN_classificaiton_rate}
	\end{table*}
	
	The table above summarizes classification rates among the five authors - Abbott, Austen, Marlowe, Shakespeare, and Twain. For four known texts per author, the CHD profile gives 84\% success rate, which outperforms the Frobenius distance (68\%) and shows similar performance as the KL divergence (87\%). It is also interesting to note that different metric shows complementary classification performance for some authors. For instance, for four known tests, Abbott is perfectly classified by the CHD profile, whereas KL-divergence has only \%81 success rate; on the other hand, Shakespeare is perfectly classified by the KL-divergence but only with \%68 accuracies with the CHD profile. We also report the average classification rates for all nine authors: CHD profile $(0,0)$ -- 53\%, KL-divergence -- 77\%, and Frobenius -- 41\%. The $(0,0)$-profile loses the classification score mainly for Aldrich (author index 1), Dickens (author index 5), and Melville (author index 7). Indeed, in Figure \ref{fig:WAN_dist_mx} left, we see the diagonal entries of these authors are not the smallest in the corresponding rows. A WAN is already a compressed mathematical summary of text data, so running an additional MCMC motif sampling algorithm and further compressing it to a profile may lose information that could simply directly be processed. We emphasize that, as we have seen in Section \ref{section:FB} as well as in Subsection \ref{subsection:CHD_WAN}, our method is more suitable for extracting interpretable and low-dimensional information from large networks.

	\section*{Acknowledgements}
	{This work has been partially supported by NSF projects DMS-1723003 and CCF-1740761. HL is partially supported by NSF DMS-2206296 and DMS-2010035. The authors are grateful to Mason Porter for sharing the Facebook100 dataset, and also to Santiago Segarra and Mark Eisen for sharing the Word Adjacency Network dataset and original codes.  }

	\vspace{1cm}

	\newpage
	
	\appendix

	\section{Motif transforms and spectral decomposition}
	\label{section:motif_transform_spectral}
	
	In this section, we compute the motif transform by paths using a certain  spectral decomposition and consider motif transforms in terms of graphons. We denote the  path and cycle motifs by $P_{k}=([k], \mathbf{1}_{\{(1,2),(2,3),\cdots,(k-1,k)\}})$ and $C_{k}=([k], \mathbf{1}_{\{(1,2),\cdots,(k-1,k), (k,1)\}}),$ respectively. 
	
	\subsection{Motif transform by paths} 
	\label{subsection:path_transform}
	For any function $f:[n]\rightarrow [0,1]$, denote by $\text{diag}(f)$ the $(n\times n)$ diagonal matrix whose $(i,i)$ entry is $f(i)$. For a given network $\G=([n],A,\alpha)$, observe that 
	\begin{align}
		\mathtt{t}(P_{k},\G) \alpha(x_{1})^{-1/2} A^{P_{k}}(x_{1},x_{k}) \alpha(x_{k})^{-1/2} &=   \sum_{x_{2},\cdots,x_{k-1}\in [n]} \prod_{\ell=1}^{k-1} \sqrt{\alpha(x_{\ell})}A(x_{\ell},x_{\ell+1})\sqrt{\alpha(x_{\ell+1})} \\
		& = \left[ \left(\text{diag}(\sqrt{\alpha})\, A \, \text{diag}(\sqrt{\alpha})\right)^{k-1} \right]_{x_{1},x_{k}}.
	\end{align}
	If we denote $B=\text{diag}(\sqrt{\alpha})\, A \, \text{diag}(\sqrt{\alpha})$, this yields 
	\begin{align}
		\mathtt{t}(P_{k},\G)  A^{P_{k}} = \text{diag}(\sqrt{\alpha}) B^{k-1} \text{diag}(\sqrt{\alpha}).
	\end{align} 
	
	Since $B$ is a real symmetric matrix, its eigenvectors form an orthonormal basis of $\mathbb{R}^{n}$. Namely, let $\lambda_{1}\ge \lambda_{2}\ge \cdots \ge \lambda_{n}$ be the eigenvalues of $B$ and let $v_{i}$ be the corresponding eigenvector of $\lambda_{i}$. Then $v_{i}$ and $v_{j}$ are orthogonal if $i\ne j$. Furthermore, we may normalize the eigenvectors so that if we let $V$ be the $(n\times n)$ matrix whose $i$th column is $v_{i}$, then $V^{T}V = I_{n}$, the $(n\times n)$ identity matrix. The spectral decomposition for $B$ gives  $B=V\, \text{diag}(\lambda_{1},\cdots,\lambda_{n})\, V^{T}$. Hence  
	\begin{align}
		\mathtt{t}(P_{k},\G) A^{P_{k}}  = \text{diag}(\sqrt{\alpha}) V\, \text{diag}(\lambda_{1}^{k-1},\cdots,\lambda_{n}^{k-1})\, V^{T} \text{diag}(\sqrt{\alpha}),
	\end{align} 
	or equivalently,  
	\begin{equation}\label{eq:path_transform_matrix_formula}
		\mathtt{t}(P_{k},\G)  A^{P_{k}}(i,j)  = \sum_{\ell=1}^{n} \lambda_{\ell}^{k-1} \sqrt{\alpha(i)}v_{\ell}(i)\sqrt{\alpha(j)}v_{\ell}(j),
	\end{equation}
	where $v_{\ell}(i)$ denotes the $i$th coordinate of the eigenvector $v_{\ell}$. Summing the above equation over all $i,j$ gives 
	\begin{equation}\label{eq:hom_density_path_spectral}
		\mathtt{t}(P_{k},\G) = \sum_{\ell=1}^{n} \lambda_{\ell}^{k-1} (\langle \sqrt{\alpha},v_{\ell}  \rangle)^{2},
	\end{equation} 
	where $\langle\cdot,\cdot \rangle$ denotes the inner product between two vectors in $\mathbb{R}^{n}$. Combining the last two equations yields 
	\begin{equation}
		A^{P_{k}}(i,j)  = \frac{\sum_{\ell=1}^{n} \lambda_{\ell}^{k-1} \sqrt{\alpha(i)}v_{\ell}(i)\sqrt{\alpha(j)}v_{\ell}(j)}{\sum_{\ell=1}^{n} \lambda_{\ell}^{k-1} \langle \sqrt{\alpha},v_{\ell}  \rangle^{2}},
	\end{equation}
	
	Now suppose $\G$ is irreducible. Then by Perron-Frobenius theorem for nonnegative irrducible matrices, $\lambda_{1}$ is the eigenvalue of $B$ with maximum modulus whose associated eigenspace is simple, and the components of the corresponding normalized eigenvector $v_{1}$ are all positive. This yields $\langle \sqrt{\alpha},v_{1}  \rangle>0$, and consequently
	\begin{equation}\label{eq:Abar_formula}
		\bar{A}:=\lim_{k\rightarrow \infty} A^{P_{k}}  = \frac{1}{\langle \sqrt{\alpha},v_{1} \rangle^{2}} \text{diag}(\sqrt{\alpha}) v_{1} v_{1}^{T} \text{diag}(\sqrt{\alpha}).
	\end{equation}
	If $\G$ is not irreducible, then the top eigenspace may not be simple and $\lambda_{1}=\cdots=\lambda_{r}>\lambda_{r+1}$ for some $1\le r<n$. By decomposing $A$ into irreducible blocks and applying the previous observation, we have $\langle \sqrt{\alpha}, v_{i}\rangle>0$ for each $1\le i \le r$ and 
	\begin{equation}\label{eq:Abar_formula_2}
		\bar{A}:=\lim_{k\rightarrow \infty} A^{P_{k}}  = \frac{1}{\sum_{i=1}^{r}\langle \sqrt{\alpha},v_{1} \rangle^{2}} \text{diag}(\sqrt{\alpha}) \left( \sum_{i=1}^{r} v_{i}v_{i}^{T} \right) \text{diag}(\sqrt{\alpha}).
	\end{equation}
	We denote $\bar{\G} = ([n],\bar{A},\alpha)$ and call this network as the \textit{transitive closure} of $\G$. 
	
	It is well-known that the Perron vector of an irreducible matrix $A$, which is the normalized eigenvector corresponding to the Perron-Frobenius eigenvalue $\lambda_{1}$ of $A$, varies continuously under small perturbation of $A$, as long as resulting matrix is still irreducible \cite{kato2013perturbation}. It follows that the transitive closure $\bar{\G}$ of an irreducible network $\G$ is stable under small perturbation. However, it is easy to see that this is not the case for reducible networks (see Example \ref{ex:instability_trans_closure}).
	
	\commHL{Below we give an example of motif transform by paths and transitive closure of a three-node network.} 
	
	\begin{ex}[Transitive closure of a three-node network]\label{ex:matrix_spectral}
		\normalfont
		Consider a network $\G=([3],A,\alpha)$, where $\alpha=((1-\epsilon)/2,\epsilon,(1-\epsilon)/2)$ and 
		\begin{equation}
			A = \begin{bmatrix}
				1 & s & 0 \\
				s & 1 & s \\
				0 & s & 1
			\end{bmatrix}.
		\end{equation}
		Then $\G$ is irreducible if and only if $s>0$. Suppose $s>0$. Also, note that 
		\begin{equation}\label{eq:ex_power_expansion1}
			\textup{diag}(\sqrt{\alpha})\, A \, \textup{diag}(\sqrt{\alpha}) = 
			\begin{bmatrix}
				(1-\epsilon)/2 & s\sqrt{(1-\epsilon)\epsilon/2} & 0 \\
				s\sqrt{(1-\epsilon)\epsilon/2}  & \epsilon  & s\sqrt{(1-\epsilon)\epsilon/2}  \\
				0 & s\sqrt{(1-\epsilon)\epsilon/2}   & (1-\epsilon)/2
			\end{bmatrix}.
		\end{equation}
		The eigenvalues of this matrix are
		\begin{eqnarray*}
			\lambda_{0} &=& \frac{1-\epsilon}{2} \\
			\lambda_{-} &=&  \frac{1}{4} \left((\epsilon +1)-\sqrt{(3\epsilon -1)^{2} - 16s^{2}\epsilon(1-\epsilon)}\right) \\
			\lambda_{+} &=& \frac{1}{4} \left((\epsilon +1)+\sqrt{(3\epsilon -1)^{2} - 16s^{2}\epsilon(1-\epsilon)}\right)
		\end{eqnarray*}
		and the corresponding eigenvectors are
		\begin{align}
			v_{0} &= (-1,0,1)^{T} \\
			v_{-} &= \left( 1, \frac{3\epsilon-1 - \sqrt{(3\epsilon-1)^{2}+16s^{2}\epsilon(1-\epsilon)}}{2s\sqrt{2\epsilon(1-\epsilon)}} ,1 \right)^{T}\\
			v_{+} &= \left( 1, \frac{3\epsilon-1 + \sqrt{(3\epsilon-1)^{2}+16s^{2}\epsilon(1-\epsilon)}}{2s\sqrt{2\epsilon(1-\epsilon)}} ,1 \right)^{T}
		\end{align}
		The Perron-Frobenius eigenvector of the matrix in (\ref{eq:ex_power_expansion1}) is $v_{+}$. Then using (\ref{eq:Abar_formula}), we can compute
		\begin{align}
			\bar{A} =
			\begin{bmatrix}
				1/4 & 0 & 1/4 \\
				0 & 0 & 0 \\
				1/4 & 0 & 1/4 
			\end{bmatrix}
			+
			\epsilon
			\begin{bmatrix}
				-s & s& -s\\
				s & 0 & s \\
				-s & s &	-s
			\end{bmatrix}
			+ O(\epsilon^{2}).
		\end{align}
		Hence in the limit as $\epsilon\searrow 0$, the transitive closure of $\G$ consists of two clusters with uniform communication strength of 1/4. However, if we change the order of limits, that is, if we first let $\epsilon\searrow 0$ and then $k\rightarrow \infty$, then the two clusters do not communicate in the limit. Namely, one can compute 
		\begin{align}
			A^{P_{k}} =
			\begin{bmatrix}
				1/2 & 0 & 0 \\
				0 & 0 & 0 \\
				0 & 0 & 1/2 
			\end{bmatrix}
			+
			\epsilon
			\begin{bmatrix}
				-(k-1)s^{2}-2s & s & ks^2\\
				s & 0 & s \\
				ks^2 & s &	-(k-1)s^{2}-2s
			\end{bmatrix}
			+ O(\epsilon^{2}),
		\end{align}
		which is valid for all $k\ge 2$. Hence for any fixed $k\ge 1$, the motif transform of $\G$ by $P_{k}$ gives two non-communicating clusters as $\epsilon\searrow 0$.  $\hfill\blacktriangle$
	\end{ex}

	\subsection{Motif transform of graphons}
	
	Recall the $n$-block graphon $U_{\G}:[0,1]^{2}\rightarrow [0,1]$ associated with network $\G=([n],A,\alpha)$, which is introduced in \eqref{eq:def_block_kernel}. For each graphon $U$ and a simple motif $F=([k],\mathbf{1}(E))$ with $k\ge 2$, define a graphon $U^{F}$ by 
	\begin{equation}
		U^{F}(x_{1},x_{k}) = \frac{1}{\mathtt{t}(F,U)}\int_{[0,1]^{k-2}} \prod_{(i,j)\in E} U( x_{i}, x_{j}) \,dx_{2}\cdots dx_{k-1}.
	\end{equation}
	It is easy to verify that the graphon corresponding to the motif transforms $\G^{F}$ agrees with $(U_{\G})^{F}$. Below we give some examples.

	\begin{ex}[path]
		\textup{ Let $P_{k}$ be the path motif on node set $[k]$. Let $U:[0,1]^{2}\rightarrow[0,1]$ be a graphon. Then 
			\begin{equation}
				\mathtt{t}(P_{k},U) U^{P_{k}}(x_{1},x_{k}) =  \int_{[0,1]^{k-2}} U(x_{1},x_{2})U(x_{2},x_{3})\cdots U(x_{k-1},x_{k})\,dx_{2}\cdots dx_{k-1}.
			\end{equation}
			We denote the graphon on the right-hand side as $U^{\circ (k-1)}$, which is called the $(k-1)$st power of $U$. }$\hfill \blacktriangle$
	\end{ex}

	\begin{ex}[cycle]
		\textup{Let $C_{k}$ be a cycle motif on node set $[k]$. Let $U:[0,1]^{2}\rightarrow[0,1]$ be a graphon. Then 
			\begin{align}
				\mathtt{t}(C_{k},U) U^{C_{k}}(x_{1},x_{k}) &=  U(x_{1},x_{k})\int_{[0,1]^{k-2}} U(x_{1},x_{2})\cdots U(x_{k-1},x_{k})\,dx_{2}\cdots dx_{k-1} \\
				&= U(x_{1},x_{k}) U^{\circ (k-1)}(x_{1},x_{k}).
			\end{align}
		}
		$\hfill \blacktriangle$
	\end{ex}
	
	\commHL{Below, we give an explicit example of motif transforms applied to graphons.}
	
	\begin{ex}\label{ex:graphon_example}
		\textup{Let $\G=([3],A,\alpha)$ be the network in Example \ref{ex:matrix_spectral}. Let $U_{\G}$ be the corresponding graphon. Namely, let $[0,1]=I_{1}\sqcup I_{2}\sqcup I_{3}$ be a partition where $I_{1}=[0,(1 - \eps)/2)$, $I_{2} = [(1-\eps)/2,(1 + \eps)/2)$, and $I_{3} = [(1+\eps)/2, 1]$. Then $U_{\G}$ is the 3-block graphon taking value $A(i,j)$ on rectangle $I_{i}\times I_{j}$ for $1\le i,j\le 3$.  Denoting $U=U_{\G}$, the three graphons $U$, $U^{\circ 2}$, and $U\cdot U^{\circ 2}$ are shown in Figure \ref{fig:motif_transform_ex}.} 
		
		\begin{figure*}[h]
			\centering
			\includegraphics[width=0.95 \linewidth]{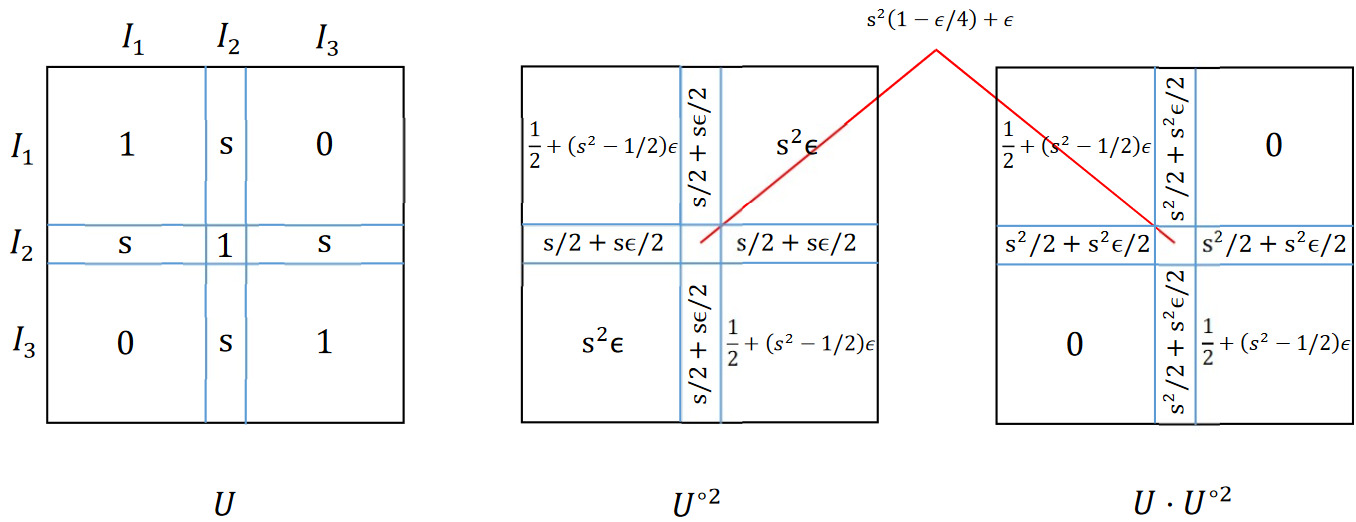}
			\caption{ Graphons $U=U_{\G}$ (left), $U^{\circ 2}$ (middle), and $U\cdot U^{\circ 2 }$ (right). \commHL{These are three-block graphons with the same block structure $I_{i}\times I_{j}$ for $1\le i,j\le 3$ whose values on each block are depicted in the figure. }
			}
			\label{fig:motif_transform_ex}
		\end{figure*}
		
		\noindent \textup{According to the previous examples,  we have 
			\begin{equation}
				U^{P_{3}} = \frac{U^{\circ 2}}{\mathtt{t}(P_{3},U)}, \quad U^{C_{3}}= \frac{U\cdot U^{\circ 2}}{\mathtt{t}(C_{3},U)},
			\end{equation}
			where $\mathtt{t}(P_{3},U)=\lVert U^{\circ 2} \rVert_{1}$ and $\mathtt{t}(P_{3},U)=\lVert U\cdot U^{\circ 2} \rVert_{1}$. See Figure \ref{fig:graphon_dendro} for hierarchical clustering dendrograms of these graphons.} $\hfill \blacktriangle$
	\end{ex}

	\subsection{Spectral decomposition and motif transform by paths}
	
	In this subsection, we assume all kernels and graphons are symmetric. 
	
	A graphon $W:[0,1]^{2}\rightarrow [0,1]$ induces a compact  Hilbert-Schmidt operator $T_{W}$ on $\mathcal{L}^{2}[0,1]$ where 
	\begin{equation}
		T_{W}(f)(x)=\int_{0}^{1}W(x,y)f(y)\,dy.
	\end{equation} 
	$T_W$ has a discrete spectrum, i.e., its spectrum is a countable multiset $\text{Spec}(W)=\{ \lambda_{1}, \lambda_{2},\cdots \}$, where each eigenvalue has finite multiplicity and $|\lambda_{n}|\rightarrow 0$ as $n\rightarrow \infty$. Since $W$ is assumed to be symmetric, all $\lambda_{i}$s are real so we may arrange them so that $\lambda_{1}\ge \lambda_{2}\ge \cdots$. Via a spectral decomposition, we may write  
	\begin{equation}\label{eq:spectral_decomp}
		W(x,y) = \sum_{j=1}^{\infty} \lambda_{j} f_{j}(x) f_{j}(y),
	\end{equation}
	where $\int_{0}^{1}f_{i}(x)f_{j}(x)\,dx = \mathbf{1}(i=j)$, that is, $f_{j}$ is an eigenfunction associated to $\lambda_{j}$ and they form an orthonormal basis for $\mathcal{L}^{2}[0,1]$. 	
	
	Let $P_{k}$ be the path on the node set $[k]$. Let $U$ be a graphon with eigenvalues $\lambda_{1}\ge \lambda_{2}\ge \cdots $. Orthogonality of the eigenfunctions easily yields 
	\begin{equation}\label{path_transform_graphon_formula}
		U^{P_{k}}(x,y) = \frac{\sum_{j} \lambda_{j}^{k} f_{j}(x) f_{j}(y)}{\sum_{j} \lambda_{j}^{k} \left( \int f_{j}(x_{1}) \,dx_{1} \right)^{2}},
	\end{equation}
	Further, suppose the top eigenvalue of $U$ has multiplicity $r\ge 1$. Then
	\begin{equation}\label{eq:path_transform_graphon_closure}
		\bar{U}(x,y)=\lim_{k\rightarrow \infty} U^{P_{k}} (x,y) = \frac{\sum_{j=1}^{r}  f_{j}(x) f_{j}(y)}{\sum_{j=1}^{r}  \left( \int f_{j}(x_{1}) \,dx_{1} \right)^{2}}. 
	\end{equation}
	Note that (\ref{path_transform_graphon_formula}) and (\ref{eq:path_transform_graphon_closure}) are the graphon analogues of formulas (\ref{eq:path_transform_matrix_formula}) and (\ref{eq:Abar_formula_2}).
	
	The network and graphon versions of these formulas are compatible through the following simple observation.

	\begin{prop}
		Let $\G=([n],A,\alpha)$ be a network such that $A$ is symmetric, and let $U=U_{\G}$ be the corresponding graphon \commHL{(see \eqref{eq:def_block_kernel})}. Let $\lambda\in \mathbb{R}$ and $v=(v_{1},\cdots,v_{n})^{T}\in \mathbb{R}^{n}$ be a pair of an eigenvalue and its associate eigenvector of the matrix $B=\text{diag}(\sqrt{\alpha})\, A \, \text{diag}(\sqrt{\alpha})$. Then the following function $f_{v}:[0,1]\rightarrow \mathbb{R}$ 
		\begin{equation}
			f_{v}(x) = \sum_{i=1}^{n} \frac{v_{i}}{\sqrt{\alpha(i)}} \mathbf{1}(x\in I_{i})
		\end{equation}
		is an eigenfunction of the integral operator $T_{U}$ associated to the eigenvalue $\lambda$. Conversely, every eigenfunction of $T_{U}$ is given this way. 
	\end{prop}      
	
	\commHL{
		\begin{proof}
			First observe that any eigenfunction $f$ of $T_{U}$ must be constant over each interval $I_{i}$. Hence we may write $f=\sum a_{i}\mathbf{1}(I_{i})$ for some $a_{i}\in \mathbb{R}$. Then for each $x\in [0,1]$,
			\begin{eqnarray}
				T_{U}(f_{v})(x) &=& \int_{0}^{1} U(x,y)f(y) \,dy \\  
				&=& \int_{0}^{1} \sum_{i,j,k} A(i,j) \mathbf{1}(x\in I_{i}) \mathbf{1}(y\in I_{j}) \, a_{k}\,\mathbf{1}(y\in I_{k})  \,dy \\  
				&=& \sum_{i} \mathbf{1}(x\in I_{i}) \sum_{j}   A(i,j) \,\alpha(j) \, a_{j}.
			\end{eqnarray}
			Hence $f$ is an eigenfunction of $T_{U}$ with eigenvalue $\lambda$ if and only if 
			\begin{equation}\label{eq:evec_graph_graphon_1}
				\sum_{j=1}^{n}  A(i,j) \alpha(j) a_{j} = \lambda a_{i} \quad \text{$\forall 1\le i \le n$},
			\end{equation}
			which is equivalent to saying that $u:=(a_{1},\cdots,a_{n})^{T}$ is an eigenvector of the matrix $A\,\text{diag}(\alpha)$ with eigenvalue $\lambda$. Further note that $A\,\text{diag}(\alpha)u = \lambda u $ is equivalent to 
			\begin{equation}
				B\, \text{diag}(\sqrt{\alpha})u = \lambda \text{diag}(\sqrt{\alpha})u.
			\end{equation}
			Writing $v_{i}:=a_{i}\sqrt{\alpha(i)}$, then shows the assertion. 
		\end{proof}
	}
	
	\commHL{
		\begin{rmkk}
			\normalfont
			\cite[Lem. 2]{ruiz2021graphon} states a similar observation for associating eigenvalue/eigenvector pairs of a network with that of the associated graphon. While our statement holds for general probability distribution $\alpha$ on the node set $[n]$, in the reference, the uniform probability distribution $\alpha\equiv 1/n$ is assumed. In this special case, \eqref{eq:evec_graph_graphon_1} reduces to 
			\begin{align}
				A u = \frac{\lambda}{n} u,
			\end{align}
			as stated in \cite[Lem. 2]{ruiz2021graphon}. 
		\end{rmkk}
	}
	
	When a graphon $U$ is not irreducible, its top eigenspace is not simple and its dimension can change under an arbitrarily small perturbation. Hence formula (\ref{eq:path_transform_graphon_closure}) suggests that the operation of transitive closure $U\rightarrow \bar{U}$ is not stable under any norm. The following example illustrates this. 
	
	\begin{ex}[Instability of transitive closure]
		\label{ex:instability_trans_closure}
		\textup{Let $f_{1}=\mathbf{1}([0,1])$ and choose a function $f_{2}:[0,1]\rightarrow \{-1,1\}$ so that $\int_{0}^{1}f_{2}(x)\,dx=0$. Then $\lVert f_{2} \rVert_{2}=1$ and $\langle f_{1},f_{2}  \rangle=0$. Now fix $\epsilon>0$, and define two graphons $U$ and $U_{\epsilon}$ through their spectral decompositions 
			\begin{equation}
				U = f_{1}\otimes f_{1}+f_{2}\otimes f_{2} \quad \text{and}\quad 
				U_{\epsilon} = f_{1}\otimes f_{1}+(1-\epsilon)f_{2}\otimes f_{2},
			\end{equation}
			where $(f_{i}\otimes f_{j})(x,y) = f_{i}(x)f_{j}(y)$. Then by (\ref{eq:path_transform_graphon_closure}), we get $\bar{U} = U$ and $\bar{U}_{\epsilon} = f_{1}\otimes f_{1}$. This yields 
			\begin{equation}
				\eps(\bar{U}-\bar{U}_{\epsilon}) = \eps f_{2}\otimes f_{2} = U-U_{\epsilon}. 
			\end{equation}
		}
		$\hfill\blacktriangle$
	\end{ex}

	\vspace{0.3cm}
	
	\section{Proof of convergence and mixing time bounds of the Glauber and pivot chains}
	\label{section:proofs_mixing}
	In this section, we establish convergence and mixing properties of the Glauber and pivot chains of homomorphisms $F\rightarrow \G$ by proving Theorems \ref{thm:stationary_measure_Glauber}, \ref{thm:gen_coloring_mixing_tree}, \ref{thm:stationary_measure_pivot}, \ref{thm:pivot_chain_mixing}, and Corollary \ref{cor:time_avg_observables}.

	\vspace{0.2cm}
	\subsection{Convergence and mixing of the pivot chain}
	
	Let $(\x_{t})_{t\ge 0}$ be a pivot chain of homomorphisms $F\rightarrow \G$. We first show that the pivot chain converges to the desired distribution $\pi_{F\rightarrow \G}$ over $[n]^{[k]}$, defined in (\ref{eq:def_embedding_F_N}). Recall the $\alpha$ is the unique stationary distribution of the simple random walk on $\G$ with the modified kernel \eqref{eq:MH_RW_kernel_G}. In this subsection, we write a rooted tree motif $F=([k],A_{F})$ as $([k],E_{F})$, where $E_{F}=\{(i,j)\in [k]^{2}\,|\, A_{F}(i,j)=1\}$. 
	
	\vspace{0.2cm}
	\begin{proof}{\textbf{of Theorem \ref{thm:stationary_measure_pivot} }.}
		Since the network $\G$ is irreducible and finite, the random walk $(\x_{t}(1))_{t\ge 0}$ of pivot on $\G$ with kernel $P$ defined at \eqref{eq:MH_RW_kernel_G} is also irreducible. It follows that the pivot chain is irreducible with a unique stationary distribution, say, $\pi$. We show $\pi$ is in fact the desired measure $\pi_{F\rightarrow \G}$. First, recall that $\x_{t}(1)$ is a simple random walk on the network $\G$ modified by the Metropolis-Hastings algorithm so that it has the following marginal distribution as its unique stationary distribution: (see, e.g., \cite[Sec. 3.2]{levin2017markov})
		\begin{equation}
			\pi^{(1)}(x_{1}) = \frac{\sum_{x_{2},\cdots,x_{k}\in [n]} \prod_{(i,j)\in E_{F}}A(x_{i},x_{j}) \alpha(x_{1})\alpha(x_{2})\cdots \alpha(x_{k}) }{\mathtt{t}(F,\G)}
		\end{equation} 
		Second, we decompose $\x_{t}$ into return times of the pivot $\x_{t}(1)$ to a fixed node $x_{1}\in [n]$ in $\G$. Namely, let $\tau(\ell)$ be the $\ell$th return time of $\x_{t}(1)$ to $x_{1}$. Then by independence of sampling $\x_{t}$ over $\{2,\cdots,k \}$ for each $t$, the strong law of large numbers yields  
		\begin{align}
			&\lim_{M\rightarrow \infty}\frac{1}{M} \sum_{\ell=1}^{M}\mathbf{1}(\x_{\tau(\ell)}(2)=x_{2},\, \cdots, \x_{\tau(\ell)}(k)=x_{k})  \\
			&\qquad \qquad \qquad =  \frac{\prod_{\{i,j \}\in E_{F}} A(x_{i},x_{j}) \alpha(x_{2})\cdots \alpha(x_{k}) }{\sum_{ x_{2},\cdots,x_{k}\in [n]} \prod_{\{i,j \}\in E_{F}} A(x_{i},x_{j}) \alpha(x_{2})\cdots \alpha(x_{k}) }.
		\end{align} 
		
		Now, for each fixed homomorphism $\x:F\rightarrow \G$, $i\mapsto x_{i}$, we use the Markov chain ergodic theorem and previous estimates to write 
		\begin{align}
			\pi(\x) &= \lim_{N\rightarrow \infty} \frac{1}{N} \sum_{t=0}^{N} \mathbf{1}(\x_{t}=\x) \\
			& =  \lim_{N\rightarrow \infty} \frac{\sum_{t=0}^{N}   \mathbf{1}(\x_{t}=\x) }{\sum_{t=0}^{N} \mathbf{1}(\x_{t}(1)=x_{1})}   \frac{\sum_{t=0}^{N} \mathbf{1}(\x_{t}(1)=x_{1})}{N}\\
			&=\frac{\prod_{\{i,j \}\in E_{F}} A(x_{i},x_{j}) \alpha(x_{2})\cdots \alpha(x_{k}) }{\sum_{1\le x_{2},\cdots,x_{k}\le n} \prod_{\{i,j \}\in E_{F}} A(x_{i},x_{j}) \alpha(x_{2})\cdots \alpha(x_{k}) } \pi^{(1)}(x_{1}) \\
			&= \frac{\prod_{\{i,j \}\in E_{F}} A(x_{i},x_{j}) \alpha(x_{1})\alpha(x_{2})\cdots \alpha(x_{k})}{\mathtt{t}(F,\G)} = \pi_{F\rightarrow \G}(\x).
		\end{align} 
		This shows the assertion.
	\end{proof}
	
	Next, we bound the mixing time of the pivot chain. Our argument is based on the well-known bounds on the mixing time and meeting time of random walks on graphs.
	
	\vspace{0.2cm}
	\begin{proof}{\textbf{of Theorem \ref{thm:pivot_chain_mixing}}.}
		Fix a rooted tree motif $F=([k],E_{F})$ and a network $\G=([n],A,\alpha)$. Let $P$ denote the transition kernel of the random walk of pivot on $\G$ given at \eqref{eq:MH_RW_kernel_G}. Note that (ii) follows immediately from the equality in (i) and known bounds on mixing times of random walks (see, e.g., \cite[Thm 12.3 and 12.4]{levin2017markov}).   
		
		Now we show (i). The entire pivot chain and the random walk of the pivot have the same mixing time after each move of the pivot, since the pivot converges to the correct marginal distribution $\pi^{(1)}$ induced from the joint distribution $\pi_{F\rightarrow \G}$, and we always sample the non-pivot nodes from the correct distribution conditioned on the location of the pivot. To make this idea more precise, let $\y:[k]\rightarrow [n]$ be an arbitrary homomorphism $F\rightarrow \G$ and let $(\x_{t})_{t\ge 0}$ denote the pivot chain $F\rightarrow \G$ with $\x_{0}=\y$. Write $\pi = \pi_{F\rightarrow \G}$ and $\pi_{t}$ for the distribution of $\x_{t}$. Let $\pi^{(1)}$ denote the unique stationary distribution of the pivot $(\x_{t}(1))_{t\ge 0}$. Let $\x:F\rightarrow \G$ be a homomorphism and write $\x(1)=x_{1}$. Then for any $t\ge 0$, note that  
		\begin{align}
			\P(\x_{t} = \x\,|\, \x_{t}(1)=x_{1}) & = \frac{\left( \prod_{(i,j)\in E_{F}}  A(x_{i},x_{j}) \right)\, \alpha(x_{2})\cdots \alpha(x_{k})  }{\sum_{x_{2},\cdots,x_{k}\in [n]}\left( \prod_{(i,j)\in E_{F}}  A(x_{i},x_{j}) \right)\, \alpha(x_{2})\cdots \alpha(x_{k})  }  \\
			& = \P_{\pi}(\x_{t} = \x\,|\, \x_{t}(1)=x_{1}).
		\end{align}
		Hence we have 
		\begin{align}
			|\pi_{t}(\x) - \pi(\x)| = |\P(\x_{t}(1)=x_{1}) -  \pi^{(1)}(x_{1})|  \cdot \P(\x_{t} = \x\,|\, \x_{t}(1)=x_{1}).
		\end{align}
		Thus summing the above equation over all homomorphisms $\x:F\rightarrow \G$, we get
		\begin{align}\label{eq:pivot_mixing_pf_2}
			\lVert \pi_{t} - \pi \rVert_{\text{TV}}  &= \frac{1}{2} \sum_{\x:[k]\rightarrow [n]} |\pi_{t}(\x) - \pi(\x)| \\
			&= \frac{1}{2} \sum_{x_{1}\in [n]}  |\P(\x_{t}(1)=x_{1}) -  \pi^{(1)}(x_{1})| \\
			&=   \lVert P^{t}(\y(1), \cdot) -  \pi^{(1)}(x_{1}) \rVert_{\textup{TV}}. 
		\end{align}
		This shows (i).
		
		To show (iii), let $(X_{t})_{t\ge 0}$ and $(Y_{t})_{t\ge 0}$ be two independent random walks on $\G$ with kernel $P$, where at each time $t$ we choose one of them independently with equal probability to move. Let $t_{M}$ be the first time that these two chains meet, and let $\tau_{M}$ be their worst-case expected meeting time, that is, 
		\begin{align}
			\tau_{M} = \max_{x_{0},y_{0}\in [n]} \mathbb{E}[t_{M}\,|\, X_{0}=x_{0}, Y_{0}=y_{0} ].
		\end{align}
		Then by a standard coupling argument and Markov's inequality, we have 
		\begin{align}\label{eq:pivot_mixing_pf_3}
			\lVert P^{t}(x,\cdot) - \alpha \rVert_{\text{TV}}  \le  \mathbb{P}(X_{t}\ne Y_{t}) =  \P(t_{M}>t) \le \frac{\tau_{M}}{t}.
		\end{align}
		By imposing the last expression to be bounded by $1/4$, this yields $t_{mix}(1/4) \le 4\tau_{M}$. Hence we get 
		\begin{align}
			t_{mix}(\eps)\le 4\tau_{M}\log_{2}(\eps^{-1}).
		\end{align}
		Now under the hypothesis in (iii), there is a universal cubic upper bound on the meeting time $\tau_{M}$ due to Coppersmith, Tetali, and Winkler \cite[Thm. 3]{coppersmith1993collisions}. This shows (iii).
	\end{proof}
	
	Lastly in this subsection,  we prove Corollary \ref{cor:time_avg_observables} for the pivot chain. The assertion for the Glauber chain follows similarly from Theorem \ref{thm:stationary_measure_Glauber}, which will be proved in Subsection \ref{subsection:Glauber_chain_mixing}.
	
	\vspace{0.2cm}
	\begin{proof}{\textbf{of Corollary \ref{cor:time_avg_observables}}.}
		Let $F=([k],E_{F})$ be a directed tree motif and $\G=([n],A,\alpha)$ be an irreducible network. Let $(\x_{t})_{t\ge 0}$ be a pivot chain of homomorphisms $F\rightarrow \G$ and let $\pi:=\pi_{F\rightarrow \G}$ be its unique stationary distribution. To show \eqref{eq:profile_ergodic}, note that 
		\begin{align}
			&\lim_{N\rightarrow \infty} \frac{1}{N} \sum_{t=1}^{N}  \prod_{1\le i,j\le k} \mathbf{1}(A(\x_{t}(i),\x_{t}(j))^{A_{H}(i,j)}\ge t)\\ 
			&\qquad =\EE_{\pi}\left[ \prod_{1\le i,j\le k} \mathbf{1}(A(\x_{t}(i),\x_{t}(j))^{A_{H}(i,j)}\ge t) \right] \\
			&\qquad =\P_{F\rightarrow \G}\left( \min_{1\le i,j\le k} A(\x(i),\x(j))^{A_{H}(i,j)}\ge t  \right), 
		\end{align} 
		where the first equality is due to Theorem \ref{thm:observable_time_avg}. In order to show \eqref{eq:cond_hom_ergodic}, note that 
		\begin{align}
			&\lim_{N\rightarrow \infty} \frac{1}{N} \sum_{t=1}^{N}  \prod_{1\le i,j\le k} A(\x_{t}(i),\x_{t}(j))^{A_{H}(i,j)}\\ 
			&\qquad  = \EE_{\pi}\left[ \prod_{1\le i,j\le k} A(\x(i),\x(j))^{A_{H}(i,j)} \right]\\
			& \qquad = \sum_{\x:[k]\rightarrow [n]} \left( \prod_{1\le i,j\le k} A(\x(i),\x(j))^{A_{H}(i,j)} \right) \frac{\left[ \prod_{(i,j)\in E_{F}} A(\x(i),\x(j))\right] \alpha(\x(1))\cdots\alpha(\x(k))}{\mathtt{t}(F,\G)} \\
			& \qquad = \sum_{\x:[k]\rightarrow [n]} \left( \prod_{1\le i,j\le k} A(\x(i),\x(j))^{A_{H}(i,j)+A_{F}(i,j)} \right) \frac{\alpha(\x(1))\cdots\alpha(\x(k))}{\mathtt{t}(F,\G)} \\
			&\qquad =\frac{\mathtt{t}(H,\G)}{\mathtt{t}(F,\G)} = \mathtt{t}(F+H,\G\,|\, F).
		\end{align}
		
		
		For the last equation \eqref{eq:transform_ergodic}, we fix $x_{1},x_{k}\in [n]$. By definition, we have 
		\begin{align}\label{eq:pf_cor_time_avg_motif}
			A^{H}(x_{1},x_{k})  = \EE_{\pi_{H\rightarrow \G}}\left[  \mathbf{1}(\x(1)=x_{1},\, \x(k)=x_{k}) \right].
		\end{align}
		By similar computation as above, we can write 
		\begin{align}
			&\lim_{N\rightarrow \infty} \frac{1}{N} \sum_{t=1}^{N} \left( \prod_{1\le i,j\le k}  A(\x_{t}(i),\x_{t}(j))^{A_{H}(i,j)}\right)\mathbf{1}(\x_{t}(1)=x_{1},\, \x_{t}(k)=x_{k})\\
			&\qquad = \EE_{\pi}\left[ \left( \prod_{1\le i,j\le k}  A(\x_{t}(i),\x_{t}(j))^{A_{H}(i,j)}\right) \mathbf{1}(\x(1)=x_{1},\, \x(k)=x_{k}) \right] \\
			&\qquad = \mathtt{t}(F+H,\G\,|\, F) \,\EE_{\pi_{H\rightarrow \G}}\left[  \mathbf{1}(\x(1)=x_{1},\, \x(k)=x_{k}) \right].
		\end{align}
		Hence the assertion follows from \eqref{eq:pf_cor_time_avg_motif}. 
	\end{proof}

	\subsection{Concentration of the pivot chain and rate of convergence}

	\begin{proof}{\textbf{of Theorem \ref{thm:McDiarmids}}.}
		This is a direct consequence of McDirmid's inequality for Markov chains \cite[Cor. 2.11]{paulin2015concentration} and the first equality in Theorem \ref{thm:pivot_chain_mixing} (i).
	\end{proof}

	Next, we prove Theorem \ref{thm:vector_concentration}. An essential step is given by the following lemma, which is due to \cite{hayes2005large} and  \cite{kallenberg1991some}. Let $\mathcal{H}$ be Hilbert space, and let $(X_{t})_{t\ge 0}$ be a sequence of $\mathcal{H}$-valued random `vectors'. We say it is a \textit{very-weak martingale} if $X_{0} = \mathbf{0}$ and
	\begin{align}
		\EE[X_{t+1}\,|\, X_{t}] = X_{t} \qquad \forall t\ge 0.
	\end{align}
	
	\begin{lemma}[Thm. 1.8 in \citep{hayes2005large}]\label{lemma:very_weak_martingale_concentration}
		Let $(X_{t})_{t\ge 0}$ be a very-weak martingale taking values in a Hilbert space $\mathcal{H}$ and $\lVert X_{t+1} - X_{t}  \rVert \le 1$ for all $t\ge 0$. Then for any $a>0$ and $t\ge 0$, 
		\begin{align}
			\P(\lVert X_{t} \rVert \ge  a) \le 2e^{2} \exp\left( \frac{-a^{2}}{2t} \right).
		\end{align}
	\end{lemma}
	
	\begin{proof}
		The original statement \cite[Thm. 1.8]{hayes2005large} is asserted for a Euclidean space $\EE$ in place of the Hilbert space $\mathcal{H}$. The key argument is given by a discrete-time version of a Theorem of \citet[Thm. 3.1]{kallenberg1991some}, which is proved by Hayes in \cite[Prop. 1.5]{hayes2005large} for Euclidean space. The gist of the argument is that given a very-weak martingale $(X_{t})_{t\ge 0}$ in a Euclidean space with norm $\lVert \cdot \rVert$, we can construct a very-weak martingale $(Y_{t})_{t\ge 0}$ in $R^{2}$ in such a way that 
		\begin{align}
			(\Vert X_{t} \Vert, \Vert X_{t+1} \Vert, \Vert X_{t+1} - X_{t} \Vert ) = (\Vert Y_{t} \Vert_{2}, \Vert Y_{t+1} \Vert_{2}, \Vert Y_{t+1} - Y_{t} \Vert_{2} ).
		\end{align}
		By examining the proof of \citet[Prop. 1.5]{hayes2005large}, one finds that the existence of such a 2-dimensional `local martingale' is guaranteed by an inner product structure and completeness with respect to the induced norm of the underlying space. Hence the same conclusion holds for Hilbert spaces. 
	\end{proof}

	\begin{proof}{\textbf{of Theorem \ref{thm:vector_concentration}}.}
		We use a similar coupling idea that is used in the proof of \citet[Thm. 12.19]{levin2017markov}. Recall that $t_{mix}\equiv t_{mix}^{(1)}$ by Theorem \ref{thm:pivot_chain_mixing} (i). Fix an integer $r\ge t_{mix}^{(1)}(\eps)=t_{mix}(\eps)$. Let $\Omega = [n]^{[k]}$ and fix a homomorphism $x:F\rightarrow \G$ for the initial state of the pivot chain $(\x_{t})_{t\ge 0}$. Let $\pi_{t}$ denote the law of $\x_{t}$ and let $\pi:=\pi_{F\rightarrow \G}$. Let $\mu_{r}$ be the optimal coupling between $\pi_{t}$ and $\pi$, so that 
		\begin{align}
			\sum_{\x\ne \y} \mu_{r}(\x,\y) = \lVert \pi_{t} - \pi \rVert_{TV}.
		\end{align}
		We define a pair $(\y_{t}, \z_{t})$ of pivot chains such that 1) The law of $(\y_{0},\z_{0})$ is $\mu_{r}$ and 2) individually $(\y_{t})_{t\ge 0}$ and $(\z_{t})_{t\ge 0}$ are pivot chains $F\rightarrow \G$, and 3) once these two chains meet, they evolve in unison. Note that $(\y_{t})_{t\ge 0}$ has the same law as $(\x_{r+t})_{t\ge 0}$. Also note that by the choice of $r$ and $\mu_{r}$,   
		\begin{align}\label{eq:pf_vec_concentration_1}
			\P(\y_{0}\ne \z_{0}) = \lVert \pi_{t} - \pi \rVert_{TV} \le \eps.
		\end{align}
		
		Now let $\mathcal{H}$ be a Hilbert space and let $g:\Omega\rightarrow \mathcal{H}$ be any function. By subtracting $\EE_{\pi}(g(\x))$ from $g$, we may assume $\EE_{\pi}(g(\x)) = 0$. Then by conditioning on whether $\y_{0}=\z_{0}$ or not, we have 
		\begin{align}
			\mathbb{P}\left( \left\Vert \sum_{t=1}^{N} g(\x_{r+t})  \right\Vert \ge N\delta \right) &= \mathbb{P}\left( \left\Vert \sum_{t=1}^{N} g(\y_{t})  \right\Vert \ge N\delta \right) \\
			&\le \mathbb{P}\left( \left\Vert \sum_{t=1}^{N} g(\z_{t})  \right\Vert \ge N\delta \right) + \P(\y_{0}\ne \z_{0}).
		\end{align}
		The last term is at most $\eps$ by \eqref{eq:pf_vec_concentration_1}, and we can apply Lemma \ref{lemma:very_weak_martingale_concentration} for the first term. This gives the assertion. 
	\end{proof}

	\subsection{Convergence and mixing of the Glauber chain}
	\label{subsection:Glauber_chain_mixing}

	In this subsection, we consider convergence and mixing of the Glauber chain $(\x_{t})_{t\ge 0}$ of homomorphisms $F\rightarrow \G$. We first investigate under what conditions the Glauber chain is irreducible. 
	
	For two homomorphisms $\x,\x':F\rightarrow \G$, denote $\x\sim \x'$ if they differ by at most one coordinate. Define a graph $\mathcal{S}(F,\G)=(\mathcal{V},\mathcal{E})$ where $\mathcal{V}$ is the set of all graph homomorphisms $F\rightarrow \G$ and $\{ \x,\x' \}\in \mathcal{E}$ if and only if $\x\sim \x'$. We say $\x'$ is \textit{reachable from $\x$ in $r$ steps} if there exists a walk between $\x'$ and $\x$ of length $r$ in $\mathcal{S}(F,\G)$. Lastly, denote the shortest path distance on $\mathcal{S}(F,\G)$ by $d_{F,\G}$. Then $d_{F, \G}(\x,\x')=r$ if $\x'$ is reachable from $\x$ in $r$ steps and $r$ is as small as possible. It is not hard to see that the Glauber chain $(\x_{t})_{t\ge 0}$ is irreducible if and only if $\mathcal{S}(F,\G)$ is connected. In the following proposition, we show that this is the case when $F$ is a tree motif and $\G$ contains an odd cycle.

	\begin{prop}\label{prop:state_space_diam_tree}
		Suppose $F=([k],A_{F})$ is a tree motif and $\G=([n],A,\alpha)$ is irreducible and bidirectional network. Further, assume that the skeleton of $\G$ contains an odd cycle. Then $\mathcal{S}(F,\G)$ is connected and 
		\begin{equation}
			\diam(\mathcal{S}(F,\G)) \le 2k\diam(\G) + 4(k-1).
		\end{equation} 
	\end{prop}

	\begin{proof}	
		We may assume $\mathtt{t}(F,\G)>0$ since otherwise $\mathcal{S}(F,\G)$ is empty and hence is connected. If $k=1$, then each Glauber update is to sample the location of $1$ uniformly at random from $[n]$, so the assertion holds. We may assume $k\ge 2$. 
		
		We first give a sketch of the proof of connectedness of $\mathcal{S}(F,\G)$. Since $\G$ is bidirectional, we can fold the embedding $\x:F\rightarrow \G$ until we obtain a copy of $K_{2}$ (complete graph with two nodes) that is still a valid embedding $F\rightarrow \G$. One can also `contract' the embedding $\x'$ in a similar way. By using irreducibility, then one can walk these copies of $K_{2}$ in $\G$ until they completely overlap. Each of these moves occurs with positive probability since $\G$ is bidirectional, and the issue of parity in matching the two copies of $K_{2}$ can be handled by `going around' the odd cycle in $\G$. 
		
		Below we give a more careful argument for the above sketch. Fix two homomorphisms $\x,\x':F\rightarrow \G$. It suffices to show that $\x'$ is reachable from $\x$ in $2k\diam(\G) + 4(k-1)$ steps. Choose a any two nodes $\ell,\ell'\in [k]$ such that $\ell$ is a leaf in $F$ (i.e., $A_{F}(\ell,i)=0$ for all $i\in [k]$) and they have a common neighbor in $F$ (i.e., $A_{F}(i,\ell)>0$ and $A_{F}(i,\ell')+A_{F}(\ell',i)>0$ for some $i\in [k]$). Consider the vertex map $\x^{(1)}:[k]\rightarrow [n]$ defined by $\x^{(1)}(j)=\x(j)$ for $i\ne j$ and $\x^{(1)}(\ell)=\x^{(1)}(\ell')$. Since $\G$ is bidirectional, we see that $\xi^{(1)}$ is a homomorphism $F\rightarrow \G$. Also note that $\x\sim \x^{(1)}$ and $\x^{(1)}$ uses at most $k-1$ distinct values in $[n]$. By repeating a similar operation, we can construct a sequence of homomorphisms $\x^{(1)},\x^{(2)},\cdots, \x^{(k-2)}=:\y^{(1)}$ such that $\y$ uses only two distinct values in $[n]$.

		Next, let $G$ denote the skeleton of $\G$, which is connected since $\G$ is irreducible and bidirectional. Suppose there exists a walk $W=(a_{1},a_{2},\cdots,a_{2m})$ in $G$ for some integer $m\ge 0$ such that $\y(1)=a_{1}$ and $\x'(1)=a_{2m-1}$. We claim that this implies $\x'$ is reachable from $\y^{(1)}$ in $k(m+1)+k-2$ steps.

		To see this, recall that the walk $W$ is chosen in the skeleton $G$ so that at least one of $A(a_{2},a_{3})$ and $A(a_{3},a_{2}))$ is positive. Hence with positive probability, we can move all nodes in $\y^{(1)}[F]$ at location $a_{1}$ in $\G$ to location $a_{3}$, and the resulting vertex map $\y^{(2)}:[k]\rightarrow \{a_{2},a_{3}\}$ is still a homomorphism $F\rightarrow \G$. By a similar argument, we can construct a homomorphism $y^{(3)}:F\rightarrow \G$ such that $\y^{(3)}(1)=a_{3}$ and $y^{(3)}$ maps all nodes of $F$ onto $\{a_{3},a_{4}\}$. Also note that $\y^{(3)}$ is reachable from $y^{(1)}$ in $k$ steps. Hence we can `slide over' $y^{(1)}$ onto the nodes $\{a_{3},a_{4}\}$ in $k$ steps. Repeating this argument, this shows that there is a homomorphism $\y^{(m)}:F\rightarrow \G$  such that $\y^{(m)}$ maps $[k]$ onto $\{a_{2m-1},a_{2m}\}$ and it is reachable from $\y^{(1)}$ in $km$ steps.  
		
		To finish the proof, we first choose a walk $W_{1}=(a_{1},a_{2},\cdots,a_{2m-1})$ in the skeleton $G$ such that $a_{1} = \y^{(1)}(1)$ and $a_{2m-1}=\x'(1)$ for some integer $m\ge 1$. We can always choose such a walk using the odd cycle in $G$, say $C$, and the connectivity of $G$: first walk from $\y^{(1)}(1)$ to the odd cycle $C$, traverse it in one of the two ways, and then walk to $\x'(1)$. Moreover, it is easy to see that this gives $2m-2\le 4\diam(\G)$. Lastly, since $\x'$ is a homomorphism $F\rightarrow \G$ with $\x'(1)=a_{2m-1}$ and since $k\ge 2$, there must exist some node $a_{2m}\in [m]$ such that $A(a_{2m-1},a_{2m})>0$. Since $\G$ is bidirectional, we also have $A(a_{2m},a_{2m-1})>0$, so $a_{2m-1}$ and $a_{2m}$ are adjacent in the skeleton $G$. Hence we can let $W$ be the walk $(a_{1},a_{2},\cdots, a_{2m-1},a_{2m})$. Then $\y'$ is reachable from $\x$ in $k-2$ steps by construction, and $\x'$ is reachable from $\y^{(1)}$ in $k(m+1)+k-2\le 2k(\diam(\G)+1) + k-2$ steps by the claim. Hence $\x'$ is reachable from $\x$ in $2k\diam(\G) + 4(k-1)$ steps, as desired. 
	\end{proof}

	When $F$ is not necessarily a tree, a straightforward generalization of the argument in the Proof of \ref{prop:state_space_diam_tree} shows the following. 
	
	\begin{prop}\label{prop:state_space_diam_gen}
		Let $F$ be any simple motif and $\G$ be an irreducible and bidirectional network. Suppose there exists an integer $r\ge 1$ with  following three conditions: 
		\begin{description}
			\item[(i)] For each $\x\in \G(F,\G)$, there exists $\mathbf{y}\in \G(F,\G)$ such that $\mathbf{y}$ is reachable from $\x$ in $k$ steps and the skeleton of  $\mathbf{y}[F]$ is isomorphic to $K_{r}$. 
			\item[(ii)] $d_{G}(u,v)< r$ implies $\{u,v\}\in E_{G}$.
			\item[(iii)] $G$ contains $K_{r+1}$ as a subgraph. 
		\end{description}
		Then $\mathcal{S}(F,\G)$ is connected and 
		\begin{equation}
			\diam(\mathcal{S}(F,\G))  \le 2k \cdot \diam(\G) + 2(k-r).
		\end{equation} 
	\end{prop}
	
	\begin{proof}
		Omitted.
	\end{proof}
	
	Next, we prove Theorem \ref{thm:stationary_measure_Glauber}.

	\begin{proof}{\textbf{of Theorem \ref{thm:stationary_measure_Glauber}}.}
		
		Proposition \ref{prop:state_space_diam_tree} and an elementary Markov chain theory implies that the Glauber chain is irreducible under the assumption of (ii) and has a unique stationary distribution. Hence it remains to show (i), that $\pi:=\mathbb{P}_{F\rightarrow\G}$ is a stationary distribution of the Glauber chain. To this end, write $F=([k],A_{F})$ and let $P$ be the transition kernel of the Glauber chain. It suffices to check the detailed balance equation is satisfied by $\pi$. Namely, let $\x,\y$ be any homomorphisms $F\rightarrow \G$ such that they agree at all nodes of $F$ but for some $\ell\in [k]$. We will show that 
		\begin{align}\label{eq:Glauber_balance_eq}
			\pi(\x) P(\x,\y) =\pi(\y) P(\y,\x).
		\end{align}
		
		Decompose $F$ into two motifs $F_{\ell}=([k],A_{\ell})$ and $F_{\ell}^{c}=([k],A_{\ell}^{c})$, where $A_{\ell}(i,j) = A_{F}(i,j)\mathbf{1}(u\in \{i,j\})$ and $A_{\ell}^{c}(i,j) = A_{F}(i,j)\mathbf{1}(u\notin \{i,j\})$. Note that $A_{F}=A_{\ell}+A_{\ell}^{c}$. Then we can write  
		\begin{align}
			\pi(\x) P(\x,\y) &=	\frac{k^{-1}}{\mathtt{t}(F,\G)}  \left( \prod_{1\le i,j\le k}  A(\mathbf{x}(i),\mathbf{x}(j))^{A_{\ell}^{c}(i,j)} \right) \left(  \prod_{\substack{i\in [k] \\ i\ne \ell}} \alpha(\x(i)) \right) \\
			&\qquad \times \frac{ \prod_{j\ne \ell} [A(\x(j),\x(\ell))A(\x(j),\y(\ell))]^{A_{\ell}(j,\ell)} [A(\x(j),\x(\ell))A(\y(\ell),\x(j))]^{A_{\ell}(\ell,j)}      }{\sum_{1\le c \le n} \left( \prod_{j\ne c} A(\x(j),c)^{A_{\ell}(j,\ell)} A(c,\x(j))^{A_{\ell}(\ell,j)}    \right)  A(c,c)^{A_{\ell}(\ell,\ell)} \alpha(c)} \\
			&\qquad \qquad \times A(\x(\ell),\x(\ell))^{A_{\ell}(\ell,\ell)}A(\y(\ell),\y(\ell))^{A_{\ell}(\ell,\ell)} \alpha(\x(\ell))\alpha(\y(\ell)).
		\end{align}	
		From this and the fact that $\x$ and $\y$ agree on all nodes $j\ne \ell$ in $[k]$, we see that the value of $\pi(\x) P(\x,\y)$ is left unchanged if we exchange the roles of $\x$ and $\y$. This shows \eqref{eq:Glauber_balance_eq}, as desired. 
	\end{proof}

	To prove Theorem \ref{thm:gen_coloring_mixing_tree}, we first recall a canonical construction of coupling $(X,Y)$ between two distributions $\mu$ and $\nu$ on a finite set $\Omega$ such that $\mu(x)\land \nu(x)>0$ for some $x\in \Omega$. Let $p=\sum_{x\in \Omega} \mu(x)\land \nu(x)\in (0,1)$. Flip a coin with the probability of heads equal to $p$. If it lands heads, draw $Z$ from the distribution $p^{-1} \mu \land \nu $ and let $X=Y=Z$. Otherwise, draw independently $X$ and $Y$ from the distributions $(1-p)^{-1}(\mu - \nu)\mathbf{1}(\mu>\nu)$ and $(1-p)^{-1}(\nu - \mu)\mathbf{1}(\nu>\mu)$, respectively. It is easy to verify that $X$ and $Y$ have distributions $\mu$ and $\nu$, respectively, and that $X=Y$ if and only if the coin lands heads. This coupling is called the \textit{optimal coupling} between $\mu$ and $\nu$, since 
	\begin{equation}\label{eq:opt.coupling_TV}
		\mathbb{P}(X\ne Y) = 1-p = \lVert \mu-\nu \rVert_{\text{TV}}.
	\end{equation}
	
	The following lemma is a crucial ingredient for the proof of Theorem \ref{thm:gen_coloring_mixing_tree}. 
	
	\begin{lemma}\label{lemma:contraction_Glauber} 
		Fix a network $\G=([n],A,\alpha)$ and a simple motif  $F=([k],A_{F})$. Let $(\x_{t})_{t\ge 0}$ and $(\x'_{t})_{t\ge 0}$ be the Glauber chains of homomorphisms $F\rightarrow G$ such that $\x_{0}$ is reachable from $\x_{0}'$. Then there exists a coupling between the two chains such that 
		\begin{align}\label{eq:Glauber_mixing_contraction}
			\mathbb{E}[ d_{F,\G}(\x_{t},\x_{t}') ] \le \exp\left(-\frac{ c(\Delta,\G) t}{k}\right) d_{F,\G}(\x_{0},\x_{0}'),
		\end{align}
		where $\Delta=\Delta(F)$ denotes the maximum degree of $F$ defined at \eqref{eq:def_max_deg_edge_weighted_graph}.	
	\end{lemma}
	
	\begin{proof}
		Denote $\rho(t) = d_{F,\G}(\x_{t},\x_{t}')$ for all $t\ge 0$. Let $P$ denote the transition kernel of the Glauber chain. We first claim that if $\x_{t}\sim \x_{t'}$, then there exists a coupling between $\x_{t+1}$ and $\x_{t+1}'$ such that 
		\begin{equation}
			\mathbb{E}[ \rho(t+1) \,|\, \rho(t)=1] = 1- \frac{ c(\Delta,\G)}{k}.
		\end{equation}
		
		Suppose $\x_{t}$ and $\x'_{t}$ differ at a single coordinate, say $u\in [k]$. Denote $N_{F}(u) = \{i\in [k]\,|\, A_{F}(i,u)+A_{F}(u,i)>0\}$.  To couple $\x_{t}$ and $\x_{t+1}'$, first sample $v\in [k]$ uniformly at random. Let $\mu=\mu_{\x_{t},v}$ and $\mu'=\mu_{\x'_{t},v}$. Note that $\mu=\mu'$ if $v\notin N_{F}(u)$. If $v\in N_{F}(u)$, then since $b:=\x_{t}(v)=\x_{t}'(v)$ and $\x_{t},\x_{t}'$ are homomorphisms $F\rightarrow \G$, we have $\mu(b)\land \mu'(b)>0$. Hence the optimal coupling $(X,Y)$ between $\mu$ and $\mu'$ are well-defined. We then let $\x_{t+1}(v)=X$ and $\x_{t+1}'(v)=Y$.
		
		Note that if $v\notin N_{F}(u)\cup\{u\}$, then $X=Y$ with probability 1 and $\rho(t+1)=1$. If $v=u$, then also $X=Y$ with probability 1 and $\rho(t+1)=0$. Otherwise, $v\in N_{F}(u)$ and noting that (\ref{eq:opt.coupling_TV}), either $X=Y$ with probability $1-\lVert \mu-\mu' \rVert_{\text{TV}}$ and $\rho(t+1)=0$, or $X\ne Y$ with probability $\lVert \mu-\mu' \rVert_{\text{TV}}$. In the last case, we have $\rho(t+1)=2$ or $3$ depending on the structure of $\G$. Combining these observations, we have 
		\begin{align}
			&\mathbb{E}[ \rho(t+1)-1\,|\, \rho(t)=1 ] \\
			&\qquad \le 2\mathbb{P}(\rho(t+1)\in \{2,3\}\,|\, \rho(t)=1) - \mathbb{P}(\rho(t+1)=0\,|\, \rho(t)=1)\\
			&\qquad\le -k^{-1} \left( 1  - 2\Delta \lVert \mu - \mu' \rVert_{\text{TV}}\right).
		\end{align}     
		Further, since $\mu$ and $\mu'$ are determined locally, the expression in the bracket is at most $c(\Delta,\G)$. This shows the claim.
		
		To finish  the proof, first note that since $\x_{0}$ and $\x_{0}'$ belongs to the same component of $\mathcal{S}(F,\G)$, so do $\x_{t}$ and $\x_{t}'$ for all $t\ge 0$. We may choose a sequence $\x_{t}=\x^{(0)}_{t}, \x^{(1)}_{t}$, $\cdots, \x^{(\rho(t))}_{t}=\x_{t}'$ of homomorphisms $x^{(i)}_{t}:F\rightarrow G$ such that $\x^{(i)}_{t}\sim \x^{(i+1)}_{t}$ for all $0\le i < m$. Use the similar coupling between each pair $\x_{t}^{(i)}$ and $\x_{t}^{(i+1)}$. Then triangle inequality and the claim yields  
		\begin{align}
			\mathbb{E}[ \rho(t+1) ] \le \sum_{i=0}^{\rho(t)} \mathbb{E}[ d_{H}(\x^{(i)}_{t+1}, \x^{(i+1)}_{t+1}) ] \le \left( -\frac{c(\Delta,\G)}{k}\right) \rho(t),
		\end{align}
		where we denoted by $\x_{t+1}^{(i)}$ the homomorphism obtained after a one-step update of the Glauber chain from $\x_{t}^{(i)}$. Iterating this observation shows the assertion. 
	\end{proof}
	
	\begin{rmkk}\label{remark:c'}
		In the second paragraph in the proof of Lemma \ref{lemma:contraction_Glauber}, we always have $\rho(t+1)\in \{0,1,2\}$ if $A(x,y)>0$ for all $x\ne y\in [n]$. In this case, Lemma \ref{lemma:contraction_Glauber} holds with $c(\Delta,\G)$ replaced by $c'(\Delta,\G)$, which is defined similarly as in (\ref{eq:def_glauber_mixing_constant}) without the factor of $2$. $\hfill\blacktriangle$
	\end{rmkk}
	
	Now Theorem \ref{thm:gen_coloring_mixing_tree} follows immediately. 
	
	\begin{proof}{\textbf{of Theorem \ref{thm:gen_coloring_mixing_tree}}.}
		Let $(\x_{t})_{t\ge 0}$ and $(\x_{t}')_{t\ge 0}$ be Glauber chains of homomorphisms $F\rightarrow \G$. Let $P$ be the transition kernel of the Glauber chain. By Proposition \ref{prop:state_space_diam_tree}, $\x_{0}$ is reachable from $\x_{0}'$ and $d_{F,\G}(\x_{0},\x_{0}') \le 2k(\diam(\G)+1)$. Using the coupling between $\x_{t}$ and $\x_{t}'$ as in Lemma \ref{lemma:contraction_Glauber} and Markov's inequality give
		\begin{equation}
			\mathbb{P}(\x_{t}\ne \x'_{t}) = \mathbb{P}(d_{F,\G}(\x_{t},\x_{t}')\ge 1) \le \mathbb{E}(d_{F,\G}(\x_{t},\x_{t}')) \le 2k\exp\left(-\frac{ c(\Delta,\G) t}{k}\right) (\diam(\G)+1).
		\end{equation}
		Minimizing the left hand side overall coupling between $P^{t}(\x_{0},\cdot)$ and $P^{t}(\x_{0}',\cdot)$ gives 
		\begin{align}
			\lVert P^{t}(\x_{0},\cdot) - P^{t}(\x'_{0},\cdot)  \rVert_{TV}  \le 2k\exp\left(-\frac{c(\Delta,\G) t}{k}\right) (\diam(\G)+1).
		\end{align}	
		Then the assertion follows. 
	\end{proof}
	
	\begin{rmkk}
		Suppose that $\G$ is the complete graph $K_{q}$ with $q$ nodes and uniform distribution on its nodes. Then a homomorphism $F\rightarrow K_{q}$ is a $q$-coloring of $F$ and it is well-known that the Glauber chain of $q$-colorings of $F$ mixes rapidly with mixing time 
		\begin{equation}\label{eq:coloring_mixing}
			t_{mix}(\eps) \le \left\lceil \left( \frac{q-2\Delta}{q-\Delta}  \right) k\log(\eps/k) \right\rceil,
		\end{equation}
		provided $q>2\Delta$ (e.g., \cite[Thm. 14.8]{levin2017markov}). This can be obtained as a special case of Lemma \ref{lemma:contraction_Glauber}. Indeed, note that $\mathcal{S}(F,K_{q})$ is connected and has a diameter at most $k$. Hence according to Lemma \ref{lemma:contraction_Glauber} and Remark \ref{remark:c'}, it is enough to show that 
		\begin{equation}\label{eq:coloring_mixing2}
			c'(\Delta,K_{q}) \ge  \frac{q-2\Delta}{q-\Delta},
		\end{equation}
		where the quantity on the left-hand side is defined in Remark \ref{remark:c'}. To see this, note that when $\G$ is a simple graph with uniform distribution on its nodes, 
		\begin{align}\label{eq:coloring_mixing3}
			1-\lVert \mu_{\x,v} - \mu_{\x',v} \rVert_{\text{TV}} = \sum_{z\in [n]} \mu_{\x,v}(z) \land \mu_{\x',v}(z) = \frac{|\text{supp}(\mu_{\x,v})\cap \text{supp}(\mu_{\x',v})|   }{|\text{supp}(\mu_{\x,v})| \lor |\text{supp}(\mu_{\x',v}) |}.
		\end{align}
		When we take $\G=K_{q}$, it is not hard to see that the last expression in (\ref{eq:coloring_mixing3}) is at most $1- 1/(q-\Delta)$. Hence we have (\ref{eq:coloring_mixing2}), as desired. $\hfill\blacktriangle$
	\end{rmkk}

	\vspace{0.5cm}
	\section{Proof of Stability inequalities}
	\label{section:proofs_stability}
	
	In this section, we provide proofs of the stability inequalities stated in Subsection \ref{subsection:stability_inequalities}, namely, Propositions \ref{prop:conditioned_counting_lemma}, \ref{prop:stability_Ftransform}, and \ref{thm:counting_filtration}.
	
	\vspace{0.2cm}
	\begin{proof}{\textbf{of Proposition \ref{prop:conditioned_counting_lemma}}.}
		First, write 
		\begin{align}
			|\mathtt{t}(H,U\,|\, F) - \mathtt{t}(H,W\,|\, F)| \le \frac{ \mathtt{t}(H,U) |\mathtt{t}(F,W)-\mathtt{t}(F,U)| + \mathtt{t}(F,U)|\mathtt{t}(H,U)-\mathtt{t}(H,W)| }{ \mathtt{t}(F,U)\mathtt{t}(F,W)}. 
		\end{align}
		Since $F$ is a subgraph of $H$, we have $\mathtt{t}(H,U)\le \mathtt{t}(F,U)$ and $|E_{H}|\le |E_{H}|$. Hence the assertion follows by (\ref{ineq:countinglemma}). 	
	\end{proof}

	In order to prove Proposition \ref{prop:stability_Ftransform}, note that the norm of a kernel $W:[0,1]^{2}\rightarrow [0,\infty)$ can also defined by the formula 
	\begin{equation}
		\lVert W \rVert_{\square} =  \sup_{0\le f,g \le 1} \left| \int_{0}^{1}\int_{0}^{1} W(x,y)f(x)g(y) \,dx\,dy \right|,
	\end{equation}
	where $f,g:[0,1]\rightarrow [0,1]$ are measurable functions.
	
	\vspace{0.2cm}
	\begin{proof}{\textbf{of Proposition \ref{prop:stability_Ftransform}}.}
		Let $F=([k],A_{F})$ be a simple motif and $U,W$ denote graphons. Write $\bar{U}=\mathtt{t}(F,U) U^{F}$ and $\bar{W} = \mathtt{t}(F,W) W^{F}$. We first claim that 
		\begin{equation}
			\lVert \bar{U}-\bar{W} \rVert_{\square} \le  \lVert A_{F}\rVert_{1}\cdot \lVert U-W \rVert_{\square},
		\end{equation} 
		from which the assertion follows easily. Indeed, 
		\begin{eqnarray*}
			\lVert U^{F}-W^{F} \rVert_{\square} &=& \frac{1}{ \mathtt{t}(F,U)\mathtt{t}(F,W)} \lVert  \mathtt{t}(F,W) \bar{U} - \mathtt{t}(F,W) \bar{W}  \rVert_{\square} \\
			&\le & \frac{1}{ \mathtt{t}(F,U)\mathtt{t}(F,W)} \left( \mathtt{t}(F,W) \cdot \lVert  \bar{U} - \bar{W} \rVert_{\square} + |\mathtt{t}(F,U)-\mathtt{t}(F,W)| \cdot \lVert \bar{W}   \rVert_{\square}   \right),
		\end{eqnarray*}
		and we have $\lVert \bar{W} \rVert_{\square}/ \mathtt{t}(F,U) = \lVert W^{F} \rVert_{\square} = \lVert W^{F} \rVert_{1} = 1$.
		Then the assertion follows from (\ref{ineq:conditioned_counting_lemma}) and a similar inequality after changing the role of $U$ and $W$.
		
		To show the claim, let $f,g:[0,1]\rightarrow [0,1]$ be two measurable functions. It suffices to show that 
		\begin{equation}
			\left| \int_{0}^{1}\int_{0}^{1} f(x_{1})g(x_{n}) (\bar{U}(x_{1},x_{n})-\bar{W}(x_{1},x_{n}))\,dx_{1} dx_{n}  \right| \le \lVert A_{F}\rVert_{1}\cdot \lVert U-W \rVert_{\square}.
		\end{equation}
		Indeed, the double integral on the left-hand side can be written as 
		\begin{equation}
			\int_{[0,1]^{n}} f(x_{1})g(x_{n})\left(\prod_{1\le i,j\le k}   U(z_{i},w_{j})^{A_{F}(i,j)}-\prod_{1\le i,j\le k} W(z_{i},w_{j})^{A_{F}(i,j)}\right) \, dx_{1}\cdots dx_{n}.
		\end{equation} 
		We say a pair $(i,j)\in [k]^{2}$ a `directed edge' of $F$ if $A_{F}(i,j)=1$. Order all directed edges of $F$ as $E=\{ e_{1},e_{2},\cdots,e_{m} \}$, and denote $e_{r}=(i_{r},j_{r})$. Since $F$ is a simple motif, there is at most one directed edge between each pair of nodes. Hence we can write the term in the parenthesis as the following telescoping sum 
		\begin{eqnarray*}
			&&\sum_{r=1}^{m} U(e_{1})\cdots U(e_{r-1})(U(e_{r})-W(e_{r})) W(e_{r+1})\cdots W(e_{m}) \\
			&& \qquad \qquad = \sum_{r=1}^{m} \alpha(z_{i_{r}}) \beta(w_{j_{r}}) (U(z_{i_{r}}, w_{j_{r}})-W(z_{i_{r}}, w_{j_{r}})),
		\end{eqnarray*}
		where $\alpha(z_{i_{r}})$ is the product of all $U(e_{k})$'s and $W(e_{k})$'s such that $e_{k}$ uses the node $i_{r}$ and $\beta(w_{j_{r}})$ is defined similarly. Now for each $1\le r \le m$, we have  
		\begin{equation}
			\left| \int_{[0,1]^{n}}  f(x_{1})g(x_{n})\alpha(z_{i_{r}}) \beta(w_{j_{r}}) (U(z_{i_{r}}, w_{j_{r}})-W(z_{i_{r}}, w_{j_{r}})) \, dx_{1}\cdots dx_{n} \right| \le \lVert U-W \rVert_{\square}.
		\end{equation}
		The claim then follows. 	
	\end{proof}

	Lastly, we prove Theorem \ref{thm:counting_filtration}. It will be convenient to introduce the following notion of distance between filtrations of kernels.
	\begin{equation}
		d_{\blacksquare}( U,W) =\int_{0}^{\infty} \lVert \mathbf{1}(U\ge t) -\mathbf{1}(W\ge t) \rVert_{\square}\,dt 
	\end{equation}
	For its `unlabeled' version, we define 
	\begin{equation}
		\delta_{\blacksquare}(U,W) = \inf_{\varphi} \,d_{\blacksquare}(U,W^{\varphi})
	\end{equation}
	where the infimum ranges over all measure-preserving maps $\varphi:[0,1]\rightarrow [0,1]$.
	
	An interesting observation is that this new notion of distance between kernels interpolates the distances induced by the cut norm and the 1-norm. For a given graphon $U:[0,1]^{2}\rightarrow [0,1]$ and $t\ge 0$, we denote by $U_{\ge t}$ the 0-1 graphon defined by 
	\begin{align}
		U_{\ge t}(x,y) = \mathbf{1}(U(x,y)\ge t).
	\end{align}
	
	\begin{prop}\label{prop:filtration_cutnormbound}
		For any two graphons $U,W:[0,1]^{2}\rightarrow [0,1]$, we have 
		\begin{equation}
			\delta_{\square}(U,W) \le \delta_{\blacksquare}(U,W) \le \delta_{1}(U,W).
		\end{equation}
	\end{prop}
	
	\begin{proof}
		It suffices to show the following `labeled' version of the assertion:
		\begin{equation}
			d_{\square}(U,W) \le d_{\blacksquare}(U,W)  \le d_{1}(U,W).
		\end{equation}
		To show the first inequality, note that for any fixed $(x,y)\in [0,1]^{2}$, 
		\begin{equation}
			\int_{0}^{1} \mathbf{1}(U(x,y)\ge t) - \mathbf{1}(W(x,y)\ge t)\,dt = W(x,y)-U(x,y).
		\end{equation}
		Hence the first inequality follows easily from the definition and Fubini's theorem: 
		\begin{eqnarray}
			\lVert U-W \rVert_{\square} &=& \sup_{S\times T\subseteq [0,1]^{2}} \left| \int_{S}\int_{T} U(x,y)-W(x,y) \,dx\,dy \right|\\
			&=& \sup_{S\times T\subseteq [0,1]^{2}} \left| \int_{S}\int_{T}\int_{0}^{1} \mathbf{1}(U>t) -\mathbf{1}(W>t) \,dt\,dx\,dy \right|\\
			&=& \sup_{S\times T\subseteq [0,1]^{2}} \left| \int_{0}^{1}\int_{S}\int_{T} \mathbf{1}(U>t) -\mathbf{1}(W>t) \,dx\,dy\,dt \right|\\
			&\le&  \int_{0}^{1}\sup_{S\times T\subseteq [0,1]^{2}} \left|\int_{S}\int_{T} \mathbf{1}(U>t) -\mathbf{1}(W>t) \,dx\,dy \right|\,dt.
		\end{eqnarray}
		
		For the second inequality, by a standard approximation argument, it is enough to show the assertion for the special case when both $U$ and $W$ are simple functions. Hence we may assume that there exists a partition $[0,1]^{2}=R_{1}\sqcup \cdots \sqcup R_{n}$ into measurable subsets such that both kernels are constant on each $R_{j}$. Define kernels $U^{0},\cdots,U^{n}$ by  
		\begin{equation*}
			U^{j}(x,y) = U(x,y) \one\{(x,y)\in R_{1}\cup\cdots\cup R_{j}  \} + W(x,y) \one\{(x,y)\in R_{j+1}\cup\cdots\cup R_{n}  \}. 
		\end{equation*}
		In words, $U^{j}$ uses values from $U$ on the first $j$ $R_{i}$'s, but agrees with $W$ on the rest. Denote by $u_{j}$ and $w_{j}$ the values of $U$ and $W$ on the $R_{j}$, respectively. Observe that 
		\begin{equation*}
			\left|\one\{U^{j}(x,y)> t\}- \one\{ U^{j-1}(x,y)>t \}\right| = \begin{cases}
				1 & \text{if $t\in [u_{j}\land w_{j},  u_{j}\lor w_{j}]$ and $(x,y)\in R_{j}$ } \\ 
				0 & \text{otherwise}.
			\end{cases}
		\end{equation*}
		This yields that, for any $p\in [0,\infty)$
		\begin{equation*}
			\lVert U^{j}_{\ge t}-U^{j-1}_{\ge t} \rVert_{\square} = \mu(R_{j}) \one \{t\in [u_{j}\land w_{j},  u_{j}\lor w_{j}] \}.
		\end{equation*}
		Now triangle inequality for the cut norm gives  
		\begin{eqnarray*}
			\int_{0}^{1} \lVert U_{\ge t}-W_{\ge t} \rVert_{\square}\,dt &\le& \sum_{j=1}^{n} |u_{j}-w_{j}|\mu(R_{j}) \\
			&=& \int_{0}^{1}\int_{0}^{1} \sum_{j=1}^{k^{2}}\left| U^{j}(x,y)-U^{j-1}(x,y) \right|\,dx\,dy\\
			&=& \int_{0}^{1}\int_{0}^{1}|U(x,y)-W(x,y)|\,dx\,dy \\
			&=& \lVert U-W \rVert_{1}.
		\end{eqnarray*}
		This shows the assertion. 
	\end{proof}

	We need one more preparation to prove Theorem \ref{thm:counting_filtration}. Let $F=([k],A_{F})$ and $H=([k],A_{H})$ be motifs and  $U:[0,1]^{2}\rightarrow [0,1]$ be a graphon. For each $t\ge 0$, denote 
	\begin{align}
		\mathtt{t}(H,U_{\ge t}\,;\, F) = \int_{[0,1]^{k}} \prod_{1\le i,j\le k} \mathbf{1}(U(x_{i},x_{j})^{A_{H}(i,j)}\ge t ) \prod_{1\le i,j\le k} U(x_{i},x_{j})^{A_{F}(i,j)} \, dx_{1}\cdots dx_{k}.
	\end{align}
	Then it is easy to see that 
	\begin{align}
		\mathtt{f}(H,U\,|\, F)(t) = \frac{1}{\mathtt{t}(F,U)} \mathtt{t}(H,U_{\ge t}\,;\, F).
	\end{align}
	
	\begin{prop}\label{prop:counting_interpolation}
		Let  $H=([k],A_{H})$ and $F=([k],A_{F})$ be simple motifs such that $H+F=([k],A_{F}+A_{H})$ is simple. Fix graphons $U,W:[0,1]^{2}\rightarrow [0,1]$. Then 
		\begin{align}\label{eq:counting_interpolation}
			|\mathtt{t}(H,U_{\ge t};F) - \mathtt{t}(H,W_{\ge t};F)| \le   \lVert A_{F} \rVert_{1}\cdot \delta_{\square}(U,W) + \lVert A_{H} \rVert_{1}\cdot \delta_{\square}(U_{\ge t}, W_{\ge t}).
		\end{align}
	\end{prop}
	
	\begin{proof}
		Denote $E_{F}=\{(i,j)\in [k]^{2}\,|\, A_{F}(i,j)>0 \}$ and  $E_{H}=\{(i,j)\in [k]^{2}\,|\, A_{H}(i,j)>0 \}$. Then by the hypothesis, $E_{F}$ and $E_{H}$ are disjoint and $E:=E_{F}\cup E_{F} = \{(i,j)\in [k]^{2}\,|\, A_{F}(i,j)+A_{H}(i,j)>0 \}$. Write $E=\{e_{1},e_{2},\cdots,e_{m}\}$, where $m=|E|$.

		Fix a vertex map $[k]\mapsto [0,1]$, $i\mapsto x_{i}$. For each $1\le \ell \le m$, define $a_{\ell}$ and $b_{\ell}$ by 
		\begin{align}
			a_{\ell} = 
			\begin{cases}
				U(e_{\ell}) & \text{if $e_{\ell}\in E_{F}$} \\
				\mathbf{1}(U(e_{\ell}\ge t)) & \text{if $e_{\ell}\in E_{H}$},
			\end{cases} 
			\qquad 
			b_{\ell} = 
			\begin{cases}
				W(e_{\ell}) & \text{if $e_{\ell}\in E_{F}$} \\
				\mathbf{1}(W(e_{\ell}\ge t)) & \text{if $e_{\ell}\in E_{H}$}.
			\end{cases} 
		\end{align}
		Then we have 
		\begin{align}
			&\prod_{(i,j)\in E_{H}} \mathbf{1}(U(x_{i},x_{j})\ge t)  \prod_{(i,j)\in E_{F}} U(x_{i},x_{j}) = \prod_{\ell=1}^{m} a_{\ell}, \\
			&\prod_{(i,j)\in E_{H}\setminus E_{F}} \mathbf{1}(W(x_{i},x_{j})\ge t)  \prod_{(i,j)\in E_{F}} W(x_{i},x_{j}) = \prod_{\ell=1}^{m} b_{\ell}.
		\end{align}
		Also,  we can write the difference between the integrands as the following telescoping sum 
		\begin{align}
			\prod_{\ell=1}^{m} a_{\ell} - \prod_{\ell=1}^{m} b_{\ell} &=  \sum_{\ell=1}^{m} (a_{1}\cdots a_{\ell}b_{\ell+1}\cdots b_{m} - a_{1}\cdots a_{\ell-1}b_{\ell}\cdots b_{m}) \\
			& = \sum_{\ell=1}^{m} c_{\ell} (a_{\ell}-b_{\ell}),
		\end{align}
		where each $c_{\ell}$ is a suitable product of $a_{i}$'s and $b_{i}$'s. Note that since $U$ and $W$ are graphons, each $c_{\ell}\in [0,1]$.  Moreover, say the $m$th summand corresponds to an edge $(i,j)\in E_{H}$. By the assumption on simple motifs, none of $a_{\ell}$ and $b_{\ell}$ depend on both coordinates $x_{i}$ and $x_{j}$ except $a_{\ell}$ and $b_{\ell}$. Hence $c_{\ell}$ can be written as the product $f_{\ell}(x_{i})g_{\ell}(x_{j})$ of two functions. Furthermore, 
		\begin{align}
			a_{\ell}-b_{\ell} = 
			\begin{cases}
				U(x_{i},x_{j})-W(x_{i},x_{j}) & \text{if $(i,j)\in E_{F}$} \\
				\mathbf{1}(U(x_{i},x_{j})\ge t)-\mathbf{1}(W(x_{i},x_{j})\ge t) &  \text{if $(i,j)\in E_{H}$}. 
			\end{cases}
		\end{align} 
		Hence if $(i,j)\in E_{F}$, we get 
		\begin{align}
			&\left| \int_{[0,1]^{k}} c_{\ell} (a_{\ell}-b_{\ell}) \,dx_{1}\cdots dx_{k}\right| \\
			& \qquad = \left| \int_{[0,1]^{k-2}} \left( \int_{[0,1]^{2}} f_{\ell}(x_{i})g_{\ell}(x_{j}) U(x_{i},x_{j})-W(x_{i},x_{j}) \, dx_{i} dx_{j} \right) \prod_{\ell\ne i,j} dx_{\ell}\right| \\
			& \qquad \le \lVert U-W \rVert_{\square}.
		\end{align}
		Similarly, for $(i,j)\in E_{H}$, we have 
		\begin{align}
			\left| \int_{[0,1]^{k}} c_{\ell} (a_{\ell}-b_{\ell}) \,dx_{1}\cdots dx_{k}\right|  \le \lVert U_{\ge t}-W_{\ge t} \rVert_{\square}.
		\end{align}
		Therefore the assertion follows from a triangle inequality and optimizing the bound over all measure-preserving maps, as well as noting that $|E_{F}|=\lVert A_{F} \rVert_{1}$ and $|E_{H}|=\lVert A_{H} \rVert_{1}$.
	\end{proof}

	Now we prove Theorem \ref{thm:counting_filtration}. 
	
	\begin{proof}{\textbf{of Theorem \ref{thm:counting_filtration}}.}
		Let $F=([k],A_{F})$ and $H=([k],A_{H})$ be simple motifs such that $H+F:=([k],A_{F}+A_{H})$ is simple. First, use a triangle inequality to write  
		\begin{align}
			&|\mathtt{f}(H,U\,|\, F)(t) - \mathtt{f}(H,W\,|\, F)(t)| \\
			&\qquad \le \frac{  \mathtt{t}(F,U) |\mathtt{t}(H,W_{t\ge t}\,;\, F)-\mathtt{t}(H,U_{t\ge t}\,;\,F )| + \mathtt{t}(H,U_{\ge t}\,;\, F)|\mathtt{t}(F,U)-\mathtt{t}(F,W)|  }{ \mathtt{t}(F,U)\mathtt{t}(F,W)}.
		\end{align}
		Note that $\mathtt{t}(F+H,U)\le \mathtt{t}(F,U)$ and for each $t\in [0,1]$ we have $\mathtt{t}(H,U_{\ge t}\,;\, F)\in [0,1]$ by definition. Hence by using Proposition \ref{prop:counting_interpolation}, we get 
		\begin{align}
			& |\mathtt{f}(H,U\,|\, F)(t) - \mathtt{f}(H,W\,|\, F)(t)| \\
			&\qquad \le \frac{\lVert A_{F} \rVert_{1}\cdot \delta_{\square}(U,W) + \lVert A_{H} \rVert_{1} \cdot \delta_{\square}(U_{\ge t}, W_{\ge t}) }{\mathtt{t}(F,W)} + \frac{\lVert A_{F} \rVert_{1}\cdot \delta_{\square}(U,W)}{\mathtt{t}(F,U)}.
		\end{align}
		Integrating this inequality over $t\in [0,1]$ and using Proposition \ref{prop:filtration_cutnormbound} then give 
		\begin{align}
			& \lVert \mathtt{f}(H,U\,|\, F)(t) - \mathtt{f}(H,W\,|\, F)(t)\rVert_{1} \\
			&\qquad \le |E_{F}|\cdot \delta_{\square}(U,W)\left(  \frac{1}{\mathtt{t}(F,W)} + \frac{1}{\mathtt{t}(F,U)} \right)+ \frac{|E_{H}\setminus E_{F}|\cdot \delta_{1}(U,W)}{\mathtt{t}(F,W)}.
		\end{align}
		We can obtain a similar inequality after we change the roles of $U$ and $W$. Then the assertion follows optimizing between the two upper bounds.  
	\end{proof}

	\section{Network data sets}\label{section:Network_datasets}
	
	In Sections \ref{section:sampling_hamiltonian} and \ref{section:FB}, we examined the following real-world  and synthetic networks:
	\begin{enumerate}[itemsep=0.1cm]
		\item \dataset{Caltech}: This connected network, which is part of the {\sc Facebook100} data set~\cite{traud2012social} (and which was studied previously as part of the {\sc Facebook5} data set~\cite{Traud2011}), has 762 nodes and 16,651 edges. The nodes represent users in the Facebook network of Caltech on one day in {fall 2005}, and the edges encode Facebook `friendships' between these accounts. 
		
		\item \dataset{Simmons}: This connected network, which is part of the {\sc Facebook100} data set~\cite{traud2012social} (and which was studied previously as part of the {\sc Facebook5} data set~\cite{Traud2011}), has 1,518 nodes and 65,976 edges. The nodes represent users in the Facebook network of Simmons on one day in {fall 2005}, and the edges encode Facebook `friendships' between these accounts. 
		
		\item \dataset{Reed}: This connected network, which is part of the {\sc Facebook100} data set~\cite{traud2012social} (and which was studied previously as part of the {\sc Facebook5} data set~\cite{Traud2011}), has 962 nodes and 37,624 edges. The nodes represent users in the Facebook network of Reed on one day in {fall 2005}, and the edges encode Facebook `friendships' between these accounts. 
		
		\item \dataset{NYU}: This connected network, which is part of the {\sc Facebook100} data set~\cite{traud2012social} (and which was studied previously as part of the {\sc Facebook5} data set~\cite{Traud2011}), has 21,679 nodes and 1,431,430 edges. The nodes represent users in the Facebook network of NYU on one day in {fall 2005}, and the edges encode Facebook `friendships' between these accounts. 	
		
		\item \dataset{Virginia}: This connected network, which is part of the {\sc Facebook100} data set~\cite{traud2012social}, has 21,325 nodes and 1,396,356 edges. The nodes represent user accounts in the Facebook network of Virginia on one day in {fall 2005}, and the edges encode Facebook `friendships' between these accounts.

		\item \dataset{UCLA}: This connected network, which is part of the {\sc Facebook100} data set~\cite{traud2012social}, has 20,453 nodes and 747,604 edges. The nodes represent user accounts in the Facebook network of UCLA on one day in {fall 2005}, and the edges encode Facebook `friendships' between these accounts. 
		
		\item \dataset{Wisconsin}: This connected network, which is part of the {\sc Facebook100} data set~\cite{traud2012social}, has 23,842 nodes and 835,952 edges. The nodes represent user accounts in the Facebook network of Wisconsin on one day in {fall 2005}, and the edges encode Facebook `friendships' between these accounts.
	\end{enumerate}

	\begin{enumerate}[resume, itemsep=0.1cm]
		\item \dataset{ER}: An Erd\H{o}s--R\'{e}nyi (ER) network~\cite{erdds1959random,newman2018}, which we denote by $\textup{ER}(n,p)$, is a random-graph model. The parameter $n$ is the number of nodes and the parameter $p$ is the independent, homogeneous probability that each pair of distinct nodes has an edge between them. The network \dataset{ER} is an individual graph that we draw from {$\textup{ER}({5000},0.01)$}.

		\item \dataset{WS}: A Watts--Strogatz (WS) network, which we denote by $\textup{WS}(n,k,p)$, is a random-graph model to study the small-world phenomenon~\cite{watts1998collective,newman2018}. In the version of WS networks that we use, we start with an $n$-node ring network in which each node is adjacent to its $k$ nearest neighbors. With independent probability $p$, we then remove and rewire each edge so that it connects a pair of distinct nodes that we choose uniformly at random. The network \dataset{WS} is an individual graph that we draw from {$\textup{WS}({5000},50, 0.10)$}.
		
		\item \dataset{BA}: A Barab\'{a}si--Albert (BA) network, which we denote by $\textup{BA}(n,{n_{0}})$, is a random-graph model with a linear preferential-attachment mechanism~\cite{barabasi1999emergence,newman2018}. In the version of BA networks that we use, we start with ${n_{0}}$ isolated nodes and we introduce new nodes with ${n_{0}}$ new edges each that attach preferentially (with a probability that is proportional to node degree) to existing nodes until we obtain a network with 
		$n$ nodes. 
		The network \dataset{BA} is an individual graph that we draw from {$\textup{BA}({5000},50)$}. 
		
		\item {\dataset{SBM}: We use stochastic-block-model (SBM) networks in which each block is an ER network~\cite{holland1983stochastic}.
			Fix disjoint finite sets $C_{1}\cup \cdots \cup C_{{k_{0}}}$ and a {$k_{0} \times k_{0}$} matrix $B$ whose entries are real numbers between $0$ and $1$. An SBM 
			network, which we denote by 
			{$\textup{SBM}(C_{1},\ldots, C_{{k_{0}}}, B)$}, {has the}
			node set $V = C_{1}\cup \cdots \cup C_{{k_{0}}}${. For} each node pair {$(x,y)$}, there is an edge between {$x$ and $y$} {with independent probabilities} {$B[i_{0},j_{0}]$, {with indices} $i_{0},j_{0}\in \{1,\ldots,k_{0}\}$} {such} that {$x\in C_{i_{0}}$ and $y\in C_{j_{0}}$. If $k_{0}=1$} and $B$ has a constant $p$ in all entries, this SBM specializes to the Erd\H{o}s--R\'{e}nyi (ER) {random-graph model} $\textup{ER}(n,p)$ with $n=|C_{1}|$. 
			The networks \dataset{SBM} is an individual graphs that we draw from {$\textup{SBM}(C_{1},\ldots,C_{k_{0}}, B)$} with $|C_{1}|=|C_{2}|=|C_{3}|= \text{1,000}$, where $B$ is the $3 \times 3$ matrix whose diagonal entries are $0.5$ and whose off-diagonal entries are $0.001$. It has 3,000 nodes and 752,450 edges.}
	\end{enumerate}

	\section{Additional figures and tables}
	
	\begin{figure*}[h]
		\centering
		\includegraphics[width=0.9 \linewidth]{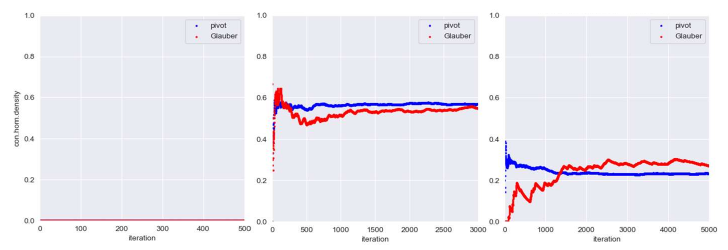}
		\caption{ Computing $\mathtt{t}(H_{k,0},\G_{n}\,|\, F_{k,0})$ by time averages of Glauber (red) and Pivot (blue) chains $F_{k,0}\rightarrow \G_{50}$ for $k=0$ (left),   $k=3$ (middle), and $k=9$  (right).  
		}
		\label{fig:torus_CHD}
	\end{figure*}
	
	\begin{figure*}[h]
		\centering
		\includegraphics[width=0.9 \linewidth]{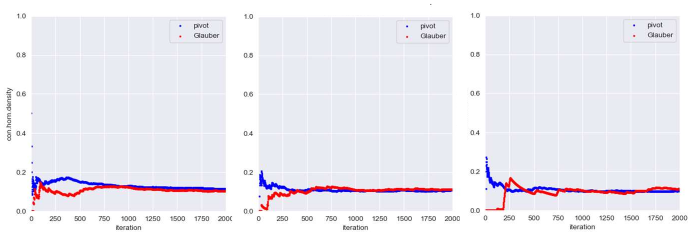}
		\caption{ Computing $\mathtt{t}(H_{k,0},\G_{n}\,|\, F_{k,0})$ by time averages of Glauber (red) and Pivot (blue) chains $F_{k,0}\rightarrow \G_{50}^{0.1,0}$ for  $k=2$ (left),   $k=3$ (middle), and  (right) $k=9$ (right).  
		}
		\label{fig:torus_long_edges_CHD}
	\end{figure*}

	\begin{figure*}[h]
		\centering
		\includegraphics[width=0.95 \linewidth]{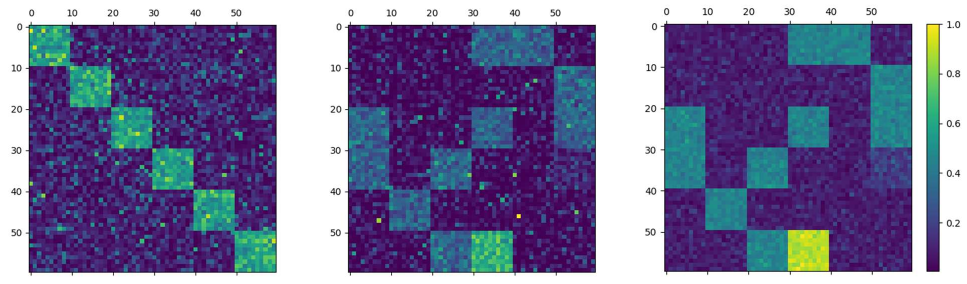}
		\caption{ Plots of random block matrices $B_{1}$ (left), $B_{2}$ (middle), and $B_{3}$ (right). Colors from dark blue to yellow denote values of entries from 0 to 1, as shown in the color bar on the right. 
		}
		\label{fig:block_networks_pic}
	\end{figure*} 
	
	\begin{figure*}[h]
		\centering
		\includegraphics[width=1 \linewidth]{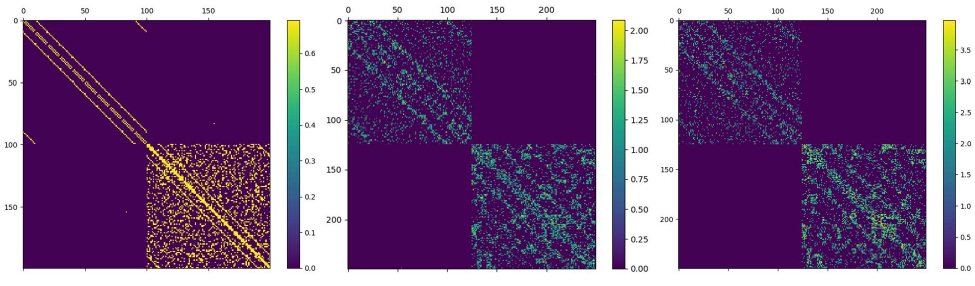}
		\caption{ Plot of log transforms of the edge weight matrices $A_{1}$ (left), $A_{2}$ (middle), and $A_{3}=A_{2}^{C_{3}}$ (right). Corresponding color bars are shown to the right of each plot.
		}
		\label{fig:barbell_pic}
	\end{figure*}

	\begin{figure*}[h]
		\centering
		\includegraphics[width=0.9 \linewidth]{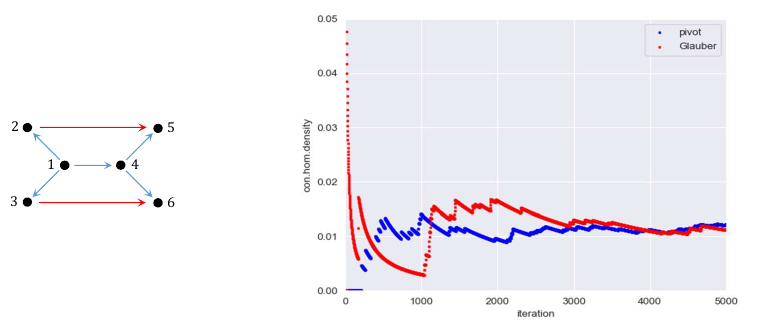}
		\caption{ Computing $\mathtt{t}(H,\G_{n}\,|\, F)$ via time averages of Glauber/Pivot chains $F\rightarrow \G_{50}^{0.1,0}$. The underlying rooted tree motif $F=([6], \mathbf{1}_{\{(1,2),(1,3),(1,4),(4,5),(4,6)\} })$ is depicted on the left, and $H=([6],A_{H})$ is obtained from $F$ by adding directed edges (red) $(2,5)$ and $(3,6)$.
		}
		\label{fig:torus_long_edges_motif_CHD}
	\end{figure*}
	
	\begin{table*}[h]
		\centering
		\includegraphics[width=0.85 \linewidth]{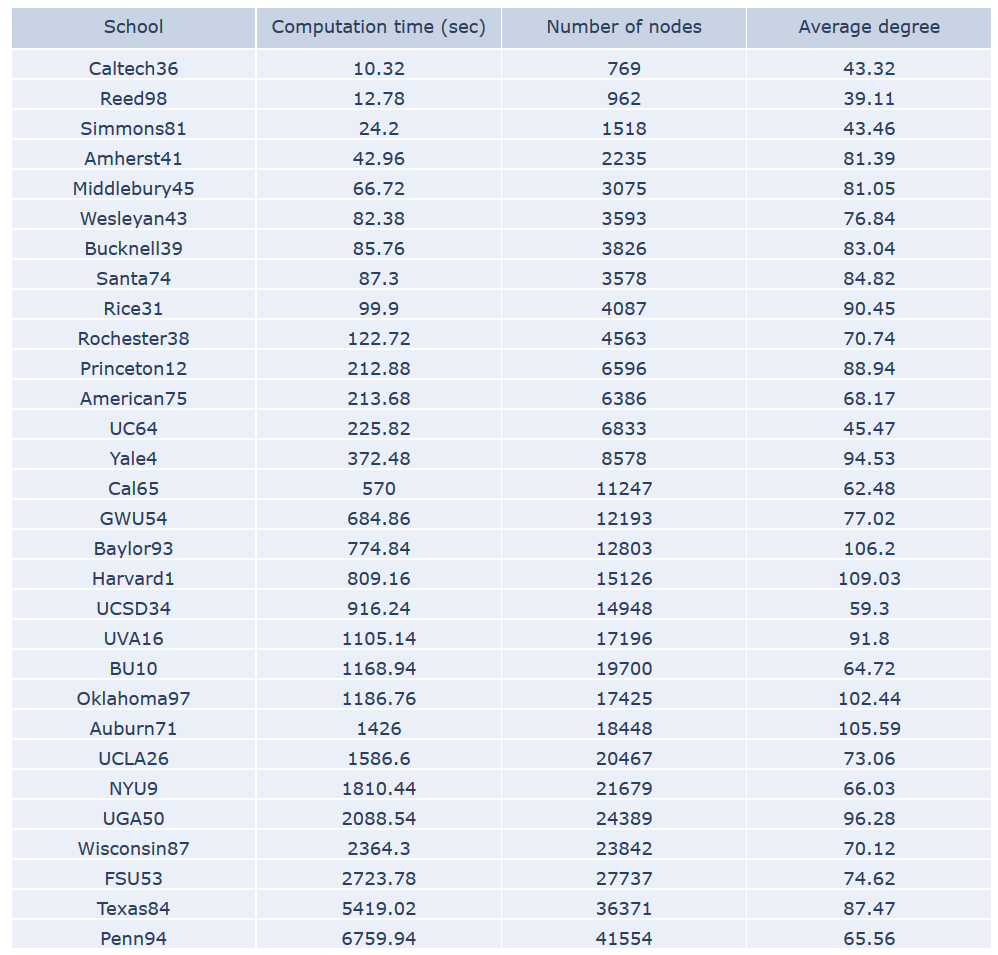}
		\vspace{0.2cm}
		\caption{Computation times for computing MACCs of the
			Facebook100 dataset shown in Figure \ref{fig:FB_MACC} and
			number of nodes and the average degree of the corresponding
			networks. Results are shown for 30 networks randomly chosen
			amongst those in the Facebook100 dataset. 
		}
		\label{fig:computation_time_table}
	\end{table*}

	\begin{table*}[h]
		\centering
		\includegraphics[width=1 \linewidth]{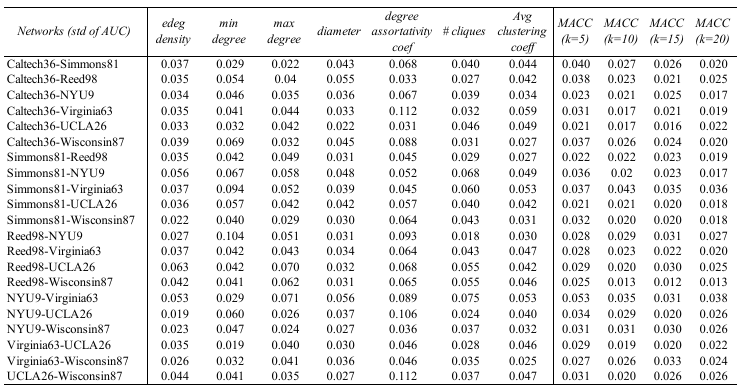}
		\caption{  \commHL{The standard deviations of AUC scores over ten independent trials of the subgraph classification tasks. See Table \ref{table:subgraph_classification1} for more details.}  
		}
		\label{table:subgraph_classification_std}
	\end{table*}

	\begin{figure*}[h]
		\centering
		\includegraphics[width=1 \linewidth]{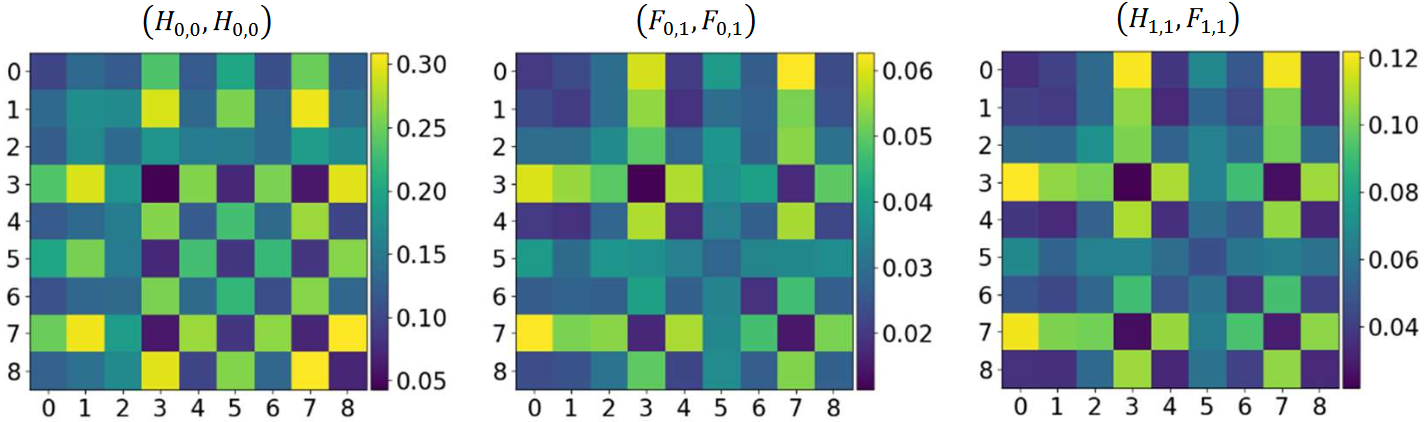}
		\caption{ Heat maps of the average $L^{1}$-distance matrices between the reference (rows) and validation (columns) CHD profiles of the nine authors for the pair of motifs $(H_{00}, F_{00})$ (left), $(H_{01},F_{01})$ (middle), and $(H_{11},F_{11})$ (right). 
		}
		\label{fig:WAN_dist_mx}
	\end{figure*}

	\begin{figure*}[h]
		\centering
		\refstepcounter{figure}
		\begin{tabular}{c }
			\begin{minipage}{0.8\linewidth}	
				\begin{subfigure}[b]{1\textwidth}            
					\includegraphics[width=\textwidth]{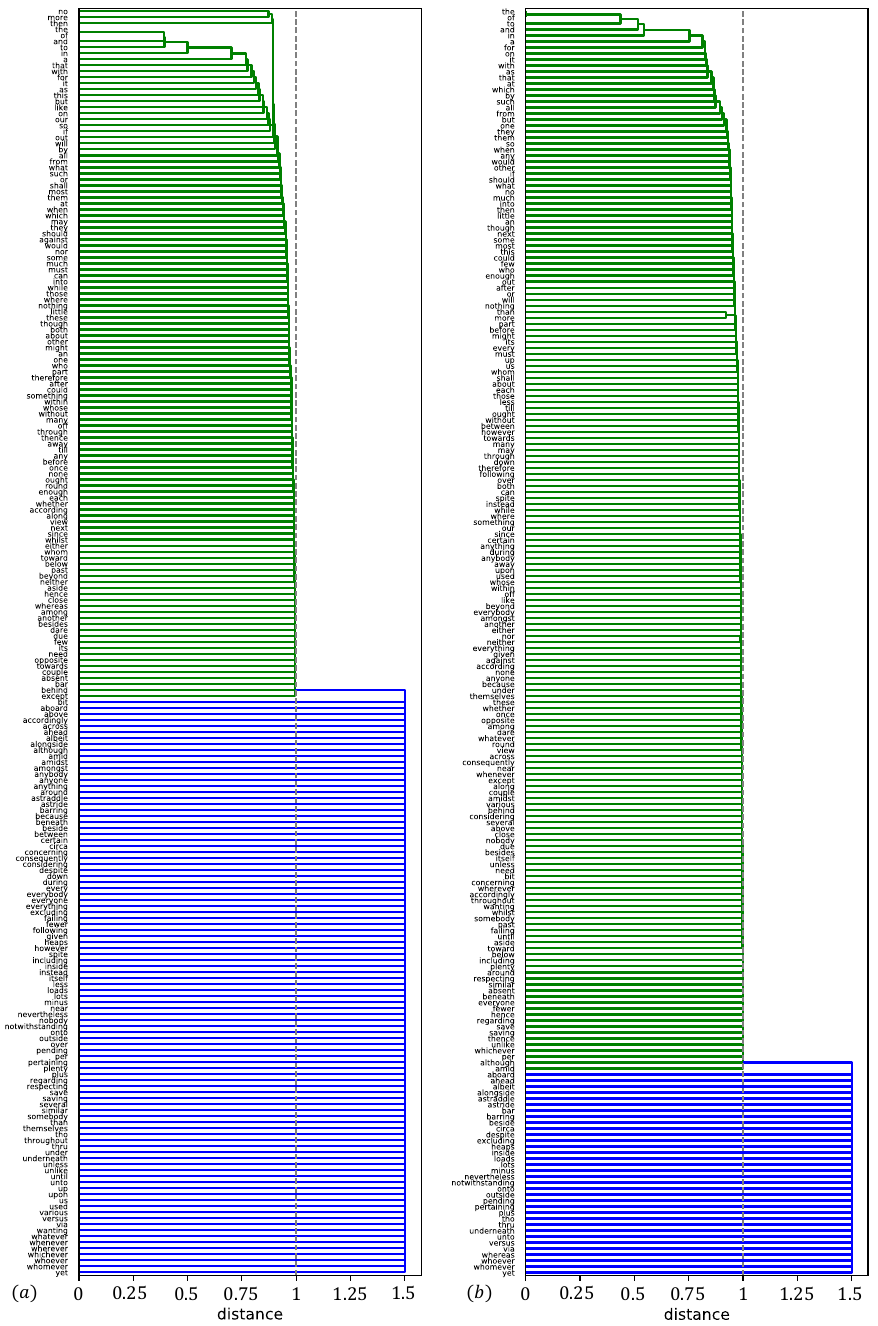}
				\end{subfigure}%
			\end{minipage}
			\rotatebox[origin=c]{90}{
				{\begin{minipage}[t]{.8\linewidth} 
						\caption{Single-linkage dendrogram of WANs corresponding to "Jane Austen - Pride and Prejudice" (a) and "Shakespeare  - Hamlet" (b).}
				\end{minipage}}
			}
		\end{tabular}
		\label{fig:austen_shakespeare_dendro}
	\end{figure*}

	\begin{figure*}[h]
		\centering
		\refstepcounter{figure}
		
		\begin{tabular}{c c}
			\begin{minipage}{0.76\linewidth}
				\includegraphics[width=1 \linewidth]{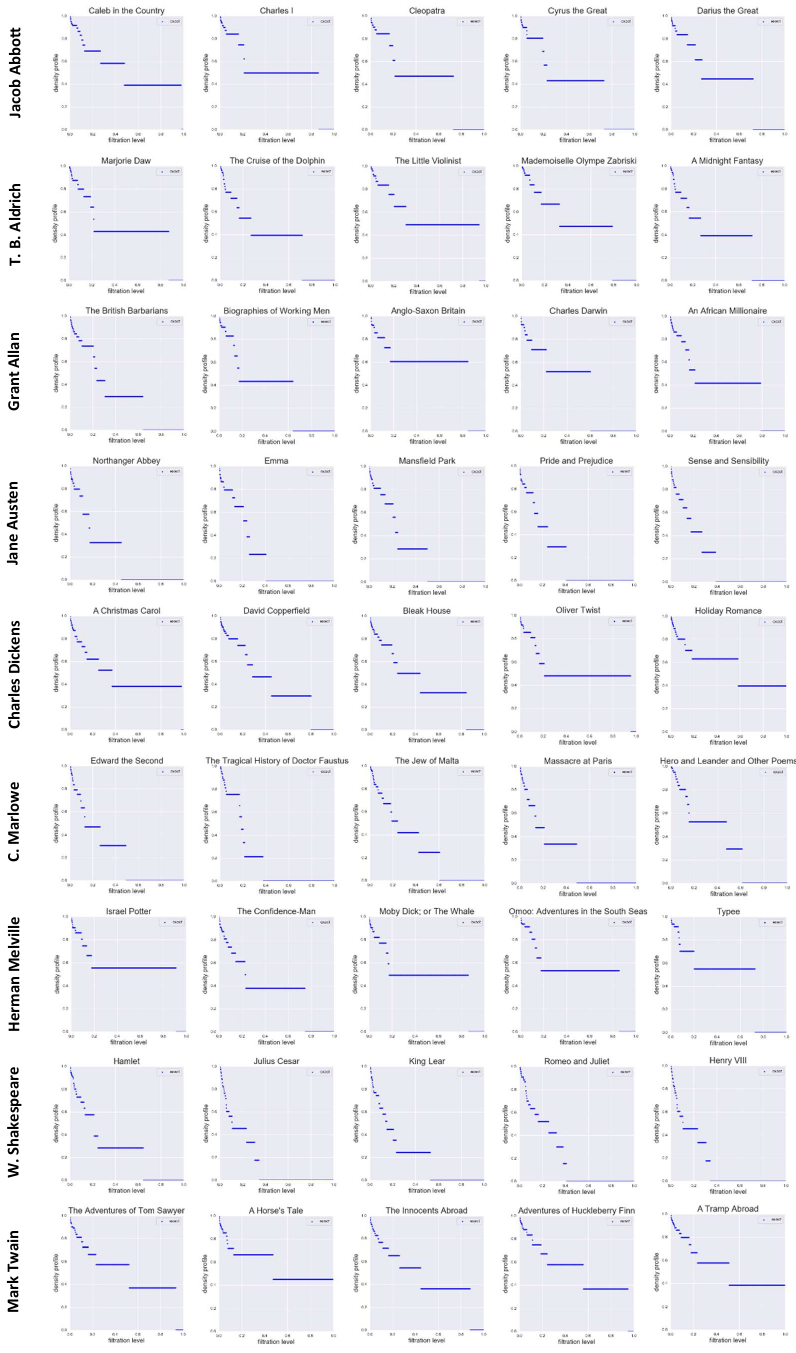}
			\end{minipage}
			\rotatebox[origin=c]{90}{
				{\begin{minipage}[t]{.8\linewidth} 
						\caption{Exact CHD profiles of the WAN dataset for the motif pair $H_{00}=F_{00}=([1],\mathbf{1}_{\{(1,1)\}})$.}
				\end{minipage}}
			}
		\end{tabular}

		\label{fig:WAN_profile_00}
	\end{figure*}

	\begin{figure*}
		\centering
		\refstepcounter{figure}
		
		\begin{tabular}{c c}
			\begin{minipage}{0.8\linewidth}
				\includegraphics[width=1 \linewidth]{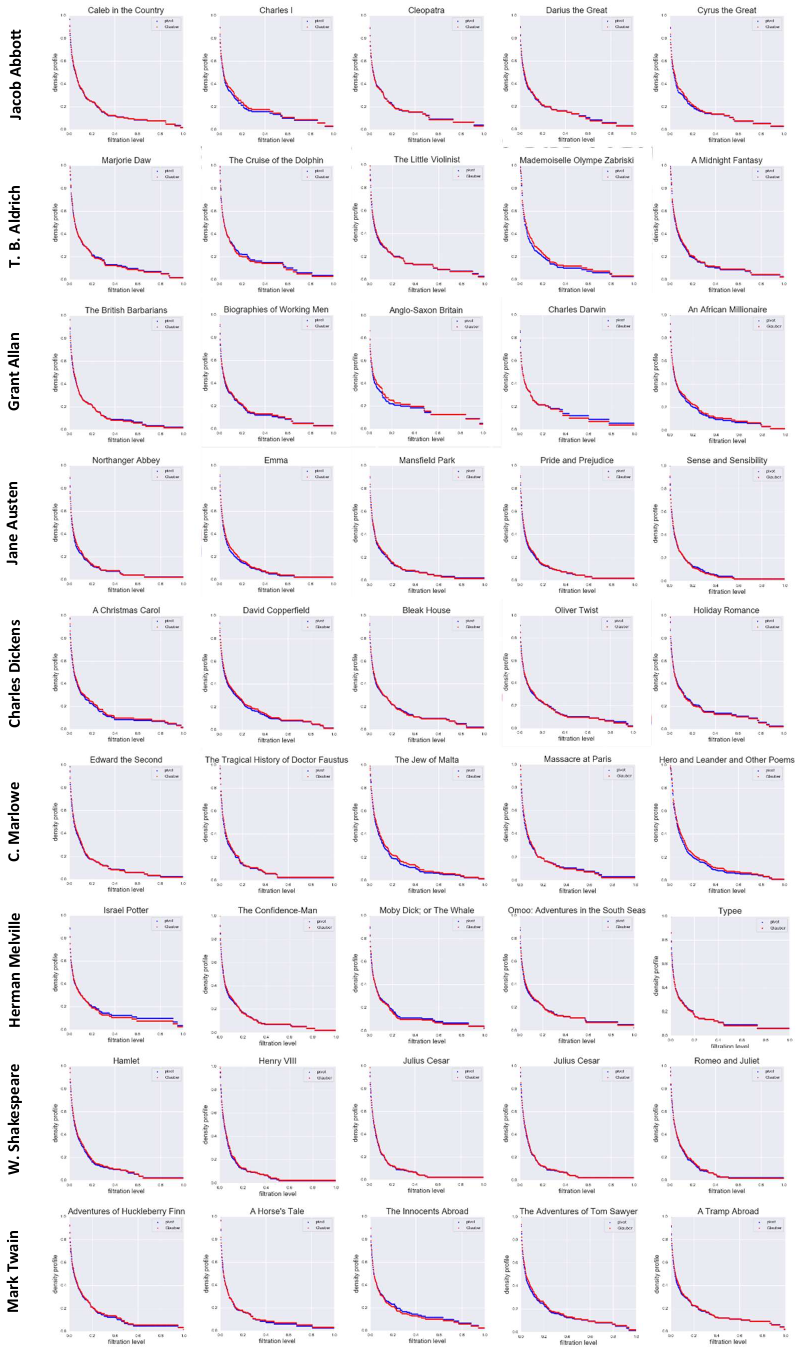}
			\end{minipage}
			\rotatebox[origin=c]{90}{
				{\begin{minipage}[t]{.8\linewidth} 
						\caption{Approximate CHD profiles of the WAN dataset for the motif pair $H_{01}=F_{01}=([1,2],\mathbf{1}_{\{(1,2)\}})$. Glauber chain (red) and Pivot chain (blue) for 5000 iterations.}
				\end{minipage}}
			}
		\end{tabular}

		\label{fig:WAN_profile_01}
	\end{figure*}

	\begin{figure*}
		\centering
		\refstepcounter{figure}
		\begin{tabular}{c c}
			\begin{minipage}{0.8\linewidth}
				\includegraphics[width=1 \linewidth]{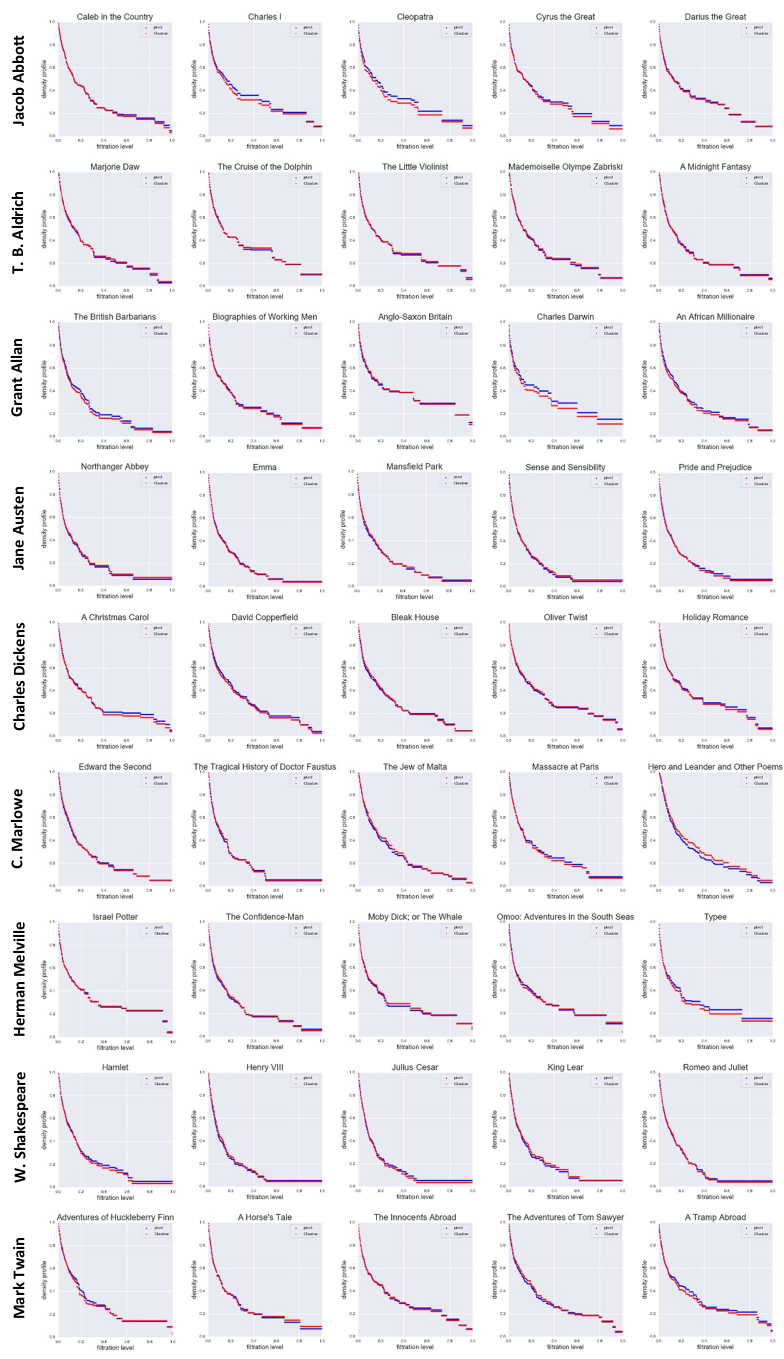}
			\end{minipage}
			\rotatebox[origin=c]{90}{
				{\begin{minipage}[t]{.8\linewidth} 
						\caption{Approximate CHD profiles of the WAN dataset for the motif pair $H_{11}=([1,2,3],\mathbf{1}_{\{2,3\}})$ and $F_{11}=([1,2,3],\mathbf{1}_{\{(1,2),(1,3)\}})$.    Glauber chain (red) and Pivot chain (blue) for 5000 iterations.}
				\end{minipage}}
			}
		\end{tabular}
		
		\label{fig:WAN_profile_11}
	\end{figure*}

\end{document}